%% file: thesis-manuscript-2022-07-19-v1.tex
\documentclass[12pt,reqno,a4letter]{article} 

\usepackage{amsthm,amsfonts,amscd,amsmath}
\usepackage[hidelinks]{hyperref} 
\usepackage{url}
\usepackage[usenames,dvipsnames]{xcolor}
\hypersetup{
    colorlinks,
    linkcolor={black},
    citecolor={black},
    urlcolor={black}
}

\usepackage{ocr}
\usepackage[T1]{fontenc}

\usepackage{graphicx} 
\usepackage{datetime} 
\usepackage{cancel}
\usepackage{subcaption}
\usepackage{stmaryrd} 
\usepackage{framed} 
\usepackage{tikzsymbols,halloweenmath}

\usepackage{enumitem}
\setlist[itemize]{noitemsep,topsep=1.5pt,leftmargin=0.35in,label={\tiny$\blacksquare$}}

\usepackage{fancyvrb}
\usepackage{rotating,adjustbox}

\usepackage{diagbox}

\usepackage{ragged2e}

\usepackage{tocloft}
\usepackage{listings}
\AtBeginDocument{\renewcommand*{\thelstlisting}{\textbf{\thesection.\arabic{lstlisting}}}}
\makeatletter
\AtBeginDocument{%
\renewcommand{\lstlistlistingname}{List of Source Code}
\renewcommand\lstlistoflistings{\bgroup
  \let\contentsname\lstlistlistingname
  \def\l@lstlisting##1##2{\@dottedtocline{1}{0em}{5.5em}{##1}{##2}}
  \let\lst@temp\@starttoc \def\@starttoc##1{\lst@temp{lol}}%
  \tableofcontents \egroup}
}
\makeatother

\usepackage{bold-extra}
\usepackage{titlesec}
\titleformat{\section}{\centering\normalfont\bfseries\uppercase}{\thesection}{1em}{}
\titleformat{\subsection}{\centering\bfseries}{\bfseries\thesubsection}{1em}{}

\newcommand{\SectionNumberFormat}[1]{\arabic{#1}}
\newcommand{\TheSectionPrefixName}[0]{Chapter}
\newcommand{\TheSectionPrefixNameUC}[0]{CHAPTER}
\renewcommand{\thesection}{\TheSectionPrefixName\ \SectionNumberFormat{section}}
\renewcommand{\thesubsection}{\SectionNumberFormat{section}.\arabic{subsection}}
\renewcommand{\thesubsubsection}{\thesubsection.\arabic{subsubsection}}

\newcounter{gtsection}
\newcommand{\SectionGTThesisFormatted}[1]{
     \addtocounter{section}{1}
     \addtocounter{gtsection}{1}
     \setcounter{subsection}{0}
     \setcounter{subsubsection}{0}
     \setcounter{figure}{0}
     \setcounter{table}{0}
     \setcounter{lstlisting}{0}
     \setcounter{equation}{0}
     \addcontentsline{toc}{section}{\thesection: #1}
     \begin{center}{\normalfont\bfseries\MakeUppercase{\thesection} \\ \normalfont\bfseries\uppercase{#1}}\end{center}
}

\newcommand{\SubsectionGTThesisFormatted}[1]{
     \addtocounter{subsection}{1}
     \setcounter{subsubsection}{0}
     \addcontentsline{toc}{subsection}{\thesubsection\hspace*{1em}#1}
     \begin{center}{\bfseries{\thesubsection}\hspace*{1em}\bfseries{#1}}\end{center}
}

\newcommand{\SubsubsectionGTThesisFormatted}[1]{
     \addtocounter{subsubsection}{1}
     {\bfseries{\thesubsubsection}\hspace*{1em}\bfseries{#1}}
}

\usepackage[onehalfspacing]{setspace}
\let\citep\cite

\newcommand{\undersetbrace}[2]{\underset{\displaystyle{#1}}{\underbrace{#2}}}

\newcommand{\gkpSI}[2]{\ensuremath{\genfrac{\lbrack}{\rbrack}{0pt}{}{#1}{#2}}} 
\newcommand{\gkpSII}[2]{\ensuremath{\genfrac{\lbrace}{\rbrace}{0pt}{}{#1}{#2}}}
\newcommand{\cf}{\textit{cf.~}} 
\newcommand{\Iverson}[1]{\ensuremath{\left[#1\right]_{\delta}}} 
\newcommand{\floor}[1]{\left\lfloor #1 \right\rfloor} 
\newcommand{\ceiling}[1]{\left\lceil #1 \right\rceil} 
\newcommand{\e}[1]{e\left(#1\right)} 
\newcommand{\seqnum}[1]{\href{http://oeis.org/#1}{\color{ProcessBlue}{\underline{#1}}}}

\usepackage{upgreek,dsfont,amssymb}
\renewcommand{\chi}{\upchi}

\newcommand{\OneFunc}[1]{\ensuremath{\mathds{1}_{#1}}}

\usepackage{ifthen}
\newcommand{\Hn}[2]{
     \ifthenelse{\equal{#2}{1}}{H_{#1}}{H_{#1}^{\left(#2\right)}}
}

\newcommand{\Floor}[2]{\ensuremath{\left\lfloor \frac{#1}{#2} \right\rfloor}}

\DeclareMathOperator{\ds}{ds} 
\DeclareMathOperator{\Id}{Id}
\DeclareMathOperator{\fg}{fg}
\DeclareMathOperator{\Div}{div}

\DeclareMathOperator{\poly}{poly}
\DeclareMathOperator{\Num}{Num}
\DeclareMathOperator{\Denom}{Denom}
\DeclareMathOperator{\ab}{a}
\DeclareMathOperator{\Conv}{Conv} 
\DeclareMathOperator{\ConvP}{P}
\DeclareMathOperator{\ConvQ}{Q}

\theoremstyle{plain} 
\newtheorem{theorem}{Theorem}
\newtheorem{conjecture}[theorem]{Conjecture}

\newtheorem{prop}[theorem]{Proposition}
\newtheorem{lemma}[theorem]{Lemma}
\newtheorem{cor}[theorem]{Corollary}
\numberwithin{theorem}{gtsection}

\theoremstyle{definition} 
\newtheorem{example}[theorem]{Example}
\newtheorem{remark}[theorem]{Remark}
\newtheorem{definition}[theorem]{Definition}
\newtheorem{notation}[theorem]{Notation}
\newtheorem{question}[theorem]{Question}
\newtheorem{discussion}[theorem]{Discussion}

\renewcommand{\arraystretch}{1.25} 

\usepackage[top=1in,bottom=1in,left=1.35in,right=1in]{geometry}
\input{PreambleGlossaries}

\newcommand{\ManuscriptTitle}{Factorization theorems and canonical representations for generating functions of special sums}
\newcommand{\ManuscriptAuthor}{Maxie Dion Schmidt}
\newcommand{\ManuscriptDateMonth}{August}
\newcommand{\ManuscriptDateYear}{2022}
\newcommand{\ManuscriptDateApproved}{July 6, 2022}

\usepackage{fontawesome,ccicons}
\newcommand{\ManuscriptCopyright}{\faCopyright}

\newcommand{\TitlePageSmallSkip}{\vspace{0.5cm} \\}

\newcommand{\TitlePageBigSkip}{\vspace{1.25cm} \\}
\newcommand{\TitlePageHugeSkip}{\vspace{2cm} \\}

\title{\ManuscriptTitle} 
\author{\ManuscriptAuthor} 

\usepackage{everyshi}
\newcounter{pagecnt}
\EveryShipout{\stepcounter{pagecnt}}
\newcommand{\pagestart}{%
    \renewcommand{\thepage}{%
        \ifnum\thepagecnt=0
            \addtocounter{pagecnt}{1}%
        \else%
            \arabic{pagecnt}%
        \fi%
    }%
}

\allowdisplaybreaks 

\begin{document} 

\pagestyle{plain}
\pagenumbering{roman} 
\setcounter{page}{1} 

\begin{titlepage}
   \begin{center}
       \vspace*{1cm}
       \large{\textbf{\MakeUppercase{\ManuscriptTitle}}}
       \TitlePageHugeSkip
       \normalsize A Dissertation \\ 
       \normalsize Presented to \\ 
       \normalsize The Academic Faculty
       \TitlePageBigSkip
       \normalsize By 
       \TitlePageBigSkip
       \ManuscriptAuthor
       \TitlePageBigSkip
       \normalsize In Partial Fulfillment \\ 
       \normalsize of the Requirements for the Degree \\ 
       \normalsize Doctor of Philosophy in the \\ 
       \normalsize School of Mathematics
       \TitlePageHugeSkip
       \normalsize Georgia Institute of Technology
       \TitlePageSmallSkip
       \ManuscriptDateMonth\ \ManuscriptDateYear 
       \vfill 
       \ManuscriptCopyright\ \ManuscriptAuthor\ \ManuscriptDateYear \\ 
   \end{center}
\end{titlepage}

\newpage
\begin{titlepage}
   \begin{center}
       \vspace*{1cm}
       \large{\textbf{\MakeUppercase{\ManuscriptTitle}}} \\
       \vspace*{2cm}
       \vspace*{1.75cm}
       \noindent
       \begin{FlushLeft}
       \normalsize Thesis committee:
       \end{FlushLeft}
       \vspace*{0.75cm}
       \begin{minipage}{0.5\textwidth}
       \normalsize Dr. Josephine Yu \\ 
       \normalsize School of Mathematics \\ 
       \normalsize \textit{Georgia Institute of Technology}
       \end{minipage}\hfil
       \begin{minipage}{0.5\textwidth}
       \normalsize Dr. Jayadev Athreya \\ 
       \normalsize Department of Mathematics \\ 
       \normalsize \textit{University of Washington}
       \end{minipage}
       \TitlePageSmallSkip
       \begin{minipage}{0.5\textwidth}
       \normalsize Dr. Matthew Baker \\ 
       \normalsize School of Mathematics \\ 
       \normalsize \textit{Georgia Institute of Technology}
       \end{minipage}\hfil
       \begin{minipage}{0.5\textwidth}
       \normalsize Dr. Bruce Berndt \\ 
       \normalsize Department of Mathematics \\ 
       \normalsize \textit{University of Illinois at Urbana-Champaign}
       \end{minipage}
       \TitlePageSmallSkip
       \begin{minipage}{\textwidth}
       \normalsize Dr. Rafael de la Llave \\ 
       \normalsize School of Mathematics \\ 
       \normalsize \textit{Georgia Institute of Technology}
       \end{minipage}
       \TitlePageSmallSkip
       \begin{FlushRight}
       \textbf{\textrm{\normalsize Date approved: \ManuscriptDateApproved}}
       \end{FlushRight}
       \vfill 
   \end{center}
\end{titlepage}

\newpage
\begin{titlepage}
\begin{center}
\fbox{\includegraphics[height=0.18\textheight]{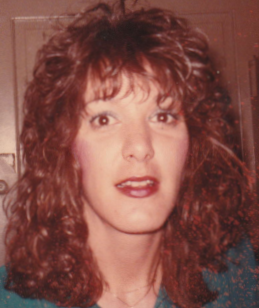}} 
\hspace*{0.075cm}
\fbox{\includegraphics[height=0.18\textheight]{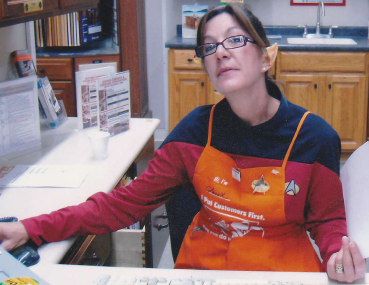}} 
\hspace*{0.075cm}
\fbox{\includegraphics[height=0.18\textheight]{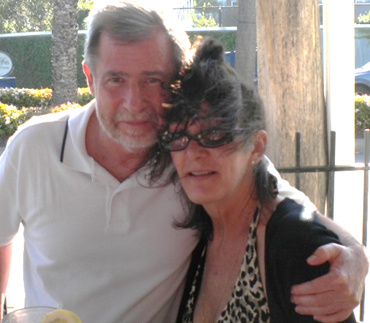}} \\
\bigskip
{\small$\mathwitch*\ \overrightwitchonbroom{}$\ }
{\textbf{In memoriam of \textit{Sarah Jane Sweat Schmidt} (1954--2022)}} 
{\ \small$\overleftwitchonbroom{}\ \reversemathwitch$}
\end{center}
\hrule\hrule\hrule\hrule\bigskip
\noindent
{\small\it
My mother, SJSS, passed away while I was in the last stages 
of completing my doctoral thesis this year. 
As a teenager, SJSS, gave me the original print version of her doctoral 
dissertation to keep with me for inspiration and as a good luck charm. 
It has been with me next to stacks of print math textbooks and countless drafts of manuscripts 
while completing my doctoral thesis at Georgia Tech. 
I would have never gotten here without her support. 
We are deeply saddened that she will not be here in person 
to celebrate this milestone. 
} 
\medskip
\begin{center}
\fbox{\includegraphics[height=0.54\textheight]{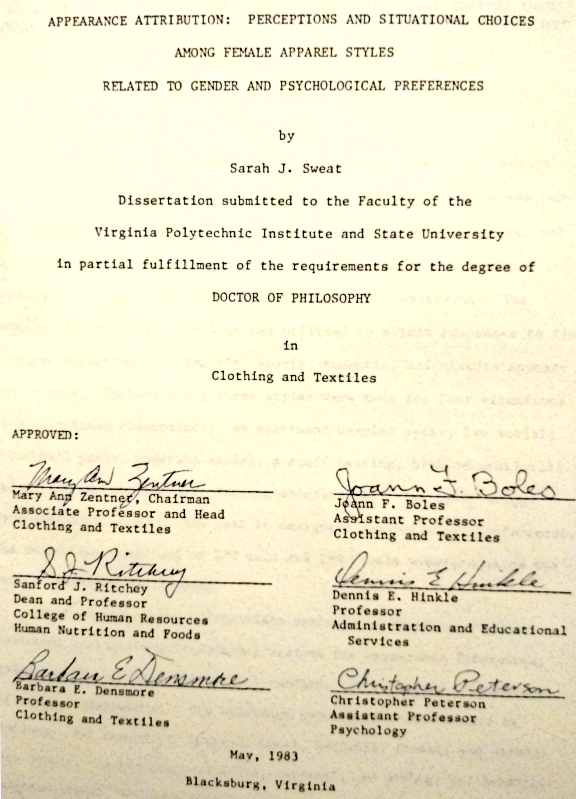}} 
\end{center}
\end{titlepage}

\newpage
\begin{titlepage}
\vspace*{\fill}
\renewcommand{\abstractname}{Acknowledgments}
\begin{abstract}
\normalsize\noindent
To my cat children: Kush, Cameo, George and Fred. 
To my parents, Dr.~Sarah Schmidt and Dr.~Don Schmidt, for all their 
support over the years that allowed me to complete this manuscript.
\end{abstract}
\vspace*{\fill}
\end{titlepage}


\newpage
\setcounter{tocdepth}{2}
\renewcommand{\contentsname}{\hfill\normalsize\bfseries\uppercase{Table of Contents}\hfill}   
\renewcommand{\cftaftertoctitle}{\vspace*{1em}}
\tableofcontents

\newpage
\addcontentsline{toc}{section}{List of Tables}
\renewcommand{\listtablename}{\hfill\normalsize\bfseries\uppercase{List of Tables}\hfill}
\listoftables
\renewcommand{\thetable}{\SectionNumberFormat{section}.\arabic{table}}

\newpage
\addcontentsline{toc}{section}{List of Figures}
\renewcommand{\listfigurename}{\hfill\normalsize\bfseries\uppercase{List of Figures}\hfill}
\listoffigures
\renewcommand{\thefigure}{\SectionNumberFormat{section}.\arabic{figure}}

\newpage
\addcontentsline{toc}{section}{List of Source Code}
\renewcommand{\lstlistlistingname}{\hfill\normalsize\bfseries\uppercase{List of Source Code}\hfill}
\lstlistoflistings
\renewcommand{\thelstlisting}{\SectionNumberFormat{section}.\arabic{lstlisting}}

\newpage
\label{Section_Glossary_NotationConvs} 
\addcontentsline{toc}{section}{Glossary of Notation and Conventions}
\printglossary[type={symbols},title={\normalsize\bfseries\uppercase{Glossary of notation and conventions}},nogroupskip=true]

\newpage
\addcontentsline{toc}{section}{Summary}
\renewcommand{\abstractname}{Summary}
\begin{center}
\MakeUppercase{\ManuscriptTitle} \\[0.5cm]
Maxie Dion Schmidt \\[0.5cm] 
\pageref{TheLastPage} Pages \\[0.5cm] 
Directed by Dr. Josephine Yu
\end{center}
\vspace*{\fill}
\begin{abstract}
\linespread{2}\normalsize
This manuscript explores many convolution (restricted summation) type sequences 
via certain types of matrix based factorizations that can be used to express their 
generating functions. 
The last primary (non-appendix) section of the thesis 
explores the topic of how to best rigorously define a so-termed ``\emph{canonically best}'' 
matrix based factorization for a given class of convolution sum sequences. 
The notion of a canonical factorization for the generating function of such sequences needs to 
match the qualitative properties we find in the factorization theorems for 
Lambert series generating functions (LGFs). The expected qualitatively most 
expressive expansion we find in the LGF case results naturally from algebraic constructions 
of the underlying LGF series type. We propose a precise quantitative requirement to generalize 
this notion in terms of optimal cross-correlation statistics for certain sequences that 
define the matrix based factorizations of the generating function expansions we study. 
We finally pose a few conjectures on the types of matrix factorizations we expect to 
find when we are able to attain the maximal (respectively minimal) 
correlation statistic for a given sum type. 
\end{abstract}
\vspace*{\fill}

\newpage
\renewcommand{\thepage}{\arabic{page}}
\label{page_LastPageOfDocumentFrontMatterCount}
\pagestyle{plain}
\pagenumbering{arabic} 
\setcounter{page}{1}   

\SectionGTThesisFormatted{Introduction} 
\label{Section_Intro}

\begin{quote}
\hrule\hrule\hrule\bigskip
{\it The full beauty of the subject of generating functions emerges only from tuning in on both channels: 
	the discrete and the continuous. See how they make the solution of difference equations into child’s 
	play. Then see how the theory of functions of a complex variable gives, virtually by inspection, the 
	approximate size of the solution. The interplay between the two channels is vitally important for the 
	appreciation of the music.} \\
	\rule{0.9\textwidth}{1.5pt}\medskip \\ 
	{\large H.~S.~Wilf} \\ 
	{\large Philadelphia, PA} \\ 
	{\large September 1, 1989} \\
\hrule\hrule\hrule
\end{quote}

\SubsectionGTThesisFormatted{Generating functions are essential tools in discrete mathematics}

When I first embarked on my college adventures as an 
undergraduate at the University of Illinois in 2004, the 
first book I checked out from the extensive 
Altgeld mathematics library stacks was \emph{Generatingfunctionology} by 
H.~S.~Wilf. The quote featured at the start of this section is transcribed from the 
preface to the book. The usage and importance of 
generating functions in the study of integer sequences is memorably 
summarized by Wilf by the analogy on the first line to his classic introductory 
survey of this subject: 
``\emph{A generating function is a clothesline on which we hang up a sequence of numbers for display}''. 

Generating functions are series expansions summing over a sequence, or arithmetic function, 
with indeterminate weights in powers of an auxiliary variable 
\cite[\S 1]{GFOLOGY} (\cf \cite{ECV1,ECV2,LANDO-GFLECT,ACOMB-BOOK,RIORDAN-COMBIDENTS,ADVCOMB}). 
A sequence generating function may be treated formally in the ring of formal power series 
(e.g., as rigorously motivated in \cite{REZNICK-STERN-NOTES}) or may take on conditional 
meaning as analytic functions of 
a complex variable depending on the context and use cases in applications. 
There are many types and flavors 
of generating functions that we find in practice. 
For any fixed sequence, $\mathcal{F} := \{f_n\}_{n \geq 0} \subset \mathbb{C}$, 
we consider the following types of generating functions\footnote{
    The notational conventions used to name these generating functions is adapted 
    from the good style in Graham, Knuth and Patashnik \cite[\cf \S 7]{GKP} 
    (\cf \cite{KNUTHNOTATION}). 
}:
\begin{itemize}
\item The \textbf{ordinary generating function (OGF)} of $\mathcal{F}$ (in the variable $z$) 
	is defined by 
	\[
	F(z) := \sum_{n \geq 0} f_n z^n. 
	\]
\item The \textbf{exponential generating function (EGF)} $\mathcal{F}$ (in the variable $z$) 
	is defined by 
	\[
	\widehat{F}(z) := \sum_{n \geq 0} \frac{f_n z^n}{n!}. 
	\]
\item The \textbf{Dirichlet generating function (DGF)} 
	of $\mathcal{F}$ (in the variable $s$) 
	is defined by 
	\[
	D[\mathcal{F}](s) := \sum_{n \geq 1} \frac{f_n}{n^s}. 
	\]
\end{itemize}
Most of my undergraduate and graduate research in mathematics 
focuses on studying integer sequences through 
techniques that facilitate exploration of their properties via 
transformations of generating functions. 
One article of mine\footnote{
	Henceforth, the author of this manuscript, abbreviated MDS. 
}
that was assembled my first semester at 
Georgia Tech in 2017 provides a short expository introduction to this type of work. 
The survey article was published in the special issue of the journal \emph{Axioms} titled 
\emph{Mathematical Analysis and Applications II}
\cite{MDS-GFSURVEY}. 
Other peer-reviewed publications on this topic since 
enrolling at Georgia Tech in 2017 include 
\cite{MDS-OJAC-V1,MDS-OJAC-V2,MDS-SQUARE-SERIES-GFTRANS,SCHMIDT-SODFORMULAS,MDS-COMBRESTRDIVSUMS-INTEGERS}. 

A significant subtopic I have explored in this area focuses on enumerating certain forms of 
generalized factorial functions and symbolic product sequences through 
\emph{Jacobi type continued fractions} (J-fractions) 
\cite{MDS-JNT-2017,MDS-RAMJ-CFRACS,MDS-JIS-V2-2017,MDS-INTEGERS-CFRACS-V1,MDS-INTEGERS-CFRACS-V2}. 
A J-fraction is a continued fraction that is symbolic in an auxiliary series, or generating function, 
variable $z$ whose infinite expansion yields the OGF of a sequence, and whose finite 
convergents approximate the OGF in accuracy as a 
truncated power series expansion of the OGF of the same sequence. 
For sequences 
$\{ c_i \}_{i=1}^{\infty}$ and $\{ \ab_i \}_{i=2}^{\infty}$, and 
some typically formal series variable $z \in \mathbb{C}$, 
we associate the J-fraction expansion with these 
sequences defined as follows: 
\begin{align} 
\label{eqn_def_JacobiType_J-Fraction_v1} 
J_{\infty}\left(z\right) 
     & = 
     \cfrac{1}{1-c_1z-\cfrac{\ab_2 z^2}{1-c_2z- 
     \cfrac{\ab_3 z^2}{\cdots},}} \\ 
\notag 
     & = 
     1 + c_1 z + \left(\ab_2+c_1^2\right) z^2 + 
     \left(2\ab_2 c_1+c_1^3+\ab_2 c_2\right) z^3 \\ 
\notag
     & \phantom{=1\ } + 
     \left(\ab_2^2+\ab_2\ab_3+3\ab_2 c_1^2 + c_1^4 + 2\ab_2 c_1c_2 + \ab_2 c_2^2 
     \right) z^4 + 
     \cdots. 
\end{align} 
Provided that the variable $z$ in the last equation can be restricted to a non-trivial 
annulus of convergence, the expansions in 
\eqref{eqn_def_JacobiType_J-Fraction_v1} can be defined as the limit as $h \rightarrow \infty$ 
of the $h^{th}$ convergents to $J_{\infty}(z)$ defined by 
\[
\Conv_h(z) := \frac{\ConvP_h(z)}{\ConvQ_h(z)}. 
\]
The sequences $\{\ConvP_h(z)\}_{h \geq 0}$ and $\{\ConvQ_h(z)\}_{h \geq 0}$ 
are formed by the finite-degree polynomials in $z$ that satisfy the following 
recurrence relations:
\begin{align} 
\label{eqn_ConvFn_PhzQhz_rdefs} 
\ConvP_h(z) & = (1-c_{h}  z) \ConvP_{h-1}(q, z) - 
     \ab_{h} z^2 \ConvP_{h-2}(q, z) + \Iverson{h = 1}, \\ 
\notag 
\ConvQ_h(z) & = (1-c_{h}  z) \ConvQ_{h-1}(q, z) - 
     \ab_{h} z^2 \ConvQ_{h-2}(q, z) + 
     (1-c_{1}  z) \Iverson{h = 1} + \Iverson{h = 0}. 
\end{align} 
A Jacobi type J-fraction expansion comes equipped with an extra structure 
that consitutes the usual rigmarole we 
can assert by working with continued fractions and their finite convergents 
\citep[\cf \S 3.10]{NISTHB} \citep{WALL-CFRACS}. 
Manuscripts originally due to the late, great innovator in combinatorial analysis, P.~Flajolet, 
from the 1980's also prove the next property. The following is a 
result that provides integer congruences for the coefficients in the expansion 
of the rational functions of $z$ that are formed by the $h^{th}$ convergents 
to \eqref{eqn_def_JacobiType_J-Fraction_v1} \cite{FLAJOLET80B,FLAJOLET82}: 
\[
[z^n] J_{\infty}(z) \equiv [z^n] \Conv_h(z) \quad \pmod{\ab_2\ab_3 \cdots \ab_{h+1}}, 
	\text{ for any } n \geq 0 \text{ and } h \geq 2. 
\]
In this manuscript, we will explore another vantage point from which 
we can use generating functions to find new meaning and interpret generating function 
based constructions that enumerate integer sequences. 

\SubsectionGTThesisFormatted{Lambert series generating functions}

In the same way that many combinatorial sequences are described succinctly through the 
expressions of their OGF or EGF, many arithmetic functions of interest in multiplicative number theory 
satisfy a structure that is simple to describe by Dirichlet convolution. 
That is, for any fixed arithmetic functions $f$ and $g$, we define their 
\emph{Dirichlet convolution at $n$} for any integer $n \geq 1$ by 
$$(f \ast g)(n) = \sum_{d|n} f(d) g\left(\frac{n}{d}\right).$$
We can then appeal to expansions of a special kind of generating function to enumerate the 
special functions which we wish to study. 
A Lambert type series, or \emph{Lambert series generating function} (LGF), for an arithmetic function
is a special generating function that allows us to capture and enumerate many multiplicative functions 
$f$ whose descriptions yield a meaningful interpretation of the divisor sums 
$(f \ast \mathds{1})(n)$ at integers $n \geq 1$. 
\nocite{CATALOG-INTDIRSERIES}

\begin{definition}
For any arithmetic function $f: \mathbb{Z}^{+} \rightarrow \mathbb{C}$, we have a corresponding 
LGF for $f$ defined by the following series expansions:
\begin{align} 
\label{eqn_Lfq_LSeriesGF_def_v1} 
L_f(q) & := \sum_{n \geq 1} \frac{f(n) q^n}{1-q^n} = \sum_{m \geq 1} (f \ast \mathds{1})(m) q^m, |q| < 1.
\end{align}
\end{definition}

From 2016--2017, my collaborator M.~Merca and MDS began work on overlapping 
interests from \cite{AA,MERCA-LSFACTTHM} to build on what we have termed as the next 
so-called \emph{Lambert series factorization theorems} 
\cite{MDS-MERCA-AMM,MERCA-SCHMIDT-LSFACTTHM,MERCA-SCHMIDT-PN,MERCA-SCHMIDT-RAMJ}. 
The divisor sum coefficient expansion on the right-hand-side of 
\eqref{eqn_Lfq_LSeriesGF_def_v1} reaffirms that there is a natural way to associate an 
OGF to generate many 
multiplicatively structured functions using $L_f(q)$ for a fixed arithmetic function $f$. 
On the  other hand, by taking common denominators of the series terms, these expansions 
involve the infinite $q$-Pochhammer symbol. 
This observation yields a clear connection of the 
multiplicative functions typically enumerated by these LGF series expansions to the more 
additive theory of integer partitions. 

\begin{subequations}
\label{eqn_LGFFactThmExps_TopLevelSubEq_v0}
The basic construction is initially stated in the following form:  
\[
L_f^{\pm}(q) := \sum_{n \geq 1} \frac{f(n) q^n}{1 \pm q^n} = \frac{1}{(\mp q; q)_{\infty}} \times 
     \sum_{n \geq 1} \left( 
     \sum_{k=1}^{n} \left(s_o(n, k) \pm s_e(n, k)\right) f(k)\right) q^n, |q| < 1. 
\]
In the last equation, the sequences $s_e(n, k)$ (and $s_o(n, k)$)
denote the the number of $k$’s in all partitions of $n$
into an even (and odd, respectively) number of distinct parts for $n \geq k \geq 1$ and where 
$(a; q)_{\infty} := \prod_{m \geq 1} (1-aq^{m-1})$ is the infinite \emph{$q$-Pochhammer symbol}. 
We can re-write the last equation as 
\begin{equation} 
\label{eqn_LambertSeriesFactThm_InitExp} 
\sum_{n \geq 1} \frac{f(n) q^n}{1-q^n} = \frac{1}{(q; q)_{\infty}} \times \sum_{n \geq 1} \left( 
     \sum_{k=1}^{n} s_{n,k} f(k)\right) q^n. 
\end{equation} 
The matrices formed by the lower triangular sequence of $s_{n,k}$ are invertible as are the 
inverse matrices $s_{n,k}^{-1}$. The lower triangular inverse sequence is 
also related to partition theoretic functions.
In particular, we can prove exactly how $s_{n,k}^{-1}$ is related to $p(n)$ as follows: 
\begin{equation}
s_{n,k}^{-1} = \sum_{d|n} p(d-k) \mu\left(\frac{n}{d}\right).
\end{equation}
\end{subequations}

\SubsectionGTThesisFormatted{Generalized forms of the factorization theorems}

\SubsubsectionGTThesisFormatted{Motivation}

\begin{definition}
The \emph{average order} of an arithmetic function $f$ is defined as the arithmetic mean of the summatory function, 
$F(x) := \sum_{n \leq x} f(n)$, as $f_{\operatorname{ave}}(x) := F(x) x^{-1}$. We typically represent the average order of 
a function, $f_{\operatorname{ave}}(x)$, 
in the form of an asymptotic formula where the average order diverges as $x \rightarrow \infty$. 
For example, it is well known that the average order of $\Omega(n)$, which counts the 
number of prime factors of $n$ (counting multiplicity), is 
asymptotically dominated by $\log\log x$. The average order of Euler's totient 
function, $\phi(n)$, grows like $\frac{6x}{\pi^2}$ as $n \rightarrow \infty$ \cite{HARDYWRIGHT}. 
\end{definition}

We are motivated by a decomposition of the partial sums of an 
arithmetic function $f(d)$ whose average order we wish estimate into sums over 
pairwise disjoint sets of component indices $d \leq x$ that correspond to the indices 
of summation in each of the three terms on the right-hand-side of the next equation. 
\begin{align} 
\label{eqn_avg_order_sum_decomp_intro_v1} 
\sum_{d \leq x} f(d) & = \sum_{\substack{1 \leq d \leq x \\ (d,x)=1}} f(d) + \sum_{\substack{d|x \\ d>1}} f(d) + 
     \sum_{\substack{1 < d \leq x \\ 1 < (d,x) < x}} f(d) 
\end{align} 
In evaluating the partial sums of an arithmetic function $f(d)$ over all $d \leq x$, we wish 
to split the terms in these partial sums into three sets: those $d$ relatively prime to $x$, the 
$d$ dividing $x$ (for $d > 1$), and the 
set of indices $d$ which are neither relatively prime to $x$ nor proper divisors of $x$. 
If we let $f$ denote any arithmetic function, we define the remainder terms in our average order 
expansions from \eqref{eqn_avg_order_sum_decomp_intro_v1} as follows: 
\begin{equation} 
\label{eqn_Tfx_remainder_sum_terms_def} 
\widetilde{S}_{f}(x) = \sum_{d \leq x} f(d) - 
     \sum_{\substack{1 \leq d \leq x \\ (d,x)=1}} f(d) - \sum_{\substack{d|x \\ d > 1}} f(d). 
\end{equation} 
We observe that the divisor sum terms in \eqref{eqn_Tfx_remainder_sum_terms_def} correspond to the 
coefficients of powers of $q$ in the Lambert series generating function over $f$ in the form of 
\[
\sum_{d|n} f(d) = [q^n] \left(\sum_{m \geq 1} \frac{f(m) q^m}{1-q^m}\right), n \geq 1. 
\] 
We can also see that the (unscaled) average order sums on the left-hand-side of 
\eqref{eqn_avg_order_sum_decomp_intro_v1} correspond to a hybrid of nested 
divisor and relatively prime divisor type sums as 
\[
\sum_{d \leq x} f(d) = \sum_{m|x} 
     \sum_{\substack{k=1 \\ \left(k, \frac{x}{m}\right) = 1}}^{\frac{x}{m}} f(km) = 
     \sum_{m|x} \sum_{\substack{k=1 \\ (k, m) = 1}}^{m} f\left(\frac{kx}{m}\right). 
\] 
Hence, we seek to reconcile the divisibility structure of 
interesting sequences (and summatory functions) of interest using an 
enumerative approach by generating functions through a study of 
generalizations of the original factorization theorems for Lambert series in 
\eqref{eqn_LambertSeriesFactThm_InitExp} 
(see Section \ref{Section_FactThmsTypeIAndIISums_MousaviSchmidt}). 

\SubsubsectionGTThesisFormatted{Precise formulations of generating function expansions for generalized convolution type sequences}

\begin{subequations}
\label{eqn_GenMatrixBasedSumsEnum_TopLevelSubEqHandle_v0} 
We will see many variants of invertible matrix-based factorizations 
of generating functions for special sums and series in this thesis. 
We first suggest the next form of the 
generating function factorization theorems that generalizes 
\eqref{eqn_LambertSeriesFactThm_InitExp} to many other applications. 
We can consider analogous matrix-based factorizations of the 
generating functions of the sequences of special sums in 
\eqref{eqn_GenMatrixBasedSumsEnum_SetsAn} below provided that these 
transformations are suitably invertible. That is, we can express generating functions for the 
sums $s_n(f, \mathcal{A}) := \sum_{k \in \mathcal{A}_n} f(k)$ where we take 
$\mathcal{A}_n \subseteq [1,n] \cap \mathbb{Z}$ for all $n \geq 1$ 
as expansions of the following form: 
\begin{equation} 
\label{eqn_GenMatrixBasedSumsEnum_SetsAn} 
\sum_{n \geq 1} \left(\sum_{\substack{k \in \mathcal{A}_n \\ 
     \mathcal{A}_n \subseteq [1,n]}} f(k)\right) q^n = 
     \frac{1}{\mathcal{C}(q)} \times 
     \sum_{n \geq 1} \left(\sum_{k=1}^n v_{n,k}(\mathcal{A}, \mathcal{C}) f(k)\right) q^n, 
     \mathcal{C}(0) \neq 0. 
\end{equation} 
The matrix entries are generated by 
\[
v_{n,k}(\mathcal{A}, \mathcal{C}) = [q^n] \mathcal{C}(q) \times \sum_{m \geq 1} 
     \Iverson{k \in \mathcal{A}_m} q^m. 
\]
These lower triangular coefficients lead to invertible transformations provided that 
$v_{n,n}(\mathcal{A}, \mathcal{C}) \neq 0$ for all $n \geq 1$. 
We can also naturally consider sums of an arithmetic function weighted by a lower triangular sequence 
of the form 
\begin{equation} 
\label{eqn_GenMatrixBasedSumsEnum_SetsAn_v2} 
     \sum_{n \geq 1} \left(\sum_{k=1}^{n} \mathcal{T}_{n,k} f(k) 
     \right) q^n = 
     \frac{1}{\mathcal{C}(q)} \times 
     \sum_{n \geq 1} \left(\sum_{k=1}^n u_{n,k}(\mathcal{T}, \mathcal{C}) f(k)\right) q^n, 
     \mathcal{C}(0) \neq 0. 
\end{equation} 
The notable special cases of the so-termed type I and type II sums 
defined in Section \ref{Section_FactThmsTypeIAndIISums_MousaviSchmidt} are given by 
\eqref{eqn_GenMatrixBasedSumsEnum_SetsAn} when 
\begin{align*}
\mathcal{A}_{1,n} &:= \{1 \leq d \leq n: (d, n) = 1\}, \\ 
\mathcal{A}_{2,n} & := \{1 \leq d \leq n: \exists 1 \leq k \leq n \text{ such that } d|(k, n)\},
\end{align*}
respectively. 
\end{subequations}

\SubsubsectionGTThesisFormatted{Identification the most qualitatively significant 
               expansions of the generalized factorization theorems}

The expansion of the OGF series of the forms displayed on the right-hand-sides of 
equation \eqref{eqn_GenMatrixBasedSumsEnum_SetsAn} and 
equation \eqref{eqn_GenMatrixBasedSumsEnum_SetsAn_v2}, respectively, 
is an algebraically very natural way to view the generating functions of the 
coefficient sums $(f \ast \mathds{1})(n)$ that are enumerated by the LGF series cases. 
The pre-mutliplier function given by the infinite $q$-Pochhammer symbol 
when $\mathcal{C}(q) := (q; q)_{\infty}^{-1}$ naturally ``falls out'', 
or is easily identified, as a key feature that distinguishes the limiting expansions 
after we perform the symbolic arithmetic to combine all terms 
in the next partial sums into a single rational expression of $q$ with common denominator, 
$(q; q)_n = (1-q)(1-q^2) \times \cdots \times (1-q^n)$: 
\[
(f \ast \mathds{1})(n) \equiv [q^n] \left(\sum_{1 \leq j \leq N} \frac{f(j) q^j}{1-q^j}\right), 
     \text{ whenever } 1 \leq n \leq N. 
\]
We will restrict our interest to those generalized convolution sum types which are invertible, 
i.e., those sequences such that $n \in \mathcal{A}_n$ for all $n \geq 1$, or where 
$\mathcal{T}_{n,n} \neq 0$ at every $n \geq 1$. 
In these cases, we obtain an invertible sequence of lower triangular matrix entries 
$v_{n,k}(\mathcal{A}, \mathcal{C})$, or $u_{n,k}(\mathcal{T}, \mathcal{C})$, upon 
multiplication of the left-hand-sides of each equation in 
\eqref{eqn_GenMatrixBasedSumsEnum_TopLevelSubEqHandle_v0} by 
considering the coefficients of each 
$f(m)$ for integers $m \geq 1$ as functions (or formal power series) in $q$. 
We assert a reasonable expectation 
that for any given convolution sum type defined by a fixed 
$\{\mathcal{A}_n\}_{n \geq 1}$, or $\{\mathcal{T}_{n,k}\}_{1 \leq k \leq n}$, 
we can prioritize certain choices of the OGF $\mathcal{C}(q)$ so that 
$\left\{[q^n] \mathcal{C}(q)^{\pm 1}\right\}_{n \geq 0}$ and the 
corresponding matrix entries $\{v_{n,k}(\mathcal{A}, \mathcal{C})\}_{n,k \geq 1}$, 
or $\{u_{n,k}(\mathcal{T}, \mathcal{C})\}_{n,k \geq 1}$, reflect a strong 
``coupling`` or ``correlation`` in qualitative relevance between these sequences. 

\begin{question}
\label{question_PrecursorToCorrStats_v1}
What are the naturally ``good'' choices, 
or is a ``\emph{canonically best}`` choice in some sense that we can unamibuously define and make precise, 
of the generating function $\mathcal{C}(q)$ for any fixed definition of the sets 
$\{\mathcal{A}_n\}_{n \geq 1}$ (of the triangles $\{\mathcal{T}_{n,k}\}_{1 \leq k \leq n}$)? 
A solution to this problem that is sufficiently well defined 
should answer this question in such a way 
that the most qualitatively meaningful expressions for the matrix entries 
$v_{n,k}(\mathcal{A}, \mathcal{C})$ (and its inverse) result in the expansion of the 
factorization theorems in equation \eqref{eqn_GenMatrixBasedSumsEnum_SetsAn} or 
equation \eqref{eqn_GenMatrixBasedSumsEnum_SetsAn_v2}?
\end{question}

\begin{remark}[Motivation from the LGF case]
The choice of scaling on the right-hand-side of 
\eqref{eqn_LambertSeriesFactThm_InitExp} 
in the LGF series expansion of $L_f^{-}(q)$ by the OGF $(q; q)_{\infty}^{-1}$ 
is \'{a} priori revealing in many unexpected ways. 
Namely, this reciprocal factor providing a multiple of the 
ordinary generating function for $p(n)$ yields a clear-cut, and algebraically natural, relation of 
both sequences of $s_{n,k}$ and $s_{n,k}^{-1}$ 
from \eqref{eqn_LGFFactThmExps_TopLevelSubEq_v0} 
to functions that arise naturally in the theory of partitions. 
Then at a minimum, to adequately answer Question \ref{question_PrecursorToCorrStats_v1}
in its most general form, we must replicate the phenomenon captured in the LGF series expansions 
formed when $\mathcal{A}_n := \{1 \leq d \leq n: d|n\}$ for $n \geq 1$ 
(or equivalently when $\mathcal{T}_{n,k} := \Iverson{k|d}$ for $1 \leq k \leq n$). 
That is, the 
\emph{canonically best} function $\mathcal{C}(q)$ predicted by the qualitatively precise model that is 
provided by our solution criteria indeed matches the expected OGF 
$\mathcal{C}(q) := (q; q)_{\infty}$. 
\end{remark}

We explore how to best define a corresponding notion of so-termed ``canonically best'' factorization theorems  
for the generating functions of other special sum types in 
Section \ref{Section_CanonicalReprsOfFactThms_KernelBasedDCVL}. 
This last section of the manuscript contains mostly expository material that 
partially addresses Question \ref{question_PrecursorToCorrStats_v1}. 
At the end of the section, we conclude by  hypothesizing a few conjectures 
that seek to identify those OGF candidates $\mathcal{C}(q)$ with 
$\mathcal{C}(0) = 1$ that enumerate integer-valued coefficients and for which we 
expect to witness the most qualitatively expressive, 
or theoretically meaningful, new identities for a given 
convolution sum type that result upon expansion of the factorization theorems. 
In the most general setting for these factorization theorems, 
we will pose quantitatively precise definitions of 
correlations statistics such that when we attain a theoretically maximal (minimal) 
value of these series for any particular sum type, we expect to witness the same 
qualitatively most significant expansions that we found in the LGF case, e.g., where we 
identified the unusual relations between the divisor sum constructions that are natural 
in the study of functions from multiplicative number theory and the more 
additive branch of mathematical partition theorey. 
A concrete approach to rigorously 
quantifying our qualitative observations in our first LGF special case is motivated by 
the discussion in Section \ref{subSection_Intro_NotionsOfCanonicallyBest_v1} below. 

\SubsectionGTThesisFormatted{Examples of canonical and illustrative examples of special sums}

The examples cited in the new few subsections 
provide a survey of interesting applications of the generalized convolution type sequences 
we consider later in the thesis.
We will aperiodically return to these as reference points for comparison with the 
new results and generalized factorization theorems. 

\SubsubsectionGTThesisFormatted{A-convolutions: Restricted classes of Dirichlet convolutions and divisor sums (ACVL type sums)} 

For each $n \geq 1$, let $A(n) \subseteq \{d: 1 \leq d \leq n, d|n\}$ be a subset of the 
divisors of $n$. We say that a natural number $n \geq 1$ is \emph{$A$-primitive} if $A(n) = \{1, n\}$. 
Under a list of assumptions so that the resulting $A$-convolutions are \emph{regular convolutions}, 
we get a generalized multiplicative M\"obius function \cite{HANDBOOKNT-2004}: 
\[
\mu_{\mathcal{A}}(p^{\alpha}) = \begin{cases} 
     1, & \alpha = 0; \\ 
     -1, & p^{\alpha} > 1\text{\ is $\mathcal{A}$-primitive; } \\ 
     0, & \text{otherwise.}
     \end{cases} 
\]
Let the set 
\[
\mathcal{A} := \left\{n \geq 1: n \text{ is $A$-primitive}\right\}. 
\]
This construction leads to a generalized form of M\"obius inversion between the $\mathcal{A}$-convolutions. 
We can consider A-convolution sums for the fixed set $\mathcal{A}$ in the following two flavors for any two 
arithmetic functions $f,g$:
\begin{align*} 
S_{\mathcal{A},1}(f, g; n) & := \sum_{\substack{d|n \\ d \in \mathcal{A}}} f(d) g\left(\frac{n}{d}\right), \\ 
S_{\mathcal{A},2}(f, g; n) & := \sum_{\substack{d|n \\ d \in \mathcal{A}}} f(d) g\left(\frac{n}{d}\right) 
	\chi_{\mathcal{A}}\left(\frac{n}{d}\right). 
\end{align*} 
We can find an inverse function for $f$ 
with respect to this A-convolution if and only if $1 \in A$ and $f(1) \neq 0$. 

\SubsubsectionGTThesisFormatted{Unitary convolutions: Restricted divisor sums (UCVL type sums)} 
\label{example_UnitaryCvlInversion} 

A so-called \emph{unitary convolution} is a divisor sum in which both the index of summation $d$ and the 
residual index $\frac{n}{d}$ are relatively prime:
\[
(f \odot g)(n) := \sum_{\substack{d|n \\ \left(d, \frac{n}{d}\right) = 1}} f(d) g\left(\frac{n}{d}\right). 
\]
An arithmetic function $f$ is invertible with respect to unitary convolution, i.e., there exists a (unique) 
function $f^{-1}$ such that $f \odot f^{-1} = \varepsilon \equiv \delta_{n,1}$, if and only if 
$f(1) \neq 0$. If $f(1) = 1$, then Cohen has given a simple formula to express the inverse of any invertible 
$f$ with respect to the $\odot$ convolution operator \cite{HANDBOOKNT-2004}: 
\[
f^{-1}(n) = \sum_{k=1}^{\omega(n)} (-1)^{k} \times \sum_{d_1 d_2 \times\cdots\times d_k = n} \left(
     \prod_{i=1}^k f(d_i)\right). 
\]

\SubsubsectionGTThesisFormatted{K-convolutions: Divisor sums with a kernel weight function (KCVL type sums)} 

Let the kernel function $K: \mathbb{N} \times \mathbb{N} \rightarrow \mathbb{C}$ be defined on all 
ordered pairs $(n, d)$ such that $n \geq 1$ and $d|n$. 
We define the $K$-convolution of two arithmetic functions $f, g$ to be 
\[
(f \ast_K g)(n) := \sum_{d|n} f(d) g\left(\frac{n}{d}\right) K(n, d). 
\]
We can define generating function factorizations of these sums of the form 
\[
(f \ast_K g)(n) = [q^n] \frac{1}{\mathcal{C}(q)} \times 
     \sum_{m \geq 1} \left(\sum_{k=1}^{m} 
     e_{n,k}(K, \mathcal{C}; g) f(k)\right) q^m. 
\]
Similarly, provided reasonable expansions of the function $K$ (e.g., some recurrence relations or other structural properties), 
we can find inverse functions of an arithmetic function $f$ with respect to $K$-convolution 
(see Section \ref{subSection_GenSummatoryFuncIdents_KCvls}). 

\begin{example}[$B$-convolutions]
Let $\nu_p(n)$ denote the maximum exponent of the prime $p$ in the factorization of $n$. 
That is, 
\[
\nu_p(n) = \begin{cases}
     \alpha, & \text{ if $n \geq 2$ and $p^{\alpha} || n$; } \\ 
     0, & \text{ otherwise. } 
     \end{cases} 
\]
For integers $n \geq d \geq 1$, let the function 
$B(n, d) := \prod_{p|n} \binom{\nu_p(n)}{\nu_p(d)}$ where the product runs over all 
prime divisors of $n$. 
We have an inversion formula given by 
\[
f(n) = \sum_{d|n} g(d) B(n, d) \iff g(n) = \sum_{d|n} f(d) \lambda(d) B(n, d), 
\]
where $\lambda(n) = (-1)^{\Omega(n)}$ is the \emph{Liouville lambda function} 
\cite{HANDBOOKNT-2004}. 
\end{example} 

\SubsubsectionGTThesisFormatted{Discrete convolutions with respect to an index set (DCVL type sums)} 
\label{subsubSection_SpecialSumCharExs_DCVL-1} 

We can form another variant of the typical discrete convolution of coefficients 
(or Cauchy product) resulting from the pointwise multiplication of 
two ordinary generating functions. This generalization involves 
summations of the form 
\[
S_{f,g}(\mathcal{A}; n) := \sum_{k \in \mathcal{A}_n} f(k) g(n+1-k), 
     \mathcal{A}_n \subseteq \{1,2,\ldots,n\} \subset \mathbb{Z}^{+}. 
\]
Given the relation to a convolution of generating functions (or formal power series), it is no 
surprise that we can ``\emph{encode}'' and invertibly 
``\emph{decode}'' auxiliary sequences by multiplying 
an arbitrary OGF by another OGF and its reciprocal. Hence, discrete convolution sums are in general 
invertible operations from which we can solve for either $f$ or $g$. 

More generally, we can define sums of the following type 
for some bivariate kernel weight function, 
$\mathcal{D}: \mathbb{Z}^{+} \times \mathbb{Z}^{+} \rightarrow \mathbb{C}$, 
that is lower triangular and invertible:
\[
(f \boxdot_{\mathcal{D}} g)(n) := \sum_{k=1}^{n} f(k) g(n+1-k) \mathcal{D}(n, k), n \geq 1. 
\]
Sums of this type have familiar expressions by matrix-vector products involving Topelitz matrices, 
which are themselves common and well studied to express the discrete convolutions of sequences 
(see Section \ref{Section_CanonicalReprsOfFactThms_KernelBasedDCVL}). 

\SubsubsectionGTThesisFormatted{Summatory function weighted sums (APT type sums)} 

For $\mathcal{A}_n \subseteq [1, n] \cap \mathbb{Z}^{+}$, consider defining sums of the form 
\[
S_{f,g}(\mathcal{A}; n) := \sum_{k \in \mathcal{A}_n} f(k) g\left(\Floor{n}{k}\right). 
\]
When the sums defined by $S_{f,g}(\mathcal{A}; n)$ 
are finite, and provided that $f$ is Dirichlet invertible
with $f(1) \neq 0$, there are inversion formulas to extract $g(x)$ from 
these sums in the special case where $\mathcal{A}_n = \{1,2,\ldots, n\}$ \cite{APOSTOLANUMT}. 
In general, a consequence of 
\emph{Perron's formula} from complex analysis and analytic number theory provides that if 
$G(s) := \sum_{n \geq 1} g(n) n^{-s}$ is the DGF of $g$, and if 
$F(s) := \int_0^{\infty} f(x) x^{s-1} dx$ is a Mellin transform of $f$, then we have 
contour integral representation given on the left-hand-side of 
\[
\sum_{n \geq 1} f(n) g\left(\frac{x}{n}\right) = 
     \frac{1}{2\pi\imath} \times \int_{c-\imath\infty}^{c+\imath\infty} 
     F(s) G(s) x^{s} ds, c > \max(\sigma_{a,f}, \sigma_{a,g}). 
\]
There is another formula for the summatory function of the Dirichlet convolution 
of any two arithmetic functions $f,h$ of the following form \cite{APOSTOLANUMT}: 
\[
\sum_{n \leq x} \sum_{d|n} f(d) h\left(\frac{n}{d}\right) = \sum_{n \leq x} 
     f(n) H\left(\Floor{x}{n}\right), x \geq 1. 
\]
The notation $H(x) := \sum_{n \leq x} h(n)$ used to state the last equation is 
the summatory function formed by the unweighted partial sums of $h$. 

\SubsectionGTThesisFormatted{Overview of topics in the manuscript}

Within this dissertation, we discuss the results in 
publications based on work completed during 2016--2021 
by the author (MDS) and from two-author collaborations with peers
about generating function based factorization theorems. 
These factorization theorem expansions express the coefficients generated by Lambert series 
and certain restricted GCD type sums that enumerate the sequences of 
partial sums of any arithmetic function $f$. 
We spend several sections recalling the proofs and relevant constructions 
behind articles published in \emph{Acta Arithmetica}, the \emph{Ramanujan Journal}, the 
\emph{American Mathematical Monthly} and \emph{INTEGERS}. 
The articles published in peer-reviewed journals over this time with MDS as the sole author 
include \cite{AA,MDS-COMBRESTRDIVSUMS-INTEGERS}. 
The publications that are coauthored work with MDS compiled include 
\cite{MDS-MERCA-AMM,MERCA-SCHMIDT-LSFACTTHM,MERCA-SCHMIDT-PN,MERCA-SCHMIDT-RAMJ,MOUSAVI-SCHMIDT-2019}. 

We consider the most general form of the convolution type sums defined by the 
next $\mathcal{D}$-convolution type sums. 
For any arithmetic functions $f$ and $g$, and a lower triangular kernel function 
$K(n, k)$ that is unambiguously defined for all integers $n \geq k \geq 1$, we define the 
\emph{$\mathcal{D}$-convolution of $f$ and $g$ at $n$} by 
\begin{equation}
\label{eqn_DCvlSummationBasedExpDefs_restated_v2-Intro_def_v1}
(f \boxdot_{\mathcal{D}} g)(n) := \sum_{k=1}^{n} f(k) g(n+1-k) \mathcal{D}(n, k), n \geq 1. 
\end{equation}
In Section \ref{Section_CanonicalReprsOfFactThms_KernelBasedDCVL}, 
we formulate a rigorous answer to 
Question \ref{question_PrecursorToCorrStats_v1} 
for sums of this most general type and state 
related conjectures that remain open at the time of this publication. 
Namely, 
we will last turn to focus on newer expository results that 
elaborate on a topical open problem rephrased by Michael Lacey 
as a remaining loose end from my Ph.D. 
oral exam presentation at Georgia Tech in August of 2020. 

\SubsectionGTThesisFormatted{Rationale for defining ``canonically best'' 
            generating function factorizations}
\label{subSection_Intro_NotionsOfCanonicallyBest_v1}

There is a natural question that works its way into the analysis of the prior research and 
publications by Merca and Schmidt that we have summarized above. 
One reflection, in hindsight, as to why these seemingly simple expansions related to Lambert 
series generating functions resulted in so many acceptances in excellent peer-reviewed journals 
is that they make special, and as at least one reviewer had pointed out ``rare'', connections 
between classically multiplicative based constructions and the theory of partitions. 
Prior to those several publications, only G.~E.~Andrews and a handful of other authors 
had found such relations, and none yet it seems so general and clear cut to spot. 
The choice of factorizing the Lambert series OGF expansions by inserting a multiple of the 
generating function for the partition function, $p(n) = [q^n] (q; q)_{\infty}^{-1}$, 
leads to a representation for the matrices $s_{n,k}$ and $s_{n,k}^{-1}$ that both 
involve special functions that are central to 
the theory of partitions \cite{ANDREWS}. The initial relation of the 
$s_{n,k}$ to partition theoretic constructions was proved independently by Merca in 2017 
\cite{MERCA-LSFACTTHM}, and seen from a different lense in \cite{AA}, 
near the same time we decided to collaborate on generalizing this material. 

\begin{example}
We can project a sense of distortion (versus similarity) between two 
tuples of values (as truncated vectors of OGF coefficients) onto an easy to identify image 
to visualize our intuitions in the LGF series case. 
This visualization relies on how clearly the 
projected data (via computation of and convolution with a correlation matrix) 
allows us to look at \emph{Tux}, the classic 
good-luck-forebearing Linux penguin, depicted by 
Figure \ref{figure_TuxedoPenguinLGFCorrelationStatsVisualizationExample_v1} 
in the form of his traditional Linux kernel emblem. 
The original Linux penguin image is modified here to assist with clearly depicting
distinguishing features in the correlation statistic values computed by the computer program 
we have used to generate the figure, e.g., to clearly distinguish distortions projected onto 
the original or a lack thereof. 
We use a variant of the built-in image processing functions within modern 
releases of \emph{Mathematica} to generate these images. 
By inspection of the featured functions, $\mathcal{C}(q)$, in the figure that satisfy 
$\mathcal{C}(0) \notin \{0, \pm \infty\}$ with integer-valued series coefficients, 
we conclude that our choice of 
$\mathcal{C}(q) := (q; q)_{\infty}$ does indeed appear to be optimal! 
This is by no means a conclusive proof of the qualitative notion of optimality we have predicted. 
However, it is a convincing 
illustration of a heuristic feature that we wish to 
quantify and make precise for completeness in the last section of the manuscript.
\end{example}

\begin{figure}[ht!]
 
  \caption[Visualization of correlation statistics for LGF factorizations]{
           Correlation statistics projected onto the image of Tux for various choices of the 
           OGF $\mathcal{C}(q)$ (with integer coefficients) to show a visual comparison of how 
           well related the corresponding sequences are. Distortions of the original image indicate 
           a less well correlated set of sequences corresponding to that choice of the 
           $\mathcal{C}(q)$ that defines the precise form of the LGF factorization theorem coefficients.}
  \label{figure_TuxedoPenguinLGFCorrelationStatsVisualizationExample_v1} 
 
  \centering
  \begin{subfigure}[t]{.3\linewidth}
    \centering\includegraphics[scale=0.75]{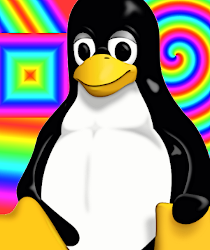}
    \caption{Original image.}
  \end{subfigure}
  \begin{subfigure}[t]{.3\linewidth}
    \centering\includegraphics[scale=0.75]{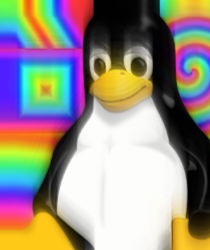}
    \caption{With $\mathcal{C}(q) := (q; q)_{\infty}$.}
  \end{subfigure}
  \begin{subfigure}[t]{.3\linewidth}
    \centering\includegraphics[scale=0.75]{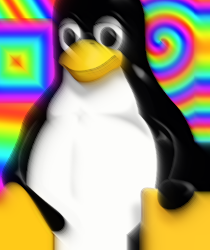}
       \caption{With $\mathcal{C}(q) := (q; q)_{\infty}^{-1}$.}
  \end{subfigure}
  
  \bigskip

  \begin{subfigure}[t]{.3\linewidth}
    \centering\includegraphics[scale=0.75]{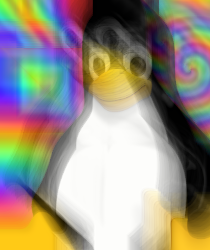}
    \caption{With $\mathcal{C}(q) := (q^2; q^5)_{\infty}$.}
  \end{subfigure}
  \begin{subfigure}[t]{.3\linewidth}
    \centering\includegraphics[scale=0.75]{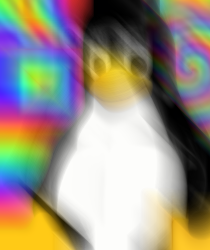}
    \caption{With $\mathcal{C}(q) := (1-q)^{-\frac{3}{2}}$.}
  \end{subfigure}
  \begin{subfigure}[t]{.3\linewidth}
    \centering\includegraphics[scale=0.75]{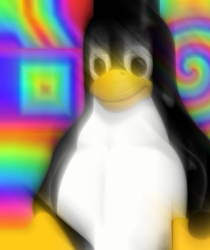}
       \caption{With $\mathcal{C}(q) := (1-q)^{-1}$.}
  \end{subfigure}

\end{figure}

\clearpage
\newpage

\SectionGTThesisFormatted{Factorization theorems for Lambert series generating functions} 
\label{Section_LGFFactTheorems_PriorWork} 

In this section, we will recall results from the following publications: 
\cite{AA,MERCA-SCHMIDT-RAMJ}. 
In Section \ref{subSection_LGFFactThms_HadamardProductsAndDerivativeExps}, 
we also reproduce the results from an unpublished 
manuscript of MDS from 2017 that proves several new, ``exotic sum`` 
type identities involving the pointwise products of 
classical multiplicative functions \cite{MDS-HADAMARD-FACTTHMS}. 

\SubsectionGTThesisFormatted{Lambert series generating functions} 

\begin{definition}
In our most general setting, we define the \emph{generalized Lambert series} expansion 
for integers $0 \leq \beta < \alpha$ and any fixed arithmetic $f$ as 
\[
L_f(\alpha, \beta; q) := \sum_{n \geq 1} \frac{f(n) q^{\alpha n - \beta}}{1-q^{\alpha n-\beta}},  
     |q| < 1. 
\]
The series coefficients of the Lambert series generating function $L_f(\alpha, \beta; q)$ 
are the divisor sums 
\[
[q^n] L_f(\alpha, \beta; q) = \sum_{\alpha d-\beta| n} f(d). 
\]
If we set $(\alpha, \beta) := (1, 0)$, the we recover the classical form of the Lambert series (LGF) 
construction denoted by the generating functions 
$L_f(q) \equiv L_f(1, 0; q)$ in \eqref{eqn_Lfq_LSeriesGF_def_v1}.
\end{definition}

\begin{example}[Famous LGF expansions]
Ramanujan discovered the following remarkable identities 
\cite[\S 2]{RAMANUJAN-LSERIES-SURVEY}: 
\begin{subequations}
\begin{align}
\sum_{n \geq 1} \frac{(-1)^{n-1} q^n}{1-q^n} & = \sum_{n \geq 1} \frac{q^n}{1+q^n} \\ 
\sum_{n \geq 1} \frac{n q^n}{1-q^n} & = \sum_{n \geq 1} \frac{q^n}{(1-q^n)^2} \\ 
\sum_{n \geq 1} \frac{(-1)^{n-1} n q^n}{1-q^n} & = \sum_{n \geq 1} \frac{q^n}{(1+q^n)^2} \\ 
\sum_{n \geq 1} \frac{q^n}{n(1-q^n)} & = \sum_{n \geq 1} \frac{q^n}{1+q^n} \\ 
\sum_{n \geq 1} \frac{(-1)^{n-1} q^n}{n(1-q^n)} & = \sum_{n \geq 1} \log\left(\frac{1}{1-q^n}\right) \\ 
\sum_{n \geq 1} \frac{\alpha^n q^n}{1-q^n} & = \sum_{n \geq 1} \log\left(1+q^n\right) \\ 
\sum_{n \geq 1} \frac{n^2 q^n}{1-q^n} & = \sum_{n \geq 1} \frac{q^n}{(1-q^n)^2} \sum_{k=1}^{n} \frac{1}{1-q^k}. 
\end{align}
\end{subequations}
We also have the following well-known ``classical'' examples of Lambert series identities 
\cite[\S 27.7]{NISTHB} \cite[\S 17.10]{HARDYWRIGHT} \cite[\S 11]{APOSTOLANUMT}:
\begin{subequations}
\begin{align} 
\label{eqn_WellKnown_LamberSeries_Examples} 
\sum_{n \geq 1} \frac{\mu(n) q^n}{1-q^n} & = q, \\ 
\sum_{n \geq 1} \frac{\phi(n) q^n}{1-q^n} & = \frac{q}{(1-q)^2}, \\ 
\sum_{n \geq 1} \frac{n^{\alpha} q^n}{1-q^n} & =  
     \sum_{m \geq 1} \sigma_{\alpha}(n) q^n, \\ 
\sum_{n \geq 1} \frac{\lambda(n) q^n}{1-q^n} & = \sum_{m \geq 1} q^{m^2}, \\ 
\sum_{n \geq 1} \frac{\Lambda(n) q^n}{1-q^n} & = \sum_{m \geq 1} \log(m) q^m, \\ 
\sum_{n \geq 1} \frac{|\mu(n)| q^n}{1-q^n} & = \sum_{m \geq 1} 2^{\omega(m)} q^m, \\ 
\sum_{n \geq 1} \frac{J_t(n) q^n}{1-q^n} & = \sum_{m \geq 1} m^t q^m, \\ 
\sum_{n \geq 1} \frac{\mu(\alpha n) q^n}{1-q^n} & = - \sum_{n \geq 0} q^{\alpha^n}, \alpha \in \mathbb{P}.
\end{align}
\end{subequations}
\end{example}

There is a natural correspondence between a sequence's OGF, and its 
Lambert series generating function. Namely, if $\widetilde{F}(q) := \sum_{m \geq 1} f(m) q^m$ is the OGF of 
$f$, then $$L_f(\alpha, \beta; q) = \sum_{n \geq 1} \widetilde{F}(q^{\alpha n-\beta}).$$ 
The Lambert series over the convolution $(f \ast g)(n)$ is given by the double sum 
\[
L_{f \ast g}(q) = \sum_{n \geq 1} f(n) L_g(q^n), |q| < 1. 
\]
We have by M\"obius inversion that the \emph{ordinary generating function} (OGF) of $f$ is 
given by 
\[
L_{f \ast \mu}(q) = \sum_{n \geq 1} f(n) q^n. 
\]
Higher-order $j^{th}$ derivatives for integer order $j \geq 1$ can be obtained by differentiating 
the Lambert series expansions termwise in the forms of 
\begin{subequations}
\label{eqn_LambertSeriesDerivsGFExps} 
\begin{align}
q^j \times D_q^{(j)}\left[\frac{q^n}{1-q^n}\right] & = 
     \sum_{m=0}^j \sum_{k=0}^m \gkpSI{j}{m} 
     \gkpSII{m}{k} \frac{(-1)^{j-k} k! i^m}{(1-q^i)^{k+1}}, \\ 
     & = \sum_{r=0}^j \left(\sum_{m=0}^j \sum_{k=0}^m \gkpSI{j}{m}
     \gkpSII{m}{k} \binom{j-k}{r} \frac{(-1)^{j-k-r} k! i^m}{(1-q^i)^{k+1}}\right) q^{(r+1)i}. 
\end{align}
\end{subequations}
By the \emph{binomial series} generating functions whose coefficients are 
given by $[z^n] (1-z)^{m+1} = \binom{n+m}{m}$, we find that 
\[
[q^n] \left(\sum_{n \geq t} \frac{f(n) q^{mn}}{(1-q^n)^{k+1}}\right) = 
     \sum_{\substack{d|n \\ t \leq d \leq \Floor{n}{m}}} 
     \binom{\frac{n}{d} - m + k}{k} f(d), 
\]
for positive integers $m,t \geq 1$ and $k \geq 0$. 
This identity leads to explicit closed-form expressions for the coefficients of the 
generating functions in \eqref{eqn_LambertSeriesDerivsGFExps} 
by so-termed ``restricted'' divisor sums of the form of the right-hand-side of the last equation 
\cite{MDS-COMBRESTRDIVSUMS-INTEGERS}. 

\SubsectionGTThesisFormatted{Relating the multiplicativity of LGF coefficients to the theory of partitions} 

The results cited next in this subsetion summarize the work in 
\cite{AA,MERCA-SCHMIDT-LSFACTTHM}. 
Let $f$ denote a fixed arithmetic function. 
For any integers $n \geq 1$, the $n^{th}$ coefficients of the LGF, $L_f(q)$, are 
generated by the truncated partial sums 
\[
[q^n] L_f(q) = [q^n]\left(\sum_{m=1}^{N} \frac{f(m) q^m}{1-q^m}\right), 1 \leq n \leq N < +\infty. 
\]
It is then natural to seek representations for the numerator and denominator of the 
rational generating function in $q$ that generates $(f \ast \mathds{1})(n)$ for any $n \geq 1$. 
The denominator terms tend to the limiting product given by the 
infinite \emph{$q$-Pochhammer symbol}, $(q; q)_{\infty} = \prod_{i \geq 1} (1-q^i)$. 
The sequence of numerator polynomials is more complicated to initially express. 

\begin{figure}[ht!]

\caption[Inverse matrices in the first forms of the LGF factorization theorems]{
         The first $29$ rows of the function $s_{i,j}^{-1}$ where the values of 
         Euler's partition function $p(n)$ are highlighted in blue and the remaining 
         values of the partition function $q(n)$ are highlighted in 
         purple or pink. } 
\label{figure_sijinv_first25_rows_highlighted} 

\centering
\includegraphics[width=\textwidth]{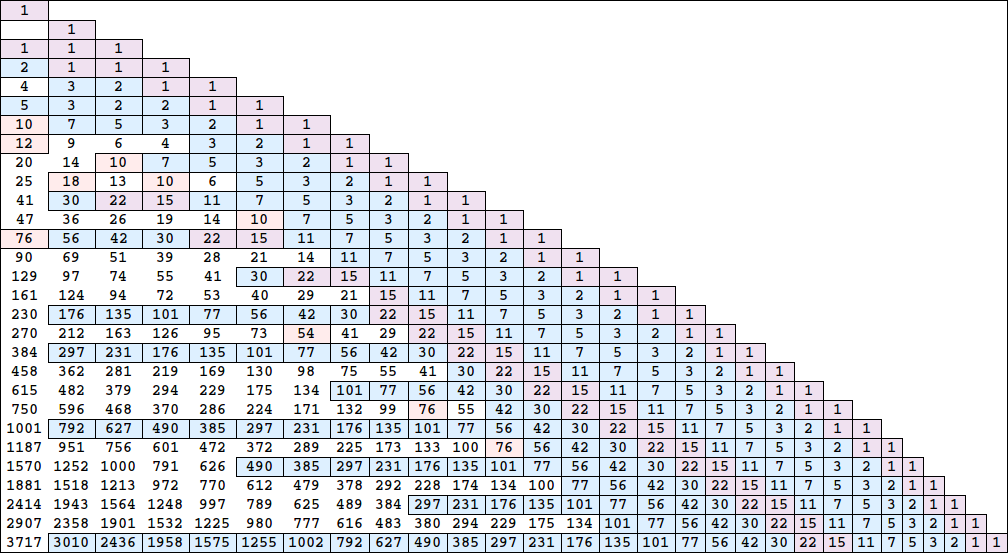}

\end{figure} 

\begin{definition}
\label{def_LGF_ani_afn_funcs_v1} 
For $1 \leq i \leq n$, let 
\begin{align*}
a_{n,i} & := \sum_{\substack{(k, s) : (s+1)i+\frac{k(3k\pm 1)}{2} = n \\ k, s \geq 0}} (-1)^k \\ 
     & \phantom{:} = 
     \sum_{b = \pm 1} \sum_{s=0}^{\left\lfloor \frac{n}{i} \right\rfloor - 1} 
     (-1)^{\left\lfloor \frac{\sqrt{24(n-(s+1)i)+1}-b}{6} \right\rfloor} \cdot 
     \Iverson{\frac{\sqrt{24(n-(s+1)i)+1}-b}{6} \in \mathbb{Z}}. 
\end{align*}
Let the $f$-scaled sums of this triangular sequence to be defined as follows: 
\[
a_f(n) := \sum_{i=1}^{n} a_{n,i} f(i), n \geq 1.
\]
\end{definition}

The arithmetic function $f(n) := n^{\alpha}$ for $\alpha > 0$ such that 
$\sigma_{\alpha}(n) = (f \ast \mathds{1})(n)$ defines the class of 
\emph{generalized sum-of-divisors functions} yields the series expansion given by 
\begin{align} 
\label{eqn_afn_special_cases} 
\sum_{m \geq 0} a_{n^\alpha}(m) q^m & = 1 + q + 2^{\alpha} q^2 + 
     \left(-1-2^{\alpha}+3^{\alpha}\right)q^3 + 
     \left(-1-3^{\alpha}+4^{\alpha}\right)q^4 \\ 
\notag 
     & \phantom{=q\ } + 
     \left(-1-2^{\alpha}-3^{\alpha}-4^{\alpha}+5^{\alpha}\right)q^5 + 
     \left(3^{\alpha}-4^{\alpha}-5^{\alpha}+6^{\alpha}\right) q^6 \\ 
\notag 
     & \phantom{=q\ } + 
     \left(-3^{\alpha}-5^{\alpha}-6^{\alpha}+7^{\alpha}\right) q^7 + \cdots. 
\end{align}
For any $n \geq 1$, the $n \times n$ square matrix $A_n := (a_{i,j})_{1 \leq i,j \leq n}$ is 
lower triangular with ones on its diagonal, and is hence invertible. 
It turns out that in the context of rationalizing the partial sums that generate 
$(f \ast \mathds{1})(n)$ for $n \geq 1$, 
the formula for the entries of these matrices is key. 
Observe the next few cases given in 
Table \ref{table_LGF_AnAnInve_MatrixExamples_v1}, 
as well as the highlighted suggestions found in 
Figure \ref{figure_sijinv_first25_rows_highlighted}. 

\begin{table}[h!]

\caption[LGF factorization matrices]{}
\label{table_LGF_AnAnInve_MatrixExamples_v1}

\begin{center}
     \begin{tabular}{||c||c|c||} \hline\hline
$n$ & $A_n$ & $A_n^{-1}$ \\ \hline 
1 & $[1]$ & $[1]$ \\ 
2 & $\begin{bmatrix} 1 & 0 \\ 0 & 1 \end{bmatrix}$ & 
    $\begin{bmatrix} 1 & 0 \\ 0 & 1 \end{bmatrix}$ \\ 
3 & $\begin{bmatrix} 1 & 0 & 0 \\ 0 & 1 & 0 \\ -1 & -1 & 1 \end{bmatrix}$ & 
    $\begin{bmatrix} 1 & 0 & 0 \\ 0 & 1 & 0 \\ 1 & 1 & 1 \end{bmatrix}$ \\ 
4 & $\begin{bmatrix} 1 & 0 & 0 & 0 \\ 0 & 1 & 0 & 0 \\ -1 & -1 & 1 & 0 \\ 
                     -1 & 0 & -1 & 1 \end{bmatrix}$ & 
    $\begin{bmatrix} 1 & 0 & 0 & 0 \\ 0 & 1 & 0 & 0 \\ 1 & 1 & 1 & 0 \\ 
                     2 & 1 & 1 & 1 \end{bmatrix}$ \\ 
5 & $\begin{bmatrix} 1 & 0 & 0 & 0 & 0 \\ 0 & 1 & 0 & 0 & 0 \\ 
                     -1 & -1 & 1 & 0 & 0 \\ -1 & 0 & -1 & 1 & 0 \\ 
                     -1 & -1 & -1 & -1 & 1 \end{bmatrix}$ & 
    $\begin{bmatrix} 1 & 0 & 0 & 0 & 0 \\ 0 & 1 & 0 & 0 & 0 \\ 
                     1 & 1 & 1 & 0 & 0 \\ 2 & 1 & 1 & 1 & 0 \\ 
                     4 & 3 & 2 & 1 & 1 \end{bmatrix}$ \\ 
\hline\hline
\end{tabular} 
\end{center} 

\end{table}

\begin{theorem}[Matrix factorizations defining $f(n)$] 
\label{prop_MatrixEquations_fi_eq_AinvBbm} 
Let the intermediate sequence of sums be defined for fixed $f$ by 
\[
B_{f \ast \mathds{1},m} :=  (f \ast \mathds{1})(m+1) - \sum_{b = \pm 1} \sum_{k=1}^{\lfloor \frac{\sqrt{24m+1}-b}{6} \rfloor} 
     (-1)^{k+1} (f \ast \mathds{1})\left(m+1-\frac{k(3k+b)}{2}\right), m \geq 0. 
\]
For all $n \geq 1$, we have the following matrix factorization equations: 
\begin{align} 
\label{eqn_fn_matrix_eqn}
\begin{bmatrix} f(1) \\ f(2) \\ \vdots \\ f(n) \end{bmatrix} & = 
     A_n^{-1} \begin{bmatrix} B_{f \ast \mathds{1}, 0} \\  B_{f \ast \mathds{1}, 1} \\ \vdots \\  B_{f \ast \mathds{1}, n-1} \end{bmatrix}
\end{align} 
\end{theorem} 

\begin{theorem}[Recurrence relations for $(f \ast \mathds{1})(n)$] 
\label{prop_bn_recs} 
For all $n \geq 1$, we have the following recurrence relation: 
\begin{align*} 
(f \ast \mathds{1})(n+1) & = \sum_{b = \pm 1} \left(\sum_{k=1}^{\lfloor \frac{\sqrt{24n+1}-b}{6} \rfloor} 
     (-1)^{k+1} (f \ast \mathds{1})\left(n+1-\frac{k(3k+b)}{2}\right)\right) + a_f(n+1). 
\end{align*} 
\end{theorem} 

\begin{lemma}[Partial sums of the Lambert series, $L_f(q)$] 
\label{lemma_LS_partial_sum_exps} 
Let $(q; q)_n := (1-q)(1-q^2) \cdots (1-q^{n})$ denote the 
\emph{$q$-Pochhammer symbol} \citep[\S 17.2]{NISTHB}, and suppose that the functions 
$\poly_{i,m}(f; q)$ are polynomials in $q$ with coefficients depending on $f$ 
(for $i = 1,2$) each of whose degree is linear in the fixed index $m$. 
For a fixed arithmetic function $f$ and for all integers $m \geq 0$ we 
have that 
\begin{subequations} 
\begin{align} 
\label{eqn_bmp1_coeff_exp_va} 
g_f(m+1) & = [q^m]\left(\frac{1}{q} \times \sum_{n=1}^{m+1} 
     \frac{f(n) q^n}{1-q^n}\right) \\ 
\label{eqn_bmp1_coeff_exp_vb} 
     & = 
     [q^m]\left(\frac{\frac{1}{q} (q; q)_{m+1} 
     \left(\frac{f(1) q}{1-q} + 
     \frac{f(2) q^2}{1-q^2} + \cdots + \frac{f(n) q^{m+1}}{1-q^{m+1}}\right)}{ 
     (1-q)(1-q^2) \times \cdots \times (1-q^{m+1})}\right) \\ 
\label{eqn_bmp1_coeff_exp_vc} 
     & = 
     [q^m]\left(\frac{\sum_{1 \leq i \leq m+1} a_f(i) q^i + 
     q^{m+2} \poly_{1,m}(f; q)}{ 
     1 + \sum\limits_{b=\pm 1} \sum\limits_{k=1}^{\left\lfloor \frac{\sqrt{24m+25}+1}{6} \right\rfloor} 
     (-1)^k q^{\frac{k(3k+b)}{2}} + 
     q^{m+2} \poly_{2,m}(f; q)}\right). 
\end{align} 
\end{subequations} 
\end{lemma} 
\begin{proof} 
To justify \eqref{eqn_bmp1_coeff_exp_va}, we observe that for all integers 
$m \geq 1$ and $1 \leq i \leq m$, we have that 
\begin{equation*} 
[q^i] \left(L_f(q) - \sum_{n > m} \frac{f(n) q^n}{1-q^n}\right) = 0, 
\end{equation*} 
i.e., that the $m^{th}$ partial sums of $L_f(q)$ generate the coefficients 
$[q^k] L_f(q)$ for each $k$ in the range $1 \leq k \leq m$. 
This is easy enough to see by considering the numerator multiples, $q^n$, of the 
geometric series, $(1-q^n)^{-1}$, in the individual Lambert series terms in the infinite 
series expansion of $L_f(q)$. 
The result in \eqref{eqn_bmp1_coeff_exp_vb} follows immediately from 
\eqref{eqn_bmp1_coeff_exp_va} by combining the terms in the first partial sum, and 
implies the third result in \eqref{eqn_bmp1_coeff_exp_vc} in two key ways. 

First, the respective form of the denominator terms in 
\eqref{eqn_bmp1_coeff_exp_vc} follows from the statement of 
\emph{Euler's pentagonal number theorem}, which states that 
\citep[\S 19.9, Thm.\ 353]{HARDYWRIGHT} 
\begin{align*} 
(q; q)_{\infty} & = \sum_{n=-\infty}^{\infty} (-1)^n q^{\frac{n(3n+1)}{2}} = 
     1 + \sum_{n \geq 1} (-1)^n \left(q^{\frac{k(3k-1)}{2}} + q^{\frac{k(3k+1)}{2}}\right) \\ 
     & = 
     1 - q - q^2 + q^5+q^7-q^{12}-q^{15}+q^{22}+q^{26} - \cdots. 
\end{align*} 
In particular, we see that the pentagonal number theorem shows that 
\begin{equation*} 
[q^i] (1-q)(1-q^2) \times \cdots \times (1-q^n) = 
     \begin{cases} 
     (-1)^k, & \text{ if $i = \frac{k(3k\pm 1)}{2}$; } \\ 
     0, & \text{otherwise, } 
     \end{cases} 
     i \leq n.
\end{equation*} 
Since $(1-q^i)$ is a factor of $(q; q)_{n}$ for all $1 \leq i \leq n$, 
we see that both of the numerator and denominator of 
\eqref{eqn_bmp1_coeff_exp_vb} are polynomials in $q$, 
each with degree greater than $m+1$. 
This implies the correctness of the 
denominator polynomial form stated in \eqref{eqn_bmp1_coeff_exp_vc}. 

Secondly, since the geometric series in $q^i$ is expanded by 
\begin{equation*} 
\frac{1}{1-q^i} = \sum_{s \geq 0} q^{si}, |q| < 1, 
\end{equation*} 
for integers $i \geq 1$, we have by the definition of $a_f(n)$ given above 
that the first $m+1$ terms of the 
numerator expansion in \eqref{eqn_bmp1_coeff_exp_vc} are correct. 
Thus since the numerator in \eqref{eqn_bmp1_coeff_exp_vc} is polynomial in $q$, 
we have that this statement holds as well. 
\end{proof} 

\begin{proof}[Proof of Theorem \ref{prop_bn_recs}]
We use \eqref{eqn_bmp1_coeff_exp_vc} in Lemma \ref{lemma_LS_partial_sum_exps} 
to prove our result. 
If we let $\Num_m(q)$ and $\Denom_m(q)$ denote the numerator and denominator 
polynomials in \eqref{eqn_bmp1_coeff_exp_vc}, respectively, we see that by 
definition, $\deg_q \left\{\Num_m(q)\right\} < \deg_q \left\{ \Denom_m(q) \right\}$. 
Next, for any sequence, $\{f_n\}_{n \geq 0}$, generated by a rational generating function 
of the form 
\begin{align*} 
\sum_{n \geq 0} f_n q^n & = \frac{a_0+a_1q+a_2q^2+\cdots+a_{k-1}q^{k-1}}{ 
     1 - b_1 q - b_2 q^2 - \cdots - b_k q^{k}}, 
\end{align*} 
for some fixed integer $k \geq 1$, we can prove that $f_n$ satisfies 
at most a  
$k$-order finite difference equation with constant coefficients of the form 
\cite[\S 2.3]{LANDO-GFLECT} 
\begin{align*} 
f_n & = \sum_{i=1}^{\min(k, n)} b_i f_{n-i} + a_n \Iverson{0 \leq n < k}. 
\end{align*} 
Then since we define $f(n) = 0$ for all $n < 1$, 
since the $m^{th}$ partial sums of $L_f(q)$ generate $g_f(i)$ for all 
$1 \leq i \leq m$ by the lemma, and since $g_f(i) = 0$ for all $i < 1$, 
we see that \eqref{eqn_bmp1_coeff_exp_vc} implies our result. 
\end{proof}

\begin{proof}[Proof of Theorem \ref{prop_MatrixEquations_fi_eq_AinvBbm}] 
The theorem is a consequence of Theorem \ref{prop_bn_recs}. 
Specifically, by rearranging terms in the result from the previous theorem, 
we see that 
\begin{align*} 
\tag{i} 
A_n \begin{bmatrix} f(1) \\ f(2) \\ \vdots \\ f(n) \end{bmatrix} & = 
     \begin{bmatrix} B_{g_f,0} \\ B_{g_f,1} \\ \vdots \\ B_{g_f,n-1} \end{bmatrix}. 
\end{align*} 
Then by the definition of $a_{n,i}$, 
it is easy to see that 
$A_n$ is lower triangular with ones on its diagonals, and so is 
invertible for all $n \geq 1$. 
Thus by applying $A_n^{-1}$ to both sides of (i), we have proved 
\eqref{eqn_fn_matrix_eqn} in the statement of the theorem. 
\end{proof} 
\begin{cor}[Recurrence relations for the summatory function of $f \ast \mathds{1}$] 
\label{cor_Sigmabx_recs} 
Let the average order of the function $g_f := f \ast \mathds{1}$ be denoted by 
$\Sigma_{g_f,x} := \sum_{n \leq x} (f \ast \mathds{1})(n)$. 
Then for all $n \geq 1$, we have that 
\begin{align} 
\label{eqn_Sigmabx_recs}
\Sigma_{g_f,n+1} & = \sum_{b = \pm 1} \left( 
     \sum_{k=1}^{\left\lfloor \frac{\sqrt{24n+1}-b}{6} \right\rfloor + 1} 
     (-1)^{k+1} \Sigma_{g_f,n+1-\frac{k(3k+b)}{2}}\right) + 
     \sum_{k=1}^{n} a_f(k+1). 
\end{align} 
\end{cor}
\begin{proof}[Proof of Corollary \ref{cor_Sigmabx_recs}]
We can show directly by computation that the statement is true for $n = 1$. 
For some $j \geq 1$, suppose that the hypothesis in \eqref{eqn_Sigmabx_recs} is 
true for $n = j$. Then wee see that 
\begin{align*} 
\widetilde{\Sigma}_{g_f,j+1} & = 
     \sum_{b = \pm 1}  
     \sum_{k=1}^{\left\lfloor \frac{\sqrt{24j+25}-b}{6} \right\rfloor}
     (-1)^{k} \left(\Sigma_{g_f,j+1-\frac{k(3k+b)}{2}} + g_f\left(j+2-\frac{k(3k+b)}{2}\right)\right) \\ 
     & \phantom{=\sum1\ } + 
     \sum_{k=1}^{j+1} a_f(k+1) \\ 
\tag{by hypothesis} 
     & = 
     \Sigma_{g_f,j+1} + \sum_{b=\pm 1} 
     \sum_{k=1}^{\left\lfloor \frac{\sqrt{24j+25}-b}{6} \right\rfloor} 
     g_f\left(j+2-\frac{k(3k+b)}{2}\right) + a_f(j+2) \\ 
     & = 
     \Sigma_{g_f,j+1} + g_f(j+2) \\ 
     & = 
     \Sigma_{g_f,j+2}. 
\end{align*} 
The second to last of the previous equations follows from 
Theorem \ref{prop_bn_recs}, the fact that 
$\left\lfloor \frac{(\sqrt{24n+25}-b)}{6} \right\rfloor \geq \left\lfloor \frac{(\sqrt{24n+1}-b)}{6} \right\rfloor$, and 
since $g_f(i) = 0$ for all $i < 1$. 
\end{proof}

\begin{cor}
For integers $x \geq 1$, let the summatory function 
$$\Sigma_{\sigma,x} := \sum_{n \leq x} \sigma(n).$$ 
The function $\sigma(n) \equiv \sigma_1(n)$ is the ordinary sum-of-divisors function.
Then 
\begin{align*} 
\Sigma_{\sigma,x+1} & = \sum_{s = \pm 1} \left( \sum_{0 \leq n \leq x} 
     \sum_{k=1}^{\left\lfloor \frac{\sqrt{24n+25}-s}{6} \right\rfloor} 
     (-1)^{k+1} \frac{k(3k+s)}{2} p(x-n)\right). 
\end{align*} 
\end{cor}

\begin{cor}[Special cases]
Suppose that for any $m \geq 1$, we define the next functions as 
\begin{align*} 
B_m(\mu) & := \Iverson{m = 0} + \sum_{b = \pm 1} 
     \sum_{k=1}^{\left\lfloor \frac{\sqrt{24m+1}-b}{6} \right\rfloor} 
     (-1)^k \Iverson{m+1-\frac{k(3k+b)}{2} = 1} \\ 
B_m(\phi) & := m+1 - 
     \sum_{b = \pm 1} 
     \sum_{k=1}^{\left\lfloor \frac{\sqrt{24m+1}-b}{6} \right\rfloor} 
     (-1)^{k+1} \left(m+1-\frac{k(3k+b)}{2}\right) \\ 
B_m(\lambda) & := \Iverson{\sqrt{m+1} \in \mathbb{Z}} - 
     \sum_{b = \pm 1} 
     \sum_{k=1}^{\left\lfloor \frac{\sqrt{24m+1}-b}{6} \right\rfloor} 
     (-1)^{k+1} \Iverson{\sqrt{m+1-\frac{k(3k+b)}{2}} \in \mathbb{Z}}. 
\end{align*}
For all $n \geq 1$, we have that 
\begin{align*} 
\mu(n) & = \sum_{k=1}^n \left(\sum_{d|n} p(d-k) \mu\left(\frac{n}{d}\right)\right) B_{k-1}(\mu) \\ 
\phi(n) & = \sum_{k=1}^n \left(\sum_{d|n} p(d-k) \mu\left(\frac{n}{d}\right)\right) B_{k-1}(\phi) \\ 
\lambda(n) & = \sum_{k=1}^n \left(\sum_{d|n} p(d-k) \mu\left(\frac{n}{d}\right)\right) B_{k-1}(\lambda). 
\end{align*}
\end{cor}
\begin{proof}
These identities follow from the well known divisor sum relations providing that 
$\mu \ast \mathds{1} = \varepsilon$, $\phi \ast \mathds{1} = \operatorname{Id}_1$, and 
$\lambda \ast \mathds{1} = \chi_{\operatorname{sq}}$, where 
$\chi_{\operatorname{sq}}$ is the characteristic function of the squares and 
$\operatorname{Id}_1(n) = n$. 
\end{proof}

\begin{theorem}
\label{theorem_MainThm_InvMatrixDivSums} 
We have that for integers $1 \leq i \leq n$, 
\[
a_{n,i} = s_o(n, i) - s_e(n, i) = [q^n] (q; q)_{\infty} \times \frac{q^i}{1-q^i}, 
\]
where $s_o(n, k)$ and $s_e(n, k)$ are respectively the number of $k$'s in all 
partitions of $n$ into an odd (even) number of distinct parts. 
Moreover, the entries of inverse matrices, $A_n^{-1} := (a_{n,i}^{-1})_{1 \leq i,j \leq n}$, satisfy
\[
a_{n,i}^{-1} = \sum_{d|n} p(d-i) \mu\left(\frac{n}{d}\right). 
\]
The function $\mu(n)$ is the classical M\"obius function and 
$p(n) = [q^n] (q; q)_{\infty}^{-1}$ is the (ordinary, i.e., Euler) partition function. 
\end{theorem} 
\begin{proof}[Proof of Theorem \ref{theorem_MainThm_InvMatrixDivSums}]
The generating function expression in the first formula follows by rearranging the 
series. The partition theoretic interpretation of these coefficients is proved by 
Merca \cite[Thm.~1.2]{MERCA-LSFACTTHM}. The original proof due to Merca relies on 
an expansion of Lambert series by elementary symmetric polynomials. 

It remains to prove the inverse matrix formula in the second equation. 
An equivalent statement is to prove that 
for all $1 \leq k \leq n$, 
\[
p(n-k) = \sum_{d|n} a_{d,k}^{-1}. 
\]
Let $r \geq 1$ be a fixed lower index. 
Consider the expansion of the LGF of the function $f(n) := a_{n,r}^{-1}$.
Notice that the coefficients of the LGF factorization satisfy 
\begin{align*}
\sum_{d|n} a_{d,r}^{-1} & = \sum_{m=0}^{n} \sum_{j=1}^{n-m} \left( 
     s_o(n-m, j) - s_e(n-m, j)\right) a_{j,r}^{-1} p(m) \\ 
     & = \sum_{m=0}^{n} \delta_{n-r,m} p(m) = p(n-r). 
\end{align*}
The second equation in the above lines is a consequence of the matrix inverse 
orthogonality given by 
\[
\sum_{j=1}^{m} (s_o(m, j) - s_e(m, j)) a_{j,r}^{-1} = \delta_{m,r}, m \geq r \geq 1. 
\]
The claimed formula then follows by M\"obius inversion on the divisor sum. 
\end{proof} 

A consequence of the last proof is that the inverse matrix entries have the LGF
\begin{align*} 
\sum_{n \geq 1} \frac{a_{n,r}^{-1} q^n}{1-q^n} & = \frac{q^r}{(q; q)_{\infty}}, r \geq 1. 
\end{align*} 

\begin{example}[More special case expansions]
\label{cor_ExactFormulas_SpArithFns} 
For natural numbers $m \geq 0$, let the next component sequences 
defined in \cite{AA,MERCA-SCHMIDT-RAMJ} be defined by the formulas 
\begin{align*} 
B_{\phi,m} & = m+1 - \frac{1}{8}\Biggl(8 - 5 \cdot (-1)^{u_1} - 4 \left( 
     -2 + (-1)^{u_1} + (-1)^{u_2}\right) m \\ 
     & \phantom{=m+1 - \frac{1}{8}\Biggl(8\ } + 
     2 (-1)^{u_1} u_1 (3u_1+2) + 
     (-1)^{u_2} (6u_2^2+8u_2-3)\Biggr) \\ 
B_{\mu, m} & = \Iverson{m = 0} + \sum_{b = \pm 1} 
     \sum_{k=1}^{\left\lfloor \frac{\sqrt{24m+25}-b}{6} \right\rfloor} 
     (-1)^k \Iverson{m+1-\frac{k(3k+b)}{2} = 1} \\ 
B_{\lambda, m} & = 
     \Iverson{\sqrt{m+1} \in \mathbb{Z}} - \sum_{b = \pm 1} 
     \sum_{k=1}^{\left\lfloor \frac{\sqrt{24m+1}-b}{6} \right\rfloor} 
     (-1)^{k+1} \Iverson{\sqrt{m+1-\frac{k(3k+b)}{2}} \in \mathbb{Z}} \\ 
B_{\Lambda, m} & = 
     \log(m+1) - \sum_{b = \pm 1} 
     \sum_{k=1}^{\left\lfloor \frac{\sqrt{24m+1}-b}{6} \right\rfloor} 
     (-1)^{k+1} \log\left(m+1-\frac{k(3k+b)}{2}\right) \\ 
B_{|\mu|, m} & = 
     2^{\omega(m+1)} - \sum_{b = \pm 1} 
     \sum_{k=1}^{\left\lfloor \frac{\sqrt{24m+1}-b}{6} \right\rfloor} 
     (-1)^{k+1} 2^{\omega\left(m+1-\frac{k(3k+b)}{2}\right)} \\ 
B_{J_t, m} & = 
     (m+1)^t - \sum_{b = \pm 1} 
     \sum_{k=1}^{\left\lfloor \frac{\sqrt{24m+1}-b}{6} \right\rfloor} 
     (-1)^{k+1} \left(m+1-\frac{k(3k+b)}{2}\right)^t, 
\end{align*} 
where $u_1 \equiv u_1(m) := \lfloor \frac{\left(\sqrt{24m+1}+1\right)}{6} \rfloor$ and 
$u_2 \equiv u_2(m) := \lfloor \frac{\left(\sqrt{24m+1}-1\right)}{6} \rfloor$. 
Then we have that 
\begin{align*} 
\phi(n) & = \sum_{m=0}^{n-1} \sum_{d|n} p(d-m-1) \mu\left(\frac{n}{d}\right) B_{\phi,m} \\ 
\mu(n) & = \sum_{m=0}^{n-1} \sum_{d|n} p(d-m-1) \mu\left(\frac{n}{d}\right) B_{\mu,m} \\ 
\lambda(n) & = \sum_{m=0}^{n-1} \sum_{d|n} p(d-m-1) \mu\left(\frac{n}{d}\right) B_{\lambda,m} \\ 
\Lambda(n) & = \sum_{m=0}^{n-1} \sum_{d|n} p(d-m-1) \mu\left(\frac{n}{d}\right) B_{\Lambda,m} \\ 
|\mu(n)| & = \sum_{m=0}^{n-1} \sum_{d|n} p(d-m-1) \mu\left(\frac{n}{d}\right) B_{|\mu|,m} \\ 
J_t(n) & = \sum_{m=0}^{n-1} \sum_{d|n} p(d-m-1) \mu\left(\frac{n}{d}\right) B_{J_t,m}. 
\end{align*}
\end{example} 

\begin{remark}[Related constructions for variants of LGF series]
Merca showed another variant of the Lambert series factorization theorem 
stated in the form of \cite[Cor.\ 6.1]{MERCA-LSFACTTHM}
\begin{align*} 
\sum_{n \geq 1} \frac{f(n) q^{2n}}{1-q^n} & = \frac{1}{(q; q)_{\infty}} \times 
     \sum_{n \geq 1} \left(\sum_{k=1}^{\lfloor n/2 \rfloor} \left( 
     s_o(n-k, k) - s_e(n-k, k)\right) f(k) \right) q^n. 
\end{align*} 
If we consider the generalized Lambert series formed by taking derivatives of 
$L_f(q)$ as in \cite{MDS-COMBRESTRDIVSUMS-INTEGERS} in the 
context of finding new relations between the generalized sum-of-divisors functions, 
$\sigma_{\alpha}(n)$, we can similarly formulate new, alternate forms of the 
factorization theorems unified within this section so far. 
For example, suppose that $k, m \geq 0$ are integers and 
consider the factorization theorem resulting from an analysis of the following sums: 
\begin{align*}  
\sum_{n \geq 1} \frac{f(n) q^{(m+1)n}}{(1-q^n)^{k+1}} & = \frac{1}{(q; q)_{\infty}} \times 
     \sum_{n \geq 1} \left(\sum_{i=1}^{\lfloor \frac{n}{m+1} \rfloor} a_{n-m, i} \times 
     \frac{f(i)}{(1-q^i)^k} \right) q^n. 
\end{align*} 
We have by our factorization theorem that the previous series 
are expanded by 
\begin{align*}  
\sum_{n \geq 1} \frac{f(n) q^{(m+1)n}}{(1-q^n)^{k+1}} & = \frac{1}{(q; q)_{\infty}} \times 
     \sum_{n \geq 1} \left(\sum_{i=1}^{\left\lfloor \frac{n}{m+1} \right\rfloor} 
     \sum_{j=0}^{\left\lfloor \frac{n-m}{i} \right\rfloor} 
     \binom{k-1+j}{k-1} a_{n-m-ji, i} f(i) \right) q^n, 
\end{align*} 
When $m \geq k$ the series coefficients of these modified Lambert series 
generating functions are given by 
\begin{align*} 
\sum_{\substack{d|n \\ d \leq \left\lfloor \frac{n}{m+1} \right\rfloor}} 
     \scriptstyle{\binom{\frac{n}{d}-1-m+k}{k}} f(d) & = 
     \sum_{r=0}^n \sum_{i=1}^{\left\lfloor \frac{n-r}{m+1} \right\rfloor} 
     \sum_{j=0}^{\left\lfloor \frac{n-r-m}{i} \right\rfloor} 
     \binom{k-1+j}{k-1} a_{n-r-m-ji, i} f(i) p(r). 
\end{align*}
\end{remark}

\SubsectionGTThesisFormatted{Expansions of generalized LGFs}

The results in this section summarize the work from \cite{MERCA-SCHMIDT-RAMJ}. 
We extend the factorization theorem results proved above for the classical LGF
expansion cases to the generalized Lambert series cases 
defined in \eqref{eqn_GenLambertSeries_LaAlphaBetaq_def}. 
In general, for $\alpha > 1$ when $\bar{f}_n \neq f(n)$ the resulting matrix with entries 
$s_{n,k}(\alpha, \beta)$ is not invertible. We still arrive at some interesting cases, 
new identities and relations by considering the LGF series factorizations abstracted to this 
level of generality. 

\SubsubsectionGTThesisFormatted{A few special cases} 

\begin{prop}
\label{prop_first_spcase_result} 
For $|q|<1$, we have that 
\begin{align*}
\sum_{n=1}^{\infty} f(n) \frac{q^{2n-1}}{1-q^{2n-1}} = \frac{1}{(q;q^2)_\infty} \times 
     \sum_{n=1}^{\infty} \left(\sum_{k=1}^n s_{n,k}^{\ast}(2, 1) f(k)\right) (-1)^{n-1} q^n,
\end{align*}
where $s_{n,k}^{\ast}(2, 1)$ denotes the number of $(2k-1)$'s in all partitions of $n$ 
into distinct odd parts.
\end{prop} 
\begin{proof}
We consider the identity \cite[eq. 2.1]{MERCA-LSFACTTHM}, namely
\begin{align*}
\sum_{k=1}^n \frac{f(k) x_k}{1-x_k} = \left( \prod_{k=1}^{n} \frac{1}{1-x_k}\right) \times \left( 
	\sum_{k=1}^{n} \sum_{1\leq i_1 < \ldots <i_k\leq n} (-1)^{k-1} 
     (f({i_1})+\cdots +f({i_k})) x_{i_1}\cdots x_{i_k}\right).
\end{align*}
By this relation with $x_k$ replaced by the powers $q^{2k-1}$ of the series indeterminate, we get
\begin{align*}
\sum_{k=1}^n & \frac{f(k) q^{2k-1}}{1-q^{2k-1}} \\ 
     & = \frac{1}{(q;q^2)_n} \times \left(  \sum_{k=1}^{n} \sum_{1\leq i_1 < \ldots <i_k\leq n} (-1)^{k-1} 
     \left(f({i_1})+\cdots +f({i_k})\right) q^{(2i_1-1)+\cdots+(2i_k-1)}\right) .
\end{align*}	
The result follows directly from this identity considering the 
limiting case as $n \rightarrow \infty$.
\end{proof}

\begin{example}[Consequences of the proposition] 
The result in Proposition \ref{prop_first_spcase_result} 
allows us to derive many specialized identities involving Euler's partition function and various arithmetic functions. 
For $n \geq 1$ we have that 
\begin{align*}
	& \sum_{k=1}^{n} \sum_{2d-1|k} d^x Q(n-k) = \sum_{k=1}^n (-1)^{n-1} k^x s_{n,k}^{\ast}(2, 1),\\
	& \sum_{k=1}^{n} \sum_{2d-1|k} \mu(d) Q(n-k) = \sum_{k=1}^{n} (-1)^{n-1} \mu(k) s_{n,k}^{\ast}(2, 1), \\
	& \sum_{k=1}^{n} \sum_{2d-1|k} \varphi(d) Q(n-k) = \sum_{k=1}^n (-1)^{n-1} \varphi(k) s_{n,k}^{\ast}(2, 1),\\
	& \sum_{k=1}^{n} \sum_{2d-1|k} \lambda(d) Q(n-k) = \sum_{k=1}^n (-1)^{n-1} \lambda(k) s_{n,k}^{\ast}(2, 1),\\
	& \sum_{k=1}^n \sum_{2d-1|k} \log(d) Q(n-k) = \sum_{k=1}^n (-1)^{n-1} \log(k) s_{n,k}^{\ast}(2, 1),\\
	& \sum_{k=1}^n \sum_{2d-1|k} |\mu(d)| Q(n-k) = \sum_{k=1}^n (-1)^{n-1} | \mu(k) | s_{n,k}^{\ast}(2, 1),\\
	& \sum_{k=1}^n \sum_{2d-1|k} J_t(d) Q(n-k) = \sum_{k=1}^n (-1)^{n-1} J_t(k) s_{n,k}^{\ast}(2, 1),
\end{align*}
where $s_{n,k}^{\ast}(2, 1)$ is defined as in 
Proposition \ref{prop_first_spcase_result}, and the partition function 
$Q(n) := s_e(n) - s_o(n)$ where $s_e(n)$ and $s_o(n)$ 
denote the numbers of partitions of $n$ into an even (respectively odd) number of parts. 
We can also similarly express the relations in the previous equations for any 
special arithmetic function $f$ in the form of 
\begin{align*} 
     \sum_{2d-1|n} f(d) & = \sum_{k=0}^n \sum_{j=1}^k (-1)^{k-1} q(n-k) s_{k,j}^{\ast}(2, 1) f(j), 
\end{align*} 
where the standard partition function $q(n)$ denotes the number partitions of $n$ into (distinct) odd parts. 
Moreover, since we have a direct factorization of the Lambert series generating function 
for the \emph{sum of squares function} as in the appendix, we may write 
\begin{align*} 
\sum_{k=1}^n r_2(k) Q(n-k) & = \sum_{k=1}^n 4 (-1)^{k+1} s_{n,k}^{\ast}(2, 1), 
\end{align*} 
using the same notation as above. Similarly, we can invert to expand $r_2(n)$ as the multiple sum 
\[
r_2(n) = \sum_{k=0}^n \sum_{j=1}^k 4 q(n-k) (-1)^{j+1} s_{k,j}^{\ast}(2, 1), 
\]
for all $n \geq 1$. 
\end{example} 

\begin{definition}
\label{def_LfAlphaBetaq_LfAlphaBetaqFBarExp_GenLGFExps_v1}
For fixed $\alpha, \beta \in \mathbb{Z}$ such that $\alpha \geq 1$ and $0 \leq \beta < \alpha$, and an 
arbitrary sequence $\{f(n)\}_{n \geq 1}$, 
we consider generalized (non-factorized) Lambert series expansions of the following form: 
\begin{align}
\label{eqn_GenLambertSeries_LaAlphaBetaq_def}
L_f(\alpha, \beta; q) & := \sum_{n \geq 1} \frac{f(n) q^{\alpha n-\beta}}{1-q^{\alpha n-\beta}} = 
     \sum_{m \geq 1} \left(\sum_{\substack{\alpha d-\beta | m}} f(d)
     \right) q^m, 
     |q^{\alpha}| < 1. 
\end{align} 
We also consider factorizations of the form 
\begin{align} 
\label{eqn_GenFactThmExp_def_v1} 
L_f(\alpha, \beta; q) & = \frac{1}{C(q)} \times \sum_{n \geq 1} \left(
     \sum_{k=1}^n s_{n,k}(\alpha, \beta) \bar{f}_k \right) q^n, 
\end{align} 
where $C(0) \neq 0$, and the intermediate coefficients 
$\bar{f}_n$ depend only on the $s_{n,k}(\alpha, \beta)$ and on the 
choice of the input arithmetic function $f(n)$. 
\end{definition}

\begin{prop}
\label{prop_GenLGF_p1} 
For $|q|<1$, $0 \leq \beta<\alpha$,
\begin{align*}
\sum_{n=1}^{\infty} f(n) \frac{q^{\alpha n-\beta}}{1-q^{\alpha n-\beta}} = 
     \frac{1}{(q^{\alpha-\beta};q^\alpha)_\infty} \times \sum_{n=1}^{\infty} \left( 
     \sum_{k=1}^n (s_o(\alpha, \beta; n,k)-s_e(\alpha, \beta; n,k)) f(k)\right)  q^n,
\end{align*}
where $s_o(\alpha, \beta; n,k)$ and $s_e(\alpha, \beta; n,k)$ 
denote the number of $(\alpha k-\beta)$'s in all partitions of $n$ into an odd 
(respectively even) number of distinct parts of the form $\alpha k-\beta$.
\end{prop}
\begin{proof}[Proof of Proposition \ref{prop_GenLGF_p1}]
The proof follows from \cite[eq. 2.1]{MERCA-LSFACTTHM}, replacing $x_k$ by $q^{\alpha k-\beta}$.
\end{proof}

\begin{prop}
\label{prop_GenLGF_p2} 
For $|q|<1$, $0\leq \beta<\alpha$,
\begin{align*}
\sum_{n=1}^{\infty} f(n) \frac{q^{\alpha n-\beta}}{1-q^{\alpha n-\beta}} = 
     (q^{\alpha-\beta};q^\alpha)_\infty \times \sum_{n=1}^{\infty} \left( 
     \sum_{k=1}^n s(\alpha, \beta; n,k) f(k)\right)  q^n,
\end{align*}
where $s(\alpha, \beta; n,k)$ denotes the number of 
$(\alpha k-\beta)$'s in all partitions of $n$ into parts of the form $\alpha k-\beta$.
\end{prop}
\begin{proof}[Proof of Proposition \ref{prop_GenLGF_p2}]
We take into account the fact that
$$\frac{q^{\alpha n-\beta}}{1-q^{\alpha n-\beta}} \times \frac{1}{(q^{\alpha-\beta};q^\alpha)_\infty},$$
is the generating function for the number of $(\alpha k-\beta)$'s in all partitions of $n$ into 
parts of the form $\alpha k-\beta$. This generating function interpretation implies our result. 
\end{proof}

\SubsubsectionGTThesisFormatted{The generalized factorization matrix entries} 
\begin{theorem} 
\label{theorem1_GenFormula_for_snk} 
For fixed $\alpha, \beta, \gamma, \delta \in \mathbb{Z}$ such that 
$\alpha, \gamma \geq 1$, $1 \leq \beta < \alpha$, and $1 \leq \delta < \gamma$, the 
pair $(\mathcal{C}(q), s_{n,k}[\mathcal{C}](\alpha, \beta, \gamma, \delta))$ in the 
generalized Lambert series factorization expanded by 
\begin{align*} 
\tag{i}
L_f(\alpha, \beta, \gamma, \delta; q) & := 
     \sum_{n \geq 1} \frac{f(n) q^{\alpha n + \beta}}{1-q^{\gamma n+\delta}} = 
     \frac{1}{\mathcal{C}(q)} \times 
     \sum_{n \geq 1} \sum_{k=1}^n s_{n,k}[\mathcal{C}](\alpha, \beta, \gamma, \delta) f(k) 
     q^n, 
\end{align*} 
satisfies 
\begin{align*} 
s_{n,k}[\mathcal{C}](\alpha, \beta, \gamma, \delta) & = 
     [q^n]\left(\mathcal{C}(q) \times \frac{q^{\alpha n + \beta}}{1-q^{\gamma n+\delta}}\right). 
\end{align*} 
\end{theorem}
\begin{proof}[Proof of Theorem \ref{theorem1_GenFormula_for_snk}]
We begin by rewriting (i) in the form of 
\begin{align*} 
\mathcal{C}(q) \times \sum_{k \geq 1} \frac{f(k) q^{\alpha k+\beta}}{1-q^{\gamma k+\delta}} & = 
     \sum_{k \geq 1} \left(\sum_{n \geq 1} s_{n,k}[\mathcal{C}](\alpha, \beta, \gamma, \delta) q^n\right) f(k). 
\end{align*} 
Then if we equate the coefficients of $a_k$ in the previous equation, we see that 
\begin{align*} 
\mathcal{C}(q) \times \frac{q^{\alpha k+\beta}}{1-q^{\gamma k+\delta}} & = 
     \sum_{n \geq 1} s_{n,k}[\mathcal{C}](\alpha, \beta, \gamma, \delta) q^n, 
\end{align*} 
which implies the stated result. 
\end{proof} 

\begin{cor} 
\label{cor_ConsequenceOfThm1_ConnectionBetweenOrdFactThms} 
Let $\alpha \geq 1$ and $0 \leq \beta < \alpha$ be integers and suppose that $\delta \in \mathbb{Z}$. 
Suppose that 
\begin{align*} 
\sum_{n \geq 1} \frac{f(n) q^{n}}{1-q^{n}} = 
     \frac{1}{\mathcal{C}(q)} \times \sum_{n \geq 0} 
     \sum_{k=1}^n s_{n,k}[\mathcal{C}] f(k) q^n, 
\end{align*} 
and that 
\begin{align*} 
\sum_{n \geq 1} \frac{f(n) q^{\alpha n-\beta+\delta}}{1-q^{\alpha n-\beta}} = 
     \frac{1}{\mathcal{C}(q)} \times \sum_{n \geq 0} \sum_{k=1}^n s_{n,k}[\mathcal{C}](\alpha, \beta; \delta) 
     f(k) q^n. 
\end{align*} 
Then we have that 
\begin{align*} 
s_{n,k}[\mathcal{C}](\alpha, \beta; \delta) & = s_{n-\delta,\alpha k-\beta}[\mathcal{C}]. 
\end{align*} 
\end{cor} 
\begin{proof}[Proof of Corollary \ref{cor_ConsequenceOfThm1_ConnectionBetweenOrdFactThms}]
The proof follows from the second statement in the last theorem, which 
provides a generating function for $s_{n,k}[\mathcal{C}]$ for all $n, k \geq 1$. 
\end{proof}

Given the difficulty in proving that which find numerically 
holds for all $n \geq k \geq 1$, we will cite some of the interesting conjectures for properties 
satisfied by the factorization matrices for a particular degenerate LGF case. 
Generalizations of this property are not as immediately apparent, though do seem to degrade nicely 
for $\alpha \geq 3$ based on preliminary computational inspection. 

\begin{conjecture}[Invertible matrix factorizations of degenerate LGF cases]
For $1 \leq k \leq n$, the factorization matrix inverse in the expansion of the degenerate LGF 
$L_f(1, 0, 2, 1; q)$ satisfies the following relation:
\begin{align} 
\label{eqn_DegenCasesOfThm_example_v1}
s_{n,k}^{-1}(1, 0, 2, 1) & = 
     p(n-k) - \sum_{i=1}^n p\left(\frac{n-i}{2i+1}-k\right) \Iverson{n \equiv i \bmod 2i+1} \\ 
\notag 
     & \phantom{=p(n-k)\ } + 
     \sum_{m=2}^n \sum_{i=1}^n p\left(\frac{n-p(m+1)i-p(m-1)}{p(m+1)(2i+1)} - k\right) \times \\ 
\notag
     & \phantom{=p(n-k)+\sum\sum\ \ } \times 
     \Iverson{n \equiv p(m+1)i+p(m-1) \bmod p(m+1)(2i+1)}. 
\end{align} 
\end{conjecture} 

\begin{theorem}[Another generalized factorization theorem] 
\label{theorem2_AnotherGenFactThm}
Define the factorization pair $(\mathcal{C}(q), s_{n,k}[\mathcal{C}](\gamma))$ where 
$\mathcal{C}(0) \neq 0$ by the requirement that 
\begin{align*}
     s_{n,k}^{-1}[\mathcal{C}](\gamma) & := \sum_{d|n} [q^{d-k}] \frac{1}{C(q)} \times \gamma\left(\frac{n}{d}\right), 
\end{align*} 
for some fixed arithmetic functions $\gamma(n)$ where $\widetilde{\gamma}(n) := \sum_{d|n} \gamma(d)$. 
We have that the sequence of 
$\bar{f}_n$ in the notation of \eqref{eqn_GenFactThmExp_def_v1} 
is given by the following formula for $n \geq 1$: 
\begin{align*} 
\bar{f}_n & = 
     \sum_{\substack{d|n \\ d \equiv \beta \bmod \alpha}} f\left({\frac{d-\beta}{\alpha}}\right) 
     \widetilde{\gamma}\left(\frac{n}{d}\right). 
\end{align*} 
\end{theorem} 
\begin{proof}[Proof of Theorem \ref{theorem2_AnotherGenFactThm}] 
We begin by noticing that 
\begin{align*} 
\bar{f}_n & = \sum_{k=1}^n s_{n,k}^{(-1)} \times [q^k] \left(\sum_{d=1}^k 
     \frac{f(d) q^{\alpha d+\beta}}{1-q^{\alpha d+\beta}}\right). 
\end{align*} 
Then if we let the series coefficients $c_n := [q^n] \mathcal{C}(q)^{-1}$ and set 
\begin{align*} 
t_{k,d}(\alpha, \beta) & := [q^k]\left(\mathcal{C}(q) \times 
     \frac{q^{\alpha d+\beta}}{1-q^{\alpha d+\beta}}\right), 
\end{align*} 
where $t_{i,d}(\alpha, \beta) = 0$ whenever $i < \alpha d+\beta$, 
we have that for each $1 \leq d \leq n$ 
\begin{align*}
[a_d] \bar{f}_n & = \sum_{k=1}^n s_{n,k}^{-1}[\mathcal{C}](\gamma) t_{k,d}(\alpha, \beta) \\ 
     & = 
     \sum_{k=d}^n \left(\sum_{r|n} c_{r-k} \gamma\left(\frac{n}{r}\right)\right) t_{k,d}(\alpha, \beta) \\ 
     & = 
     \sum_{r|n} \left(\sum_{i=d}^r c_{r-i} t_{i,d}(\alpha, \beta)\right) \gamma\left(\frac{n}{r}\right) \\ 
     & = 
     \sum_{r|n} \left(\sum_{i=0}^r c_{r-i} t_{i,d}(\alpha, \beta)\right) \gamma\left(\frac{n}{r}\right), 
\end{align*} 
The inner sums in the previous equations are generated by 
\begin{align*} 
[q^r] \frac{1}{\cancel{\mathcal{C}(q)}} \cdot 
     \frac{q^{\alpha d+\beta}}{1-q^{\alpha d+\beta}} \cancel{\mathcal{C}(q)} & = 
     \Iverson{\alpha d+\beta | r}. 
\end{align*} 
Then we have that for integers $d \geq 1$ 
\begin{align*} 
[f(d)] \bar{f}_n & = \sum_{\substack{r|n \\ \alpha d + \beta | r}} \gamma\left(\frac{n}{r}\right), 
\end{align*} 
and that 
\begin{align*} 
\bar{f}_n & = \sum_{\substack{d|n \\ d \equiv \beta \bmod \alpha}} f\left({\frac{d-\beta}{\alpha}}\right)
     \sum_{\substack{r|n \\ d | r}} \gamma\left(\frac{n}{r}\right) \\ 
     & = 
\sum_{\substack{d|n \\ d \equiv \beta \bmod \alpha}} f\left({\frac{d-\beta}{\alpha}}\right) 
     \sum_{r | \frac{n}{d}} \gamma\left(\frac{n}{dr}\right). 
     \qedhere
\end{align*} 
\end{proof} 

\begin{example}[A new identity for the sum of squares function]
We have that 
\begin{align*} 
\sum_{\substack{d|n \\ d\text{ odd}}} & (-1)^{(d+1)/2} \left(r_2\left(\frac{n}{d}\right) 
     -4 d\left(\frac{n}{2d}\right) \Iverson{\frac{n}{d}\text{ even}}\right) \\ 
     & = 
     \sum_{k=1}^n \sum_{d|n} p(d-k) (-1)^{n/d+1} \times [q^k] (q; q)_{\infty} \vartheta_3(q)^2, 
\end{align*} 
where $d(n) \equiv \sigma_0(n)$ denotes the divisor function and where the 
powers of the Jacobi theta function, $\vartheta_3(q) = 1 + 2 \sum_{n \geq 1} q^{n^2}$, 
generate the series over the sums of squares functions, 
$r_k(n) = [q^n] \vartheta_3(q)^k$. 
\end{example}

\SubsubsectionGTThesisFormatted{Factorization theorems for classical LGFs over Dirichlet convolutions} 
\label{Section_LSFactThms_CvlOfTwoFns_and_Apps}

\begin{definition}
Given any two arithmetic functions $f$ and $g$ we define their \emph{Dirichlet convolution}, 
denoted by $h = f \ast g$, to be the function \cite[\S 2.6]{APOSTOLANUMT}
\begin{align*} 
(f \ast g)(n) & := \sum_{d|n} f(d) g\left(\frac{n}{d}\right), n \geq 1.
\end{align*} 
The function $h^{-1}(n)$ is called the \emph{Dirichlet inverse} of $h(n)$ 
(or inverse of $h$ with respect to Dirichlet convolution) if 
$h^{-1} \ast h = h \ast h^{-1} = \varepsilon$, where $\varepsilon(n) = \delta_{n,1}$ 
is the multiplicative identity with respect to Dirichlet convolution. 
An arithmetic function $h$ has a Dirichlet inverse iff $h(1) \neq 0$. 
If $h^{-1}$ exists, then it is unique and can be computed recursively by the formula 
\[
h^{-1}(n) = \begin{cases} 
     \frac{1}{h(1)}, & n = 1; \\ 
     -\frac{1}{h(1)} \times 
     \sum\limits_{\substack{d|n \\ d>1}} h(d) h^{-1}\left(\frac{n}{d}\right), & n \geq 2. 
     \end{cases} 
\]
\end{definition}

\begin{table}[h!]

\caption[Symbolic computations of the Dirichlet inverse of a function]{
	 The first few cases of a Dirichlet invertible $h$ evaluated symbolically. 
         Recall that an arithmetic function $h$ is Dirichlet invertible if and only if 
         $h(1) \neq 0$. When the Dirichlet inverse of the function $h$ exists, it is unique. }

\begin{center}
\begin{tabular}{|c|l|c|l|c|l|} \hline 
$\mathbf{n}$ & $h^{-1}(n)$ & $\mathbf{n}$ & $h^{-1}(n)$ & $\mathbf{n}$ & $h^{-1}(n)$ \\ \hline 
\textbf{1} & $\displaystyle\frac{1}{h(1)}$ & \textbf{4} & $-\displaystyle\frac{h(1) h(4)-h(2)^2}{h(1)^3}$ & 
\textbf{7} & $-\displaystyle\frac{h(7)}{h(1)^2}$ \\ 
\textbf{2} & $-\displaystyle\frac{h(2)}{h(1)^2}$ & \textbf{5} & $-\displaystyle\frac{h(5)}{h(1)^2}$ & 
\textbf{8} & $-\displaystyle\frac{h(2)^3-2 h(1) h(4) h(2)+h(1)^2 h(8)}{h(1)^4}$ \\ 
\textbf{3} & $-\displaystyle\frac{h(3)}{h(1)^2}$ & \textbf{6} & $-\displaystyle\frac{h(1) h(6)-2 h(2) h(3)}{h(1)^3}$ & 
\textbf{9} & $-\displaystyle\frac{h(1) h(9)-h(3)^2}{h(1)^3}$ \\ \hline 
\end{tabular} 
\end{center} 

\end{table}

\begin{prop}[One possible factorization] 
\label{prop_OnePossibleLSFact} 
Let $f$ and $g$ denote non-identically-zero arithmetic functions. 
Suppose that we have an ordinary Lambert series factorization for any prescribed 
arithmetic function $f(n)$ of the form 
\begin{align*} 
\tag{i}
\sum_{n \geq 1} \frac{f(n) q^n}{1-q^n} & = \frac{1}{\mathcal{C}(q)} \times 
     \sum_{n \geq 1} \left(\sum_{k=1}^n 
     s_{n,k}[\mathcal{C}] f(k) \right) q^n, 
\end{align*} 
A factorization theorem for the Lambert series over the 
Dirichlet convolution $h = f \ast g$ is expanded as follows: 
\begin{align*} 
\tag{ii} 
\sum_{n \geq 1} \frac{(f \ast g)(n) q^n}{1-q^n} & = \frac{1}{\mathcal{C}(q)} \times 
     \sum_{n \geq 1} \left(\sum_{k=1}^n 
     \widetilde{s}_{n,k}[\mathcal{C}](g) f(k) \right) q^n. 
\end{align*} 
The matrix coefficients in the previous equation satisfy 
\begin{equation*} 
\tag{iii}
\widetilde{s}_{n,k}[\mathcal{C}](g) = \sum_{j=1}^n s_{n,kj}[\mathcal{C}] g(j). 
\end{equation*} 
\end{prop} 
\begin{proof}[Proof of Proposition \ref{prop_OnePossibleLSFact}] 
It is apparent by the expansions on the left-hand side of (ii) that there is some sequence of 
$\widetilde{s}_{n,k}[\mathcal{C}](g)$ depending on the function $g$ that satisfies the factorization of the 
form in (i) when $a_n \mapsto (f \ast g)(n)$. 
For a fixed $k \geq 1$, we begin by evaluating the coefficients of $f(k)$ on the 
right-hand side of (ii) as follows: 
\begin{align*} 
[f(k)] \left(\sum_{n \geq 1} \sum_{i=1}^n s_{n,i}[\mathcal{C}] (f \ast g)(i) q^n\right) 
     & = \sum_{n \geq 1} \left(\sum_{i=1}^n [f(k)] s_{n,i}[\mathcal{C}] \sum_{d|i} f(d) 
     g\left(\frac{i}{d}\right)\right) q^n \\ 
     & = \sum_{n \geq 1} \left(\sum_{i=1}^n s_{n,i}[\mathcal{C}] g\left(\frac{i}{d}\right) 
     \Iverson{k|i}\right) q^n \\ 
     & = \sum_{n \geq 1} \left(\sum_{j=1}^n s_{n,kj}[\mathcal{C}] g(j)\right) q^n. 
\end{align*} 
Thus we have that the formula for $\widetilde{s}_{n,k}[\mathcal{C}](g)$ given in equation 
(iii) holds. 
\end{proof} 

\begin{table}[h!]

\caption[Symbolic inverse matrices for LGFs of a Dirichlet convolution]{}
\label{table_InvMatrixExtries_Discussion_v1}

\begin{equation*}
\arraycolsep=1.4pt\def\arraystretch{2.2}
\begin{array}{|c|llllll|} \hline 
\mathbf{n \setminus k} & \mathbf{1} & \mathbf{2} & \mathbf{3} & \mathbf{4} & \mathbf{5} & \mathbf{6} \\ \hline 
\mathbf{1} & \frac{1}{g(1)} & 0 & 0 & 0 & 0 & 0 \\
\mathbf{2} & -\frac{g(2)}{g(1)^2} & \frac{1}{g(1)} & 0 & 0 & 0 & 0 \\
\mathbf{3} & \frac{g(1)^5-g(1)^4 g(3)}{g(1)^6} & \frac{1}{g(1)} & \frac{1}{g(1)} & 0 & 0 & 0 \\
\mathbf{4} & \frac{2 g(1)^5-g(4) g(1)^4+g(2)^2 g(1)^3}{g(1)^6} & \frac{g(1)^5-g(1)^4 g(2)}{g(1)^6} & \frac{1}{g(1)} & \frac{1}{g(1)} & 0 & 0 \\
\mathbf{5} & \frac{4 g(1)^5-g(1)^4 g(5)}{g(1)^6} & \frac{3}{g(1)} & \frac{2}{g(1)} & \frac{1}{g(1)} & \frac{1}{g(1)} & 0 \\
\mathbf{6} & \frac{5 g(1)^5-g(2) g(1)^4-g(6) g(1)^4+2 g(2) g(3) g(1)^3}{g(1)^6} & \frac{3 g(1)^5-g(2) g(1)^4-g(3) g(1)^4}{g(1)^6} & \frac{2 g(1)^5-g(1)^4
   g(2)}{g(1)^6} & \frac{2}{g(1)} & \frac{1}{g(1)} & \frac{1}{g(1)} \\ \hline 
\end{array}
\end{equation*} 

\end{table}

\begin{table}[h!]

\caption[Inverse matrices for LGFs of a Dirichlet convolution (special case)]{}
\label{table_InvMatrixExtries_Discussion_v2}

\tiny
\begin{equation*}
\arraycolsep=1.6pt\def\arraystretch{2.2}
\begin{array}{|c|l|l|l|l|} \hline 
 & \mathbf{1} & \mathbf{2} & \mathbf{3} & \mathbf{4} \\ \hline 
\mathbf{1} & 1 & 0 & 0 & 0 \\
\mathbf{2} & -g(2) & 1 & 0 & 0 \\
\mathbf{3} & 1-g(3) & 1 & 1 & 0 \\
\mathbf{4} & g(2)^2-g(4)+2 & 1-g(2) & 1 & 1 \\
\mathbf{5} & 4-g(5) & 3 & 2 & 1 \\
\mathbf{6} & 2 g(3) g(2)-g(2)-g(6)+5 & -g(2)-g(3)+3 & 2-g(2) & 2 \\
\mathbf{7} & 10-g(7) & 7 & 5 & 3 \\
\mathbf{8} & -g(2)^3+2 g(4) g(2)-2 g(2)-g(8)+12 & g(2)^2-g(2)-g(4)+9 & 6-g(2) & 4-g(2) \\
\mathbf{9} & g(3)^2-g(3)-g(9)+20 & 14-g(3) & 10-g(3) & 7 \\
\mathbf{10} & 2 g(5) g(2)-4 g(2)-g(10)+25 & -3 g(2)-g(5)+18 & 13-2 g(2) & 10-g(2) \\ \hline 
\end{array}
\end{equation*} 

\end{table}

\begin{discussion}[Formulating some intuition for the forms of the resulting inverse matrices]
The next examples suggest insight that is gathered by computation and experimental mathematics on the 
inverse matrix entry dataset with \emph{Mathematica}. 
Consider factorizing the Lambert series generating function of $f \ast g$: 
\[
\sum_{n \geq 1} \frac{(f \ast g)(n) q^n}{1-q^n} = \frac{1}{(q; q)_{\infty}} \times \sum_{n \geq 1} 
     \left(\sum_{k=1}^{n} \widetilde{s}_{n,k}(g) f(k)\right) q^n. 
\]
Then we can prove that 
$\widetilde{s}_{n,k}(g) = \sum_{j=1}^{n} s_{n,kj} g(j)$. 
The pattern that characterizes the inverse matrix entries is less obvious to see immediately. 
Consider the listing of the first several entries of the inverse matrix sequence, 
$\widetilde{s}_{n,k}^{-1}(g)$, given in 
Table \ref{table_InvMatrixExtries_Discussion_v1}. 
Looking closely at the entries in the last table 
suggests that the terms are partition-scaled multiple ($m$-fold) Dirichlet convolutions
related to the indices $(n, k)$.
Upon setting $g(1) \mapsto 1$, the special case makes clear the relation of these matrix inverses, 
the partition function $p(n)$, and the Dirichlet inverse of the function $g$ 
we observe in Table \ref{table_InvMatrixExtries_Discussion_v2}.
\end{discussion}

\begin{definition}
For a fixed arithmetic function $g$ with $g(1) := 1$, 
let the functions $\ds_{j,g}(n)$ be defined recursively for natural numbers 
$j \geq 1$ as 
\begin{align*} 
\ds_{j,g}(n) & := 
     \begin{cases} 
     g_{\pm}(n), & \text{ if $j = 1$; } \\ 
     \sum\limits_{\substack{d|n \\ d>1}} g(d) \ds_{j-1,g}\left(\frac{n}{d}\right), & \text{ if $j > 1$, } 
     \end{cases} 
\end{align*} 
where $g_{\pm}(n) := g(n) \Iverson{n > 1} - \delta_{n,1}$. 
Let the notation for the \emph{$k$-shifted partition function} be defined as 
$p_k(n) := p(n-k)$ for $k \geq 1$. If we set the function $\widetilde{\ds}_{j,g}(n)$ 
to be the $j$-fold convolution of $g$ with itself, i.e., that 
$$\widetilde{\ds}_{j,g}(n) = \underset{\text{$j$ times}}{\underbrace{\left(g_{\pm} \ast g \ast \cdots \ast g\right)}}(n),$$ 
then we can prove easily by induction that for all $m, n \geq 1$ we have the expansion 
$$\ds_{m,g}(n) = \sum_{i=0}^{m-1} \binom{m-1}{i} (-1)^{m-1-i} \widetilde{\ds}_{i+1,g}(n).$$ 
We then define the following notation for the sums of the variant convolution functions for 
use in the theorem below for any $n \geq 1$: 
\begin{align*} 
D_{n,g}(n) & := \sum_{j=1}^n \ds_{2j,g}(n) = \sum_{m=1}^{\lfloor \frac{n}{2} \rfloor} \sum_{i=0}^{2m-1} 
     \binom{2m-1}{i} (-1)^{i+1} \widetilde{\ds}_{i+1,g}(n). 
\end{align*}
\end{definition} 

\begin{theorem}[Inverse matrix sequences]
\label{claim_snk_inverses}
Suppose that $g(1) \neq 0$, i.e., that $g$ is Dirichlet invertible. 
When $C(q) = (q; q)_{\infty}$ we write $s_{n,k}^{-1}[\mathcal{C}](g) \equiv s_{n,k}^{-1}(g)$. 
For fixed integers $k \geq 1$, the inverse matrix sequences from Proposition \ref{prop_OnePossibleLSFact} satisfy 
\begin{align*} 
\sum_{d|n} s_{n,k}^{-1}(g) & = p_k(n) + \sum_{j=1}^{\Omega(n)} (p_k \ast \ds_{2j,g})(n) \\ 
     & = p_k(n) + (p_k \ast D_{n,g})(n), 
\end{align*} 
Equivalently, we have that 
\begin{align*} 
s_{n,k}^{(-1)}(g) & = (p_k \ast \mu)(n) + \sum_{j=1}^{\Omega(n)} (p_k \ast \ds_{2j,g} \ast \mu)(n) \\ 
     & = (p_k \ast \mu)(n) + (p_k \ast D_{n,g} \ast \mu)(n). 
\end{align*} 
\end{theorem} 

\begin{proof}[Proof of Theorem \ref{claim_snk_inverses}]  
Let the proposed inverse sequence function be defined in the following notation: 
\begin{align*} 
\widehat{s}_{n,k}^{(-1)}(g) & := (p_k \ast \mu)(n) + (p_k \ast D_{n,g} \ast \mu)(n). 
\end{align*} 
We begin as in the proof of Theorem {3.2} in \cite{MERCA-SCHMIDT-LSFACTTHM} to consider the ordinary, 
non-convolved Lambert series. 
More precisely, by the expansion in (i) of the proposition we must show that 
\begin{align*} 
\sum_{d|n} \widehat{s}_{d,k}^{(-1)} & := p_k(n) + (p_k \ast D_{n,g})(n) \\ 
     & = \sum_{m=0}^n \sum_{j=1}^{n-m} \widetilde{s}_{n-m,j}(g) \widehat{s}_{j,k}^{(-1)} p(m). 
\end{align*} 
For integers $n, k, i \geq 1$ with $k, i \leq n$, let the coefficient functions, $\rho_{n,k}^{(i)}$ 
be defined as 
\[
\rho_{n,k}^{(i)} := \sum_{j=1}^n s_{n,ij} \widetilde{s}_{j,k}^{(-1)}. 
\] 
Then for any fixed arithmetic function $h$,
by considering the related expansions of the 
factorizations in (ii) of the proposition for $\widehat{s}_k \ast g$, we can prove that 
\begin{align*} 
\tag{i}
t_{n,k}(h) & := \sum_{j=1}^n s_{n,j} (\widetilde{s}_{n,k}^{(-1)} \ast h)(j) 
      = \sum_{i=1}^n \rho_{n,k}^{(i)} h(i). 
\end{align*} 
It remains to show that 
\[
\tag{ii} 
\sum_{m=0}^n t_{n-m,k}(h) p(m) = (p_k \ast h)(n). 
\] 
Since we can expand the left-hand side of the previous sum as 
\begin{align*} 
\sum_{m=0}^n \sum_{i=1}^{n-m} \rho_{n-m,k}^{(i)} h(i) p(m) & = 
     \sum_{i=1}^n h(i) \underset{ := u_{n,k}^{(i)}}{\underbrace{\left( 
     \sum_{m=0}^n \rho_{n-m,k}^{(i)} p(m)\right)}}, 
\end{align*} 
to complete the proof of (ii) we need to prove a subclaim 
that (I) $u_{n,k}^{(i)} = 0$ if $i \not{\mid} n$; and (II) 
if $i | n$ then $u_{n,k}^{(i)} = p\left(\frac{n}{i}-k\right)$. 

\noindent
\textit{Proof of Subclaim}: 
For $i := 1$, this is clearly the case since $\rho_{n-m,k}^{(i)} = \Iverson{n-m=k}$. 
For subsequent cases of $i \geq 2$, it is apparent that 
$$\rho_{n,k}^{(i)} = \rho_{n-(k-1)i,1}$$ much as in the cases of the tables for the inverse sequences, 
$s_{n,k}^{(-1)} = (p_k \ast \mu)(n)$. Finally, we claim that generating functions for the 
sequences of $u_{n,k}^{(i)}$ for each $i \geq 2$ are expanded in the form of
\[
\sum_{n \geq 0} u_{n,k}^{(i)} \cdot q^n = 
     \prod_{j=1}^{i-1} (q^j; q^i)_{\infty} \times \frac{q^{ik}}{(q; q)_{\infty}} = 
     q^{ik} \times \sum_{n \geq 0} p(n) q^{in}, 
\] 
which we see by comparing coefficients on the right-hand side of the 
previous equation implies our claim. 

\noindent
\textit{Completing the Proof of the Inverse Formula}: 
What we have shown by proving (ii) above is an inverse formula for an ordinary 
Lambert series factorization over the sequence of 
$a_j := (\widetilde{s}_{n,k}^{(-1)} \ast g)(j)$. 
In particular, by M\"obius inversion (ii) shows that we have 
\begin{align*} 
(f \ast g)(n) = (s_{n,k}^{(-1)} \ast g)(n) & \iff 
(f - s_{n,k}^{(-1)}) \ast g \equiv 0 \\ 
     & \implies 
     f_n = s_{n,k}^{(-1)}, \text{ when $g \not\equiv 0$. } 
\end{align*} 
More to the point, when we define $f_n := \widetilde{s}_{n,k}^{(-1)}(g)$ 
where by convenience and experimental suggestion we let 
$$\widetilde{s}_{n,k}^{(-1)}(g) = s_{n,k}^{(-1)} \ast t_{n,k}^{(-1)}(g),$$ 
for some convolution-wise factorization of this inverse sequence, we can now prove the 
exact formula for the inverse sequence claimed in the theorem statement. 
In the forward direction, we suppose that 
\[t_{n,k}^{(-1)}(g) = D_{n,g}(n) + \varepsilon(n),\] 
where $\varepsilon(n) = \delta_{n,1}$ 
is the multiplicative identity, and then see from the formulas for $D_{n,g}(n)$ discussed 
before the claim that $g \ast (D_{n,g}+\varepsilon) = \varepsilon$, which proves that 
our inverse formula is correct in this case. 
Conversely, if we require that 
\[
s_{n,k}^{(-1)} \ast t_{n,k}^{(-1)}(g) \ast g = s_{n,k}^{(-1)} 
\] 
for all $n$ and choices of the function $g$, we must have that 
$t_{n,k}^{(-1)} \ast g = \varepsilon$, and so we see that 
$t_{n,k}^{(-1)} = D_{n,g} + \varepsilon$ as required. 
That is to say, we have proved our result using the implicit statement that 
$t \ast g = \varepsilon$ if and only if $t = D_{n,g} + \varepsilon$, i.e., 
that $t = D_{n,g} + \varepsilon$ is the unique function such that 
$t \ast g = \varepsilon$ for all $n \geq 1$, 
a result which we do not prove here and only mention for the sake of brevity. 
\end{proof}

\begin{cor}[Formulas for the Dirichlet inverse] 
\label{cor_FormulasForTheDirichletInverse_v1}
For any arithmetic function $f$ defined such that $f(1) = 1$, we have 
a formula for its Dirichlet inverse function given by 
\begin{align*}
f^{-1}(n) & = \sum_{k=1}^n \left((p_k \ast \mu)(n) + (p_k \ast D_{n,f} \ast \mu)(n)\right) \times 
     [q^{k-1}] \frac{(q; q)_{\infty}}{1-q}. 
\end{align*} 
\end{cor} 
\begin{proof}[Proof of Corollary \ref{cor_FormulasForTheDirichletInverse_v1}]
The proof follows from Theorem \ref{claim_snk_inverses} applied in the form of 
Proposition \ref{prop_OnePossibleLSFact}. In particular, since 
$f^{-1} \ast f = \delta_{n,1}$ by definition, the right-hand side of our Lambert series expansion 
over the convolved function $a_n := f^{-1} \ast f$ is given by $q (1-q)^{-1}$. 
\end{proof} 

\begin{remark}
We also note that given any sequence $b(n)$, we can generate $b(n)$ by the 
Lambert series over $b \ast \mu$. This implies that we have recurrence relations for any 
arithmetic function $b$ defined such that $b(n) = 0$ for all $n < 0$ expanded in the 
following two forms where $s_{n,k} := [q^n] (q; q)_{\infty} q^k (1-q^k)^{-1}$: 
\begin{align*} 
b(n) & = \sum_{k=1}^n \left(p_k \ast \mu+p_k \ast D_{n,\mu} \ast \mu\right)(n) \left( 
     b(k) + \sum_{s = \pm 1} \sum_{j=1}^k (-1)^j b\left(k-\frac{j(3j+s)}{2}\right)\right), \\ 
b(n) & = \sum_{j=1}^n \sum_{k=1}^j \left(\sum_{i=1}^{\left\lfloor \frac{j}{k} \right\rfloor} 
     s_{j,ki} \cdot \mu(i)\right) b(k) p(n-j). 
\end{align*}
\end{remark}

\begin{cor}[Convolution formulas for arbitrary arithmetic functions] 
Suppose that we have any two arithmetic functions $f$ and $h$ and we seek the form 
of a third function $g$ satisfying $f \ast g = h \ast \mu$ for all $n \geq 1$. Then we have a formula for the 
function $g$ expanded in the form 
\begin{align*}
g(n) & = \sum_{k=1}^n \left((p_k \ast \mu)(n) + (p_k \ast D_{n,f} \ast \mu)(n)\right) \times \\ 
     & \phantom{=\sum\ } \times 
     \left(h(k) + \sum_{s = \pm 1} \sum_{j=1}^{\left\lfloor \frac{\sqrt{24k+1}-s}{6} \right\rfloor} 
     (-1)^j h\left(k-\frac{j(3j+s)}{2}\right)\right). 
\end{align*} 
\end{cor} 
\begin{proof} 
This result is an immediate consequence of Proposition \ref{prop_OnePossibleLSFact} and the formula 
for the inverse sequences defined by Theorem \ref{claim_snk_inverses}. 
\end{proof} 

\SubsectionGTThesisFormatted{Hadamard products and derivatives of LGFs} 
\label{subSection_LGFFactThms_HadamardProductsAndDerivativeExps} 

\begin{definition}[Hadamard products for Lambert series generating functions] 
\label{def_HP_for_LSGFs}
For any fixed arithmetic functions $f$ and $g$, we define the Hadamard product of the 
two Lambert series over $f$ and $g$ to be the auxiliary Lambert series generating function 
over the composite function $a_{\fg}(n)$ whose coefficients are given by 
\[
\sum_{d|n} a_{\fg}(d) = [q^n] \sum_{m \geq 1} \frac{a_{\fg}(m) q^m}{1-q^m} := 
     \underset{:=\fg(n)}{\underbrace{\left(\sum_{d|n} f_d\right) \times \left(\sum_{d|n} g_d\right)}}. 
\] 
By M\"obius inversion we have that 
\[
a_{\fg}(n) = \sum_{d|n} \fg(d) \mu\left(\frac{n}{d}\right) = (\fg \ast \mu)(n). 
\] 
We also define the following function to expand divisor sums over arithmetic functions as 
ordinary sums for any integers $1 \leq k \leq n$: 
\[
T_{\Div}(n, k) := 
     \begin{cases} 
     1, & \text{ if $k|n$; } \\ 
     0, & \text{ otherwise. }
     \end{cases} 
\]
\end{definition} 

The results in this subsection are found in the unpublished manuscript \cite{MDS-HADAMARD-FACTTHMS}. 
The theorems we present here provide some useful extensions of the factorization theorems and 
related constructions we have cited in the previous subsections so far. 
The next theorems in this section define the key matrix sequences, 
$h_{n,k}(f)$ and $h_{n,k}^{-1}(f)$, in terms of the 
next factorization of the Lambert series over $a_{\fg}(n)$ in the form of 
\begin{align} 
\label{eqn_HPFactGenExp} 
\sum_{n \geq 1} \frac{a_{\fg}(n) q^n}{1-q^n} & = \frac{1}{(q; q)_{\infty}} \times \sum_{n \geq 1} 
     \sum_{k=1}^n h_{n,k}(f) g_k q^n, 
\end{align} 
where the matrix entries $h_{n,k}(f)$ are independent of the function $g$. This expansion 
is equivalent to defining the factorization expansion by the inverse matrix sequences as 
\begin{align} 
\label{eqn_HPFactGenExp_v2} 
g_n & = \sum_{k=1}^n h_{n,k}^{-1} \times 
     [q^k] \left((q; q)_{\infty} \times \sum_{n \geq 1} \frac{a_{\fg}(n) q^n}{1-q^n}\right). 
\end{align} 

\begin{theorem} 
For all integers $1 \leq k \leq n$, we have the following definition of the factorization matrix 
sequence defining the expansion on the right-hand-side of \eqref{eqn_HPFactGenExp} where 
we adopt the notation $\widetilde{f}(n) := \sum_{d|n} f_d$: 
\begin{align*} 
 & h_{n,k}(f) = T_{\Div}(n, k) \widetilde{f}(n) \\ 
 & \phantom{=\quad\ } + \sum_{b=\pm 1} 
     \sum_{j=1}^{\left\lfloor \frac{\sqrt{24(n-k)+1}-b}{6} \right\rfloor} 
     (-1)^j T_{\Div}\left(n-\frac{j(3j+b)}{2}, k\right) 
     \widetilde{f}\left(n-\frac{j(3j+b)}{2}\right).
\end{align*} 
\end{theorem} 
\begin{proof} 
By the factorization in \eqref{eqn_HPFactGenExp} and the definition of $a_{\fg}(n)$ given above, we have 
that for $\widetilde{f}(n) = \sum_{d|n} f_d$ 
\begin{align*} 
h_{n,k}(f) & = [g_k] \left(\sum_{d|n} f_d\right) \times \sum_{d=1}^n g_d T_{\Div}(n, d), \\ 
     & = 
     [q^n] (q; q)_{\infty} \times \sum_{n \geq 1} T_{\Div}(n, k) \widetilde{f}(n) q^n. 
\end{align*} 
The last equation gives the stated expansion of the sequence by 
Euler's pentagonal number theorem. The statement of the pentagonal number theorem 
is that for $|q| < 1$
\begin{align*} 
(q; q)_{\infty} & = 1 + \sum_{j \geq 1} (-1)^j \left(q^{\frac{j(3j-1)}{2}} + 
     q^{\frac{j(3j+1)}{2}}\right). 
\end{align*} 
The theorem statement follows thusly.
\end{proof} 

\begin{theorem}[Inverse matrix sequences] 
\label{theorem_HP_InvSeqs} 
For all integers $1 \leq k \leq n$, we have the next definition of the inverse factorization matrix 
sequence which equivalently defines the expansion on the right-hand-side of \eqref{eqn_HPFactGenExp}. 
\[
h_{n,k}^{-1}(f) = \sum_{d|n} \frac{p(d-k)}{\widetilde{f}(d)} \mu\left(\frac{n}{d}\right). 
\] 
\end{theorem} 
\begin{proof}[Proof of Theorem \ref{theorem_HP_InvSeqs}] 
We expand the right-hand-side of the factorization in \eqref{eqn_HPFactGenExp} 
for the sequence $g_n := h_{n,r}^{-1}(f)$, i.e., the exact inverse sequence, 
for some fixed $r \geq 1$ as follows: 
\begin{align*} 
\widetilde{f}(n) \times \sum_{d|n} h_{d,r}^{-1}(f) & = 
     \sum_{j=0}^n \sum_{k=1}^j h_{j,k} h_{k,r}^{-1} p(n-j) \\ 
     & = 
     \sum_{j=0}^n \Iverson{j = r} p(n-j) = p(n-r). 
\end{align*} 
Then the last equation implies that 
$$\sum_{d|n} h_{d,r}^{-1}(f) = \frac{p(n-r)}{\widetilde{f}(n)}, $$ 
which by M\"obius inversion implies our stated result. 
\end{proof}

\begin{example}[Applications] 
\label{example_HPThm_consequences}
If we form the Hadamard product of generating functions of the two 
Lambert series over Euler's totient function, $\phi(n)$, we obtain the following more 
exotic-looking sum for our multiplicative function of interest: 
\[
\phi(n) = \sum_{k=1}^n \sum_{d|n} \frac{p(d-k)}{d} \mu\left(\frac{n}{d}\right) \left[k^2 + 
     \sum_{b=\pm 1} \sum_{j=1}^{\left\lfloor \frac{\sqrt{24k-23}-b}{6} \right\rfloor} (-1)^j 
     \left(k-\frac{j(3j+b)}{2}\right)^2\right]. 
\]
We consider the arithmetic function pairs 
$$(f, g) := (n^t, n^s), (\phi(n), \Lambda(n)), (n^t, J_t(n)), $$ 
respectively, and some constants $s,t \in \mathbb{C}$ where $\sigma_{\alpha}(n)$ denotes the 
generalized sum-of-divisors function, $\Lambda(n)$ is von Mangoldt's function, 
$\phi(n)$ is Euler's totient function, and $J_t(n)$ is the Jordan totient function. 
We then employ the equivalent expansions of the factorization result in \eqref{eqn_HPFactGenExp_v2} 
to formulate the following ``\emph{exotic}'' \emph{sums} as consequences of the theorems above: 
\begin{align}
\label{eqn_nsPow_Sigmast_HPIdent_v1} 
n^s & = \sum_{k=1}^n \sum_{d|n} \frac{p(d-k)}{\sigma_t(d)} \mu\left(\frac{n}{d}\right) 
     \Biggl[\sigma_t(k)\sigma_s(k) \\ 
\notag 
     & \phantom{=\sum\ } + 
     \sum_{b=\pm 1} \sum_{j=1}^{\left\lfloor \frac{\sqrt{24k+1}-b}{6} \right\rfloor} (-1)^j 
     \sigma_t\left(k-\frac{j(3j+b)}{2}   \right)\sigma_s\left(k-\frac{j(3j+b)}{2}\right)\Biggr], \\ 
\notag 
\Lambda(n) & = \sum_{k=1}^n \sum_{d|n} \frac{p(d-k)}{d} \mu\left(\frac{n}{d}\right) \Biggl[k\log(k) \\ 
\notag 
     & \phantom{=\sum\ } + 
     \sum_{b=\pm 1} \sum_{j=1}^{\left\lfloor \frac{\sqrt{24k-23}-b}{6} \right\rfloor} (-1)^j 
     \left(k-\frac{j(3j+b)}{2}\right)\log\left(k-\frac{j(3j+b)}{2}\right)\Biggr], \\ 
\notag 
J_t(n) & = \sum_{k=1}^n \sum_{d|n} \frac{p(d-k)}{d^t} \mu\left(\frac{n}{d}\right) \left[k^{2t} + 
     \sum_{b=\pm 1} \sum_{j=1}^{\left\lfloor \frac{\sqrt{24k-23}-b}{6} \right\rfloor} (-1)^j 
     \left(k-\frac{j(3j+b)}{2}\right)^{2t}\right].  
\end{align}
By forming a second sum over the divisors of $n$ on both sides of the first equation above, 
the first more exotic-looking sum for the sum-of-divisors functions leads to an expression for 
$\sigma_s(n)$ as a sum over the paired pointwise products, $\sigma_t(n) \sigma_s(n)$. 
We are not aware of another identity like the first equation in 
\eqref{eqn_nsPow_Sigmast_HPIdent_v1} 
relating Hadamard products of the generalized sum-of-divisors functions in 
a review of surrounding literature.
\end{example}

\begin{cor}[The Riemann zeta function] 
\label{cor_SumOfDivFns_ZetaFnIdents}
For fixed $s,t \in \mathbb{C}$ such that $\Re(s) > 1$, 
we have the following representation of the Riemann zeta function: 
\begin{align*} 
\zeta(s) & = \sum_{n \geq 1} \sum_{k=1}^n \sum_{d|n} \frac{p(d-k)}{\sigma_t(d)} \mu\left(\frac{n}{d}\right) \times 
	\sum_{\substack{j \geq 0 \\ G_j < k}} \frac{(-1)^{\lceil \frac{j}{2} \rceil} 
	\sigma_t(k-G_j) \sigma_s(k-G_j)}{(k-G_j)^s}. 
\end{align*} 
The sequence of interleaved pentagonal numbers is denoted by 
$G_j = \frac{1}{2} \left\lceil \frac{j}{2} \right\rceil \left\lceil \frac{3j+1}{2} \right\rceil$ for 
$j \geq 0$ \cite[\seqnum{A001318}]{OEIS}. 
\end{cor} 
\begin{proof} 
These two identities follow as special cases of the theorem in the form of 
\eqref{eqn_nsPow_Sigmast_HPIdent_v1} above where we note that the symmetry 
identity for the generalized 
sum-of-divisors functions which provides that $\sigma_{-\alpha}(n) = \sigma_{\alpha}(n) n^{-\alpha}$ 
for all $\alpha \in \mathbb{R}$. 
The pentagonal number theorem employed in the inner sums depending on $j$ 
utilizes the classical expansion 
\[
(q; q)_{\infty} = \sum_{j \geq 0} (-1)^{\lceil \frac{j}{2} \rceil} q^{G_j}, |q| < 1. 
\] 
The convergence of these infinite series is guaranteed by our hypothesis that $\Re(s) > 1$. 
\end{proof} 

In the unpublished manuscript \cite{MDS-HADAMARD-FACTTHMS}, we have other variants of 
factorization type theorems corresponding to derivatives of Lambert series generating functions 
and the identities we proved in \cite{MDS-COMBRESTRDIVSUMS-INTEGERS}. 

\newpage
\SectionGTThesisFormatted{Factorization theorems for GCD type sums} 
\label{Section_FactThmsTypeIAndIISums_MousaviSchmidt}

This section summarizes the work of Mousavi and Schmidt published in the 
\emph{Ramanujan Journal} in 2020 (completed for review by $2018$) 
\cite{MOUSAVI-SCHMIDT-2019}. 
The article provides analogs of the factorization theorems for Lambert series 
generating functions studied so far to other combinatorially relevant sums. 
Our new results here provide generating function expansions for the type I and type II 
sums in the form of matrix-based factorization theorems. 
These new results characterizing the expansions of the OGFs of these sum types are 
analogous in many ways to the LGF cases studied in the last section. 
The matrix products involved in expressing the coefficients of these generating 
functions for arbitrary arithmetic functions $f$ and $g$ are also closely related to the 
partition function $p(n)$ (\cf the discussion motivated in 
Section \ref{Section_CanonicalReprsOfFactThms_KernelBasedDCVL}). 

The known Lambert series factorization theorems 
proved in the references, and which are summarized in 
the previous sections on variants above 
demonstrate the flavor of the matrix-based expansions of these forms for 
ordinary divisor sums of the form $(f \ast \mathds{1})(n) = \sum_{d|n} f(d)$. 
Our extensions of these factorization theorem in the context of the 
forms of the type I and type II sums similarly relate special arithmetic functions 
from number theory to partition functions and more additive branches of mathematics. 
The last results proved in 
Section \ref{subSection_DFT_exps} are expanded in the spirit of these 
matrix factorization constructions using discrete Fourier transforms of functions 
evaluated at greatest common divisors. We pay special 
attention to illustrating our new results with many relevant examples and 
new identities.

\begin{definition}
\label{def_TypeITypeIISums_v1}
For any arithmetic functions $f$ and $g$, 
the classes of sums termed \emph{type I} and \emph{type II} sums by the authors 
are defined in respective order as follows:
\begin{align} 
\label{eqn_sum_vars_TL_defs_intro} 
T_{f}(x) & = \sum_{\substack{d=1 \\ (d,x)=1}}^x f(d), x \geq 1 \\ 
\notag 
L_{f,g,k}(x) & = \sum_{d|(k,x)} f(d) g\left(\frac{x}{d}\right), x \geq k \geq 1, x \geq 1. 
\end{align}
We seek to write the type I, or GCD type sums, denoted by $T_f(x)$ 
for any arithmetic function $f$ and $x \geq 2$, 
as the coefficients of the following matrix factorized OGF expansion:
\begin{align} 
\label{eqn_intro_two_sum_fact_types_def_v1}
T_{f}(x) & = [q^x]\left(\frac{1}{(q; q)_{\infty}} \times \sum_{n \geq 2} \sum_{k=1}^n t_{n,k} f(k) q^n + 
     f(1) q\right). 
\end{align} 
For an indeterminate parameter $w$, we seek to factorize the type-II, or Anderson-Apostol like sums, 
$L_{f,g,k}(n)$, according to the expansions 
\begin{align}
\label{eqn_intro_two_sum_fact_types_def_v2}
g(x) & = [q^x]\left(\frac{1}{(q; q)_{\infty}} \times 
     \sum_{n \geq 2} \sum_{k=1}^n u_{n,k}(f, w) \left(
     \sum_{m=1}^k L_{f,g,m}(k) w^m\right) q^n\right),\ w \in \mathbb{C} \setminus \{0\}. 
\end{align} 
The sequence $u_{n,k}(f, w)$ is lower triangular and 
invertible for suitable choices of the indeterminate parameter $w$. 
For a fixed $N \geq 1$, we can truncate these sequences after $N$ rows 
and form the $N \times N$ matrices whose entries are 
$u_{n,k}(f, w)$ for $1 \leq n,k \leq N$. 
The corresponding inverse matrices have terms denoted by 
$u_{n,k}^{(-1)}(f, w)$. 
That is to say, for $n \geq 2$, these inverse matrices satisfy 
\begin{align} 
\label{eqn_intro_two_sum_fact_types_def_v4_invs}
\sum_{m=1}^n L_{f,g,m}(k) w^m & = \sum_{k=1}^n u_{n,k}^{(-1)}(f, w) [q^k] \left((q; q)_{\infty} 
     \times \sum_{n \geq 1} g(n) q^n\right),\ w \in \mathbb{C} \setminus \{0\}.
\end{align} 
\end{definition}

\SubsectionGTThesisFormatted{Factorization theorems for a class of GCD sums (type I sums)}

Since the resulting matrices with entries $t_{n,k}$ are lower triangular and invertible, 
we then obtain that 
\begin{align} 
\label{eqn_intro_two_sum_fact_types_def_v3_invs}
f(n) & = \sum_{k=1}^n t_{n,k}^{(-1)} [q^k](q; q)_{\infty} \times \left(\sum_{n \geq 1} 
     T_f(n) q^n\right), n \geq 1.
\end{align}
In the next results, the function $\chi_{1,k}(n)$ 
refers to the principal Dirichlet character modulo $k$ for some $k \geq 1$.

\begin{theorem}[Exact formulas for the factorization matrix sequences] 
\label{theorem_snk_snkinv_seq_formulas_v1} 
Let lower triangular sequence $t_{n,k}$ be defined by the first expansion 
in \eqref{eqn_intro_two_sum_fact_types_def_v1} above. 
The corresponding inverse matrix coefficients are denoted by $t_{n,k}^{(-1)}$. 
For integers $n \geq k \geq 1$, the two lower triangular 
factorization sequences defining the expansion of 
\eqref{eqn_intro_two_sum_fact_types_def_v1} satisfy exact formulas given by 
\begin{align*} 
\tag{i} 
t_{n,k} & = \sum_{j=0}^{n} (-1)^{\lceil \frac{j}{2} \rceil} \chi_{1,k}(n+1-G_j) \Iverson{n-G_j \geq 1}, \\ 
\tag{ii} 
t_{n,k}^{(-1)} & = \sum_{d=1}^n p(d-k) \mu_{n,d}. 
\end{align*} 
We define the sequence of interleaved pentagonal numbers $G_j$ 
as in the introduction, and the lower tringular sequence $\mu_{n,k}$ as in 
the M\"obius inversion analog proved in 
Proposition \ref{prof_inversion_formula} (see below). 
\end{theorem}

Before we prove the main theorem, we provide several key examples that apply these results to 
formulate new expansions of classical number theoretic functions and polynomials 
in terms of partition functions. 
First, we obtain the following identities for 
Euler's totient function based on our new constructions: 
\begin{align*} 
\phi(n) & = \sum_{j=0}^n \sum_{k=1}^{j-1} \sum_{i=0}^{j} 
     p(n-j) (-1)^{\lceil \frac{i}{2} \rceil} \chi_{1,k}(j-k-G_i) \Iverson{j-k-G_i \geq 1} + 
     \Iverson{n=1} \\ 
\phi(n) & = \sum_{\substack{d=1 \\ (d,n)=1}}^n 
     \left(\sum_{k=1}^{d+1} \sum_{i=1}^d \sum_{j=0}^k 
     p(i+1-k) (-1)^{\lceil \frac{j}{2} \rceil} \phi(k-G_j) \mu_{d,i} \Iverson{k-G_j \geq 1}\right). 
\end{align*} 
To give another related example that applies to classical multiplicative functions, 
recall that we have a known representation for the M\"obius function given as an 
exponential sum in terms of powers of the 
$n^{th}$ primitive roots of unity of the following form \citep[\S 16.6]{HARDYWRIGHT}: 
\[ 
\mu(n) = \sum_{\substack{d=1 \\ (d,n)=1}}^n \omega_n^d. 
\] 
The \emph{Mertens function}, $M(x)$, is defined as the summatory function over the 
M\"obius function $\mu(n)$ for all $n \leq x$. Using the definition of the M\"obius function as one of our 
type I sums defined above, we have new expansions for the Mertens function given by 
(\cf Corollary \ref{cor_Mertens_function_v2}) 
\[
M(x) = \sum_{1 \leq k < j \leq n \leq x} \left(\sum_{i=0}^j p(n-j) (-1)^{\lceil i/2 \rceil} 
     \chi_{1,k}(j-k-G_i) \Iverson{j-k-G_i \geq 1} \omega_n^k\right). 
\] 
Finally, we can form another related polynomial sum of the type indicated above 
when we consider that the logarithm of the \emph{cyclotomic polynomials} leads to the sums 
\begin{align*} 
\log \Phi_n(z) & = \sum_{\substack{1 \leq k \leq n \\ (k, n) = 1}} 
     \log\left(z-\omega_n^k\right) \\ 
     & = 
     \sum_{1 \leq k < j \leq n} \left(\sum_{i=0}^j p(n-j) (-1)^{\lceil i/2 \rceil} 
     \chi_{1,k}(j-k-G_i) \Iverson{j-k-G_i \geq 1} 
     \log\left(z-\omega_n^k\right)
     \right). 
\end{align*} 

\SubsubsectionGTThesisFormatted{Inversion relations} 

We begin our exploration by expanding an inversion formula which is analogous to 
M\"obius inversion for ordinary divisor sums. 
We prove the following result which is the analog to the sequence 
inversion relation provided by the M\"obius transform in the context of our sums over the 
integers relatively prime to $n$ \cite[\cf \S 2, \S 3]{RIORDAN-COMBIDENTS}. 

\begin{prop}[Inversion formula] 
\label{prof_inversion_formula}
For all $n \geq 2$, there is a unique lower triangular sequence, denoted by $\mu_{n,k}$, 
which satisfies the next lower triangular inversion relation, i.e., 
so that $\mu_{n,d} = 0$ whenever $n < d$. 
\begin{subequations}
\begin{equation}
\label{eqn_propInvFormula_v1}
g(n) = \sum_{\substack{d=1 \\ (d, n)=1}}^n 
     f(d) \quad\iff\quad f(n) = \sum_{d=1}^n g(d+1) \mu_{n,d}. 
\end{equation}
Moreover, if we form the matrix $(\mu_{i,j} \Iverson{j \leq i})_{1 \leq i,j \leq n}$ for any $n \geq 2$, we have that the 
inverse sequence satisfies 
\begin{equation} 
\label{eqn_propInvFormula_v2}
\mu_{n,k}^{(-1)} = \Iverson{(n+1, k) = 1} \Iverson{k \leq n}. 
\end{equation} 
\end{subequations} 
\end{prop} 
\begin{proof}[Proof of Proposition \ref{prof_inversion_formula}]
Consider the $(n-1) \times (n-1)$ matrix 
\begin{equation}
\label{eqn_InvRel_proof_matrix_def_v1} 
\left(\Iverson{(i,j-1)=1\mathrm{\ and \ } j \leq i}\right)_{1 \leq i,j < n},
\end{equation} 
which effectively corresponds to the formula on the left-hand-side of \eqref{eqn_propInvFormula_v1} by 
applying the matrix to the vector of $[f(1)\ f(2)\ \cdots f(n)]^{T}$ and extracting the 
$(n+1)^{th}$ column of the matrix formed by extracting the $\{0,1\}$-valued coefficients of $f(d)$. 
Since $\gcd(i, j-1)=1$ for all $i = j$ with $i,j \geq 1$, we see that the 
matrix \eqref{eqn_InvRel_proof_matrix_def_v1} 
is lower triangular with ones on its diagonal. Thus the matrix is non-singular and its 
unique inverse, which we denote by $(\mu_{i,j})_{1 \leq i,j < n}$, leads to the sum on the 
right-hand-side of the sum in \eqref{eqn_propInvFormula_v1} when we shift $n \mapsto n+1$. 
The second equation stated in \eqref{eqn_propInvFormula_v2} 
restates the form of the first matrix of $\mu_{i,j}$ 
as on the right-hand-side of \eqref{eqn_propInvFormula_v1}. 
\end{proof} 

\begin{figure}[ht!]

\caption[GCD sum inversion coefficients]{Inversion formula coefficient sequences.} 
\label{figure_munk_and_inv_seqs} 

\begin{minipage}{\linewidth} 
\begin{center} 
\small
\begin{equation*} 
\boxed{ 
\begin{array}{ccccccccccccccccc}
 1 & 0 & 0 & 0 & 0 & 0 & 0 & 0 & 0 & 0 & 0 & 0 & 0 & 0 & 0 & 0 & 0 \\
 -1 & 1 & 0 & 0 & 0 & 0 & 0 & 0 & 0 & 0 & 0 & 0 & 0 & 0 & 0 & 0 & 0 \\
 -1 & 0 & 1 & 0 & 0 & 0 & 0 & 0 & 0 & 0 & 0 & 0 & 0 & 0 & 0 & 0 & 0 \\
 1 & -1 & -1 & 1 & 0 & 0 & 0 & 0 & 0 & 0 & 0 & 0 & 0 & 0 & 0 & 0 & 0 \\
 -1 & 0 & 0 & 0 & 1 & 0 & 0 & 0 & 0 & 0 & 0 & 0 & 0 & 0 & 0 & 0 & 0 \\
 1 & 0 & 0 & -1 & -1 & 1 & 0 & 0 & 0 & 0 & 0 & 0 & 0 & 0 & 0 & 0 & 0 \\
 1 & 0 & -1 & 0 & -1 & 0 & 1 & 0 & 0 & 0 & 0 & 0 & 0 & 0 & 0 & 0 & 0 \\
 -1 & 0 & 2 & -1 & 0 & 0 & -1 & 1 & 0 & 0 & 0 & 0 & 0 & 0 & 0 & 0 & 0 \\
 -1 & 0 & 0 & 0 & 1 & 0 & -1 & 0 & 1 & 0 & 0 & 0 & 0 & 0 & 0 & 0 & 0 \\
 1 & 0 & -1 & 1 & 0 & -1 & 1 & -1 & -1 & 1 & 0 & 0 & 0 & 0 & 0 & 0 & 0 \\
 -1 & 0 & 1 & 0 & 0 & 0 & -1 & 0 & 0 & 0 & 1 & 0 & 0 & 0 & 0 & 0 & 0 \\
 1 & 0 & -1 & 0 & 0 & 0 & 1 & 0 & 0 & -1 & -1 & 1 & 0 & 0 & 0 & 0 & 0 \\
 3 & 0 & -2 & 0 & -2 & 0 & 2 & 0 & -1 & 0 & -1 & 0 & 1 & 0 & 0 & 0 & 0 \\
 -3 & 0 & 1 & 0 & 3 & 0 & -1 & -1 & 1 & 0 & 0 & 0 & -1 & 1 & 0 & 0 & 0 \\
 -1 & 0 & 1 & 0 & 1 & 0 & -1 & 0 & 0 & 0 & 0 & 0 & -1 & 0 & 1 & 0 & 0 \\
 1 & 0 & 0 & 0 & -2 & 0 & 0 & 1 & 0 & 0 & 1 & -1 & 1 & -1 & -1 & 1 & 0 \\
 -3 & 0 & 2 & 0 & 2 & 0 & -2 & 0 & 1 & 0 & 0 & 0 & -1 & 0 & 0 & 0 & 1 \\
\end{array}
}
\end{equation*}
\end{center} 
\subcaption*{$\mu_{n,k}$ for $1 \leq n,k < 18$ with rows indexed by $n$ and columns by $k$.} 
\end{minipage} 

\end{figure} 

\begin{remark} 
There is not an apparent 
simple closed-form function for the sequence $\{\mu_{n,k}\}_{1 \leq k \leq n}$. 
These coefficients are defined recursively by  applying the rules 
\[
\begin{cases}
\mu_{n,1} = 1, & \text{ if $n = 1$ and $k = 1$; } \\ 
\sum\limits_{\substack{2 \leq j \leq n+1 \\ (n+1,j)=1}} \mu_{n+1-k,j} = -\chi_{1,k}(n+1), & 
     \text{ if $n \geq 2$ and $1 \leq k \leq n$. } \\ 
\mu_{n,k} = 0, & \text{ otherwise. }
\end{cases}
\]
On the other hand, we can readily see by construction that the 
sequence and its inverse satisfy 
\begin{align*} 
\sum_{\substack{d=1 \\ (d,n)=1}}^n \mu_{d,k} & = 0 \\ 
\sum_{\substack{d=1 \\ (d,n)=1}}^n \mu_{d,k}^{(-1)} & = \phi(n), 
\end{align*} 
where $\phi(n)$ is Euler's totient function. 
The first columns of the $\mu_{n,1}$ appear in the online integer sequences 
database (OEIS) as the entry \cite[\seqnum{A096433}]{OEIS}. 
\end{remark} 

\SubsubsectionGTThesisFormatted{Exact formulas for the factorization matrices} 

The next result is key to proving the exact formulas for the matrix sequences, 
$t_{n,k}$ and $t_{n,k}^{(-1)}$, and their expansions by the partition functions defined 
in the introduction. We prove the following result
first as a lemma which we will use in the proof of 
Theorem \ref{theorem_snk_snkinv_seq_formulas_v1} given below. 
The first several rows of the matrix sequence $t_{n,k}$ and its inverse 
implicit to the factorization theorem in \eqref{eqn_intro_two_sum_fact_types_def_v1} are tabulated in 
Figure \ref{figure_lseriesfact_s3nk_tables} for intuition on the 
formulas we prove in the next proposition and following theorem. 

\begin{lemma}[A convolution identity] 
\label{lemma_cvl_ident_rpints_indicatorfn} 
For all natural numbers $n \geq 2$ and $k \geq 1$ with $k \leq n$, we have the 
following expression for the principal Dirichlet character modulo $k$: 
\[ 
\sum_{j=1}^n t_{j,k} p(n-j) = \chi_{1,k}(n). 
\] 
Equivalently, we have that 
\begin{align} 
\label{eqn_Proof_cvl_ident_rpints_indicatorfn_EQ1}
t_{n,k} & = \sum_{i=0}^n (-1)^{\lceil \frac{i}{2} \rceil} \chi_{1,k}(n-G_i) \Iverson{n-G_i \geq k+1} \\ 
\notag
     & = 
     \sum_{b = \pm 1} \left[\sum_{i=0}^{\floor{\frac{\sqrt{24(n-k-1)+1}-b}{6}}} \
     (-1)^{\ceiling{\frac{i}{2}}} \chi_{1,k}\left(n - \frac{i(3i-b)}{2}\right)
     \right].  
\end{align} 
\end{lemma} 
\begin{proof} 
We begin by noticing that the right-hand-side expression in the statement of the lemma 
is equal to $\mu_{n,k}^{(-1)}$ by the construction of the sequence in 
Proposition \ref{prof_inversion_formula}. 
Next, we see that the factorization in 
\eqref{eqn_intro_two_sum_fact_types_def_v1} is equivalent to the 
expansion 
\begin{align}
\label{eqn_tnk_formula_proof_formula_v2}  
\sum_{d=1}^{n-1} f(d) \mu_{n,d}^{(-1)} & = 
     \sum_{j=1}^n \sum_{k=1}^{j} p(n-j) t_{j,k} f(k). 
\end{align} 
Since $\mu_{n,k}^{(-1)} = \Iverson{(n+1, k) = 1}$, we may take the coefficients of $f(k)$ on 
each side of \eqref{eqn_tnk_formula_proof_formula_v2} 
for each $1 \leq k < n$ to establish the result we have claimed in this lemma. 
The equivalent statement of the result follows by a generating function argument 
applied to the product that generates the left-hand-side Cauchy product in 
\eqref{eqn_Proof_cvl_ident_rpints_indicatorfn_EQ1}. 
\end{proof}

\begin{figure}[ht!]

\caption[The factorization matrices and their inverses for type I GCD sums]{
	 The factorization matrices, $t_{n,k}$ and $t_{n,k}^{(-1)}$, for 
         $1 \leq n,k < 14$ with rows indexed by $n$ and columns by $k$.} 
\label{figure_lseriesfact_s3nk_tables} 

\begin{minipage}{\linewidth} 
\begin{center} 
\small
\begin{equation*} 
\boxed{ 
\begin{array}{cccccccccccccc}
1 & \\
0 & 1 & \\
-1 & -1 & 1 & \\
-1 & 0 & 0 & 1 & \\
-1 & -1 & -2 & -1 & 1 & \\
0 & 0 & 0 & 0 & 0 & 1 & \\
0 & 0 & 0 & -1 & -1 & -1 & 1 & \\
1 & 0 & -1 & 0 & -1 & -1 & 0 & 1 & \\
1 & 1 & 1 & 0 & -2 & 0 & -1 & -1 & 1 & \\
1 & 0 & 1 & 0 & 1 & 1 & -1 & 0 & 0 & 1 & \\
1 & 1 & 0 & 1 & 1 & 0 & -1 & -1 & -2 & -1 & 1 & \\
1 & 0 & 1 & 0 & 1 & 0 & 0 & 0 & 0 & 0 & 0 & 1 & \\
0 & 1 & 1 & 1 & 1 & 0 & -1 & 0 & 0 & -1 & -1 & -1 & 1 & \\
0 & -1 & 0 & 0 & -1 & -1 & 2 & 0 & -1 & -1 & -1 & -1 & 0 & 1 \\
\end{array}
}
\end{equation*}
\end{center} 
\subcaption*{\rm{(i)} $t_{n,k}$} 
\end{minipage} 

\begin{minipage}{\linewidth} 
\begin{center} 
\small
\begin{equation*} 
\boxed{ 
\begin{array}{ccccccccccccc}
 1 & \\
 0 & 1 & \\
 1 & 1 & 1 & \\
 1 & 0 & 0 & 1 & \\
 4 & 3 & 2 & 1 & 1 & \\
 0 & 0 & 0 & 0 & 0 & 1 & \\
 5 & 3 & 2 & 2 & 1 & 1 & 1 & \\
 4 & 4 & 3 & 1 & 1 & 1 & 0 & 1 & \\
 15 & 11 & 8 & 5 & 4 & 2 & 1 & 1 & 1 & \\
 -1 & -1 & -1 & 1 & 0 & 0 & 1 & 0 & 0 & 1 & \\
 32 & 24 & 18 & 12 & 9 & 6 & 4 & 3 & 2 & 1 & 1 & \\
 -6 & -4 & -3 & -1 & -1 & 0 & 0 & 0 & 0 & 0 & 0 & 1 & \\
 24 & 17 & 13 & 12 & 8 & 7 & 6 & 3 & 2 & 2 & 1 & 1 & 1 \\
\end{array}
}
\end{equation*}
\end{center} 
\subcaption*{\rm{(ii)} $t_{n,k}^{(-1)}$} 
\end{minipage} 

\end{figure} 

\begin{proof}[Theorem \ref{theorem_snk_snkinv_seq_formulas_v1}: Proof of (i)]  
It is plain to see by the considerations in our construction of the factorization theorem that 
both matrix sequences are lower triangular. Thus, we need only consider the cases where 
$n \leq k$. By a convolution of generating functions, the identity in 
Lemma \ref{lemma_cvl_ident_rpints_indicatorfn} shows that 
\[ 
t_{n,k} = \sum_{j=k}^n [q^{n-j}] (q; q)_{\infty} \Iverson{(j+1, k) = 1}. 
\] 
Then shifting the index of summation in the previous equation implies (i). 
\end{proof} 
\begin{proof}[Theorem \ref{theorem_snk_snkinv_seq_formulas_v1}: Proof of (ii)] 
To prove (ii), we consider the factorization theorem when $f(n) := t_{n,r}^{(-1)}$ for some 
fixed $r \geq 1$. We then expand \eqref{eqn_intro_two_sum_fact_types_def_v1} as 
\begin{align*} 
\sum_{\substack{d=1 \\ (d,n)=1}}^n 
     t_{d,r}^{(-1)} & = [q^n] \frac{1}{(q; q)_{\infty}} \times \sum_{n \geq 1} 
     \left(\sum_{k=1}^{n-1} t_{n,k} t_{k,r}^{(-1)}\right) q^n \\ 
     & = 
     \sum_{j=1}^n p(n-j) \left(\sum_{k=1}^{j-1} t_{j,k} t_{k,r}^{(-1)}\right) \\ 
     & = 
     \sum_{j=1}^n p(n-j) \Iverson{r=j-1} \\ 
     & = 
     p(n-1-r). 
\end{align*} 
Hence, we may perform the inversion by Proposition \ref{prof_inversion_formula} to the 
left-hand-side sum in the previous equations to obtain our stated result. 
\end{proof} 

\begin{remark}[Relations to the Lambert series factorization theorems]
We notice that by inclusion-exclusion applied to the right-hand-side of 
\eqref{eqn_intro_two_sum_fact_types_def_v1}, 
we may write our matrices $t_{n,k}$ in terms of the triangular 
sequence expanded as differences of restricted partitions in the 
ordinary Lambert series factorizations involving the sequence 
$s_{n,k} := [q^{n-k}] \frac{(q; q)_{\infty}}{1-q^k}$. 
For example, when $k := 12$ we see that 
\[
\sum_{n \geq 12} \Iverson{(n, 12) = 1} q^n = \frac{q^{12}}{1-q} - \frac{q^{12}}{1-q^2} - 
     \frac{q^{12}}{1-q^3} + \frac{q^{12}}{1-q^6}. 
\]
In general, when $k > 1$ we can expand 
\begin{equation}
\label{eqn_GCDnk_GFSeries_FiniteMobiusExp_v1} 
\sum_{n \geq k} \Iverson{(n, k) = 1} q^n = \sum_{d|k} \frac{q^k \mu(d)}{1-q^d}. 
\end{equation}
Thus, we can relate the triangles $t_{n,k}$ in this article to the 
$s_{n,k} = [q^n] (q; q)_{\infty} q^k (1-q^k)^{-1}$ for $n \geq k \geq 1$ employed in the 
expansions from the references as follows: 
\[
t_{n,k} = \begin{cases} 
     s_{n,k}, & k = 1; \\ 
     \sum\limits_{d|k} \mu(d) s_{n+1-k+d,d}, & k > 1. 
     \end{cases} 
\]
\end{remark} 

\SubsubsectionGTThesisFormatted{Completing the proofs of the main applications} 

We remark that as in the Lambert series factorization results from the references \cite{MERCA-LSFACTTHM}, 
we have three primary types of expansion identities that we will consider for 
any fixed choice of the arithmetic function $f$ in the forms of 
\begin{subequations} 
\label{eqn_factthm_exp_idents_v1} 
\begin{align} 
\sum_{\substack{d=1 \\ (d,n)=1}}^n 
     f(d) & = \sum_{j=1}^n \sum_{k=1}^{j-1} p(n-j) t_{j-1,k} f(k) + 
     f(1) \Iverson{n = 1} \\ 
\sum_{k=1}^{n-1} t_{n-1,k} f(k) & = \sum_{j=1}^n \left(\sum_{\substack{d=1 \\ (d,j)=1}}^j 
      f(d) \times [q^{n-j}] (q; q)_{\infty}\right) - f(1) \times [q^{n-1}](q; q)_{\infty}, 
\end{align} 
and the 
corresponding inverted formula providing that 
\begin{align} 
\label{eqn_factthm_exp_idents_v1_eq.III} 
f(n) & = \sum_{k=1}^{n} t_{n,k}^{(-1)}\left( 
     \sum_{\substack{j \geq 0 \\ k+1-G_j > 0}} (-1)^{\ceiling{\frac{j}{2}}} 
     T_f(k+1-G_j) - [q^k] (q; q)_{\infty} f(1)\right). 
\end{align} 
\end{subequations}
Now the applications cited in the introduction follow immediately and require no 
further proof other than to cite these results for the respective special cases of $f$. 

\begin{example}[Sum-of-divisors functions] 
For any $\alpha \in \mathbb{C}$, the expansion identity given in 
\eqref{eqn_factthm_exp_idents_v1_eq.III} also implies the following new formula for the 
generalized sum-of-divisors functions, $\sigma_{\alpha}(n) = \sum_{d|n} d^{\alpha}$: 
\[
\sigma_{\alpha}(n) = \sum_{d|n} \sum_{k=1}^{d} t_{d,k}^{(-1)}\left( 
     \sum_{\substack{j \geq 0 \\ k+1-G_j > 0}} (-1)^{\ceiling{\frac{j}{2}}} 
     \phi_{\alpha+1}(k+1-G_j) - [q^k] (q; q)_{\infty}\right). 
\] 
In particular, when $\alpha := 0$ we obtain the next identity for the divisor function 
$d(n) \equiv \sigma_0(n)$ expanded in terms of Euler's totient function, $\phi(n)$. 
\[
d(n) = \sum_{d|n} \sum_{k=1}^{d} t_{d,k}^{(-1)}\left( 
     \sum_{\substack{j \geq 0 \\ k+1-G_j > 0}} (-1)^{\ceiling{\frac{j}{2}}} 
     \phi(k+1-G_j) - [q^k] (q; q)_{\infty}\right). 
\] 
\end{example} 

\begin{example}[Menon's identity and related arithmetical sums] 
\label{example_Menons_Identity_Toth} 
We can use our new results proved in this section to expand new identities for 
known closed-forms of special arithmetic sums. For example, 
\emph{Menon's identity} \cite{TOTH} states that 
\[
\phi(n) d(n) = \sum_{\substack{1 \leq k \leq n \\ (k,n)=1}} \gcd(k-1, n), 
\]
where $\phi(n)$ is Euler's totient function and $d(n) = \sigma_0(n)$ is the 
divisor function. We can then expand the right-hand-side of Menon's identity as 
follows: 
\[
\phi(n) d(n) = \sum_{j=0}^n \sum_{k=1}^{j-1} \sum_{i=0}^j 
     p(n-j) (-1)^{\ceiling{\frac{i}{2}}} \chi_{1,k}(j-k-G_i) \Iverson{j-k-G_i \geq 1} 
     \gcd(k-1, n). 
\] 
As another application, we show a closely related identity considered by T\'oth in \cite{TOTH}. 
T\'oth's identity states that (\cf \cite{GCD-SUMS}) for an arithmetic function $f$ we have
\[
\sum_{\substack{1 \leq k \leq n \\ (k,n)=1}} f\left(\gcd(k-1, n)\right) = 
     \phi(n) \times \sum_{d|n} \frac{(\mu \ast f)(d)}{\phi(d)}. 
\] 
We can use our new formulas to write a gcd-related recurrence relation for $f$ 
in two steps. First, we observe that the right-hand-side divisor sum in the 
previous equation is expanded by 
\begin{align*} 
\sum_{d|n} \frac{(\mu \ast f)(d)}{\phi(d)} & = 
     \sum_{\substack{0 \leq j \leq n \\ 0 \leq k \leq j-2 \\ 0 \leq i \leq j}}
     \frac{(-1)^{\ceiling{\frac{i}{2}}} p(n-j)}{\phi(n)} 
     \chi_{1,k}(j-k-2-G_i) f(\gcd(k, n)) \\ 
     & \phantom{=\sum\qquad\ } + f(1) \Iverson{n = 1}. 
\end{align*} 
Next, by M\"obius inversion and noting that the Dirichlet inverse of $\mu(n)$ is 
$\mu \ast \mathds{1} = \varepsilon$, where $\varepsilon(n) = \delta_{n,1}$ is the multiplicative 
identity with respect to Dirichlet convolution, we can express $f(n)$ as follows: 
\begin{align*} 
f(n) & = \sum_{d|n} \sum_{r|d} \sum_{j=0}^r \sum_{k=1}^{j-1} \sum_{i=0}^j \Biggl[
     p(r-j) (-1)^{\ceiling{\frac{i}{2}}} \chi_{1,k}(j-k-G_i) \Iverson{j-k-G_i \geq 1} \times \\ 
     & \phantom{=\sum_{d|n} \sum_{r|d} \sum_{j=0}^r \sum_{k=1}^{j-1} \sum_{i=0}^j\quad\ } 
     \times 
     f(\gcd(k-1, r)) \frac{\phi(d)}{\phi(r)} \mu\left(\frac{d}{r}\right)\Biggr] + 
     \sum_{d|n} f(1) \phi(d) \mu(d). 
\end{align*} 
We can expand the last right-hand-side term by noting that for $f$ multiplicative 
\cite[\S 27]{NISTHB}
\[
\sum_{d|n} f(d) \mu(d) = \prod_{p|n} (1-f(p)), n \geq 1. 
\]
Therefore, the last term satisfies 
\[
\sum_{d|n} \phi(d) \mu(d) = \prod_{p|n} (2-p), n \geq 1. 
\]
\end{example} 

\SubsectionGTThesisFormatted{Factorization theorems for Anderson-Apostol sums (type II sums)}

Recall the notation from Definition \ref{def_TypeITypeIISums_v1}. 
The sums $L_{f,g,k}(n)$ are sometimes refered to as \emph{Anderson-Apostol sums} named 
after the authors who first defined these sums 
(\cf \cite[\S 8.3]{APOSTOLANUMT} \cite{APOSTOL-APGRS}). 
Other variants and generalizations of these sums are studied in \cite{IKEDA,KIUCHI-AVGS}. 
There are many number theoretic applications of the periodic sums factorized in this form. 
For example, the famous expansion of Ramanujan's sum $c_q(n)$ is expressed as the 
following right-hand-side divisor sum \cite[\S IX]{RAMANUJAN}: 
\[
c_q(n) = \sum_{\substack{d=1 \\ (d, n)=1}}^n 
     \omega_q^{dn} = \sum_{d|(q,n)} d \cdot \mu\left(\frac{q}{d}\right). 
\] 
The applications of our new results to Ramanujan's sum include the 
expansions 
\begin{align*} 
c_n(x) & = [w^{x}]\left(\sum_{k=1}^n u_{n,k}^{(-1)}(\mu, w) \times \sum_{j \geq 0} 
     (-1)^{\lceil \frac{j}{2} \rceil} \mu(k-G_j)\right) \\ 
     & = 
     \sum_{k=1}^n \left(\sum_{d|(n,x)} d \cdot p\left(\frac{n}{d}-k\right)\right) \times \sum_{j \geq 0} 
     (-1)^{\lceil \frac{j}{2} \rceil} \mu(k-G_j), 
\end{align*} 
where the inverse matrices $u_{n,k}^{(-1)}(\mu, w)$ 
are expanded according to Proposition \ref{prop_unk_inverse_matrix}. 
We then immediately have the following new results for the next special expansions of the 
generalized sum-of-divisors functions when $\Re(s) > 0$: 
\begin{align*} 
\sigma_s(n) & = n^s \zeta(s+1) \times \sum_{i=1}^{\infty} \sum_{k=1}^i \left(\sum_{d|(n,i)} 
     d \cdot p\left(\frac{i}{d}-k\right)\right) \sum_{j \geq 0} 
     \frac{(-1)^{\lceil \frac{j}{2} \rceil} \mu(k-G_j)}{i^{s+1}}. 
\end{align*}

\SubsubsectionGTThesisFormatted{Formulas for the inverse matrices} 
\label{subSection_typeII_inv_matrices} 

It happens that in the case of the series expansions we defined in 
\eqref{eqn_intro_two_sum_fact_types_def_v2}, the 
corresponding terms of the inverse matrices $u_{n,k}^{(-1)}(f, w)$ 
satisfy considerably simpler formulas that the ordinary matrix entries themselves. 
We first prove a partition-related explicit formula for these inverse matrices as 
Proposition \ref{prop_unk_inverse_matrix} and 
then discuss several applications of this result. 

\begin{prop}[Formulas for the inverse matrices of type II sums] 
\label{prop_unk_inverse_matrix} 
For all $n \geq 1$ and $1 \leq k \leq n$, any fixed arithmetic function $f$, 
and $w \in \mathbb{C} \setminus \{0\}$, we have that 
\[
u_{n,k}^{(-1)}(f, w) = \sum_{m=1}^n \left(\sum_{d|(m,n)} f(d) p\left(\frac{n}{d}-k\right)\right) w^m. 
\]
\end{prop} 
\begin{proof}[Proof of Proposition \ref{prop_unk_inverse_matrix}]
Let $1 \leq r \leq n$ and for some suitably chosen arithmetic function $g$ define 
\[
\tag{i} 
u_{n,r}^{(-1)}(f, w) := \sum_{m=1}^n L_{f,g,m}(n) w^m. 
\]
By directly expanding the series on the right-hand-side of 
\eqref{eqn_intro_two_sum_fact_types_def_v2}, we obtain that 
\begin{align*} 
g(n) & = \sum_{j=0}^n \left( 
     \sum_{k=1}^j u_{j,k}(f, w) u_{k,r}^{(-1)}(f, w)\right) p(n-j) \\ 
     & = \sum_{j=0}^n p(n-j) \Iverson{j = r} 
     = p(n-r). 
\end{align*} 
Hence the choice of the function $g$ which satisfies (i) above is given by 
$g(n) := p(n-r)$. The claimed expansion of the inverse matrices then follows. 
\end{proof}

\begin{prop}
\label{prop_exact_exp_of_Lfgmn} 
We have the following identity that holds for any arithmetic functions $f,g$: 
\begin{equation} 
\label{eqn_gen_Lfgmn_formula} 
L_{f,g,m}(n) = \sum_{k=1}^n \sum_{d|(m,n)} f(d) p\left(\frac{n}{d}-k\right) \times 
     \sum_{\substack{j \geq 0 \\ k > G_j}} (-1)^{\ceiling{\frac{j}{2}}} 
     g\left(k-G_j\right).
\end{equation} 
\end{prop}
\begin{proof}[Proof of Proposition \ref{prop_exact_exp_of_Lfgmn}]
Since the coefficients on the left-hand-side of the next equation correspond to a 
right-hand-side matrix product as 
\[
[q^n] (q; q)_{\infty} \times \sum_{m \geq 1} g(m) q^m = 
     \sum_{k=1}^n u_{n,k}(f, w) \times \sum_{m=1}^k L_{f,g,m}(k) w^m, 
\] 
we can invert the matrix product on the right to obtain that 
\[
\sum_{m=1}^k L_{f,g,m}(k) w^m = \sum_{k=1}^{n} \left(\sum_{m=1}^n \sum_{d|(n,m)} 
     f(d) p\left(\frac{n}{d}-k\right) w^m\right) \times 
     [q^k] (q; q)_{\infty} \times \sum_{m \geq 1} g(m), 
\]
so that by comparing coefficients of $w^m$ for $1 \leq m \leq n$, we obtain 
\eqref{eqn_gen_Lfgmn_formula}. 
\end{proof} 

\begin{cor}[A new formula to express Ramanujan sums]
\label{cor_RamSumNewFormula} 
For any natural numbers $x,m \geq 1$, we have that 
\[
c_x(m) = \sum_{k=1}^x \sum_{d|(m,x)} d \cdot p\left(\frac{x}{d}-k\right) \times 
     \sum_{\substack{j \geq 0 \\ k > G_j}} (-1)^{\ceiling{\frac{j}{2}}} 
     \mu\left(k-G_j\right). 
\] 
\end{cor}
\begin{proof}[Proof of Corollary \ref{cor_RamSumNewFormula}]
The Ramanujan sums correspond to the special case of 
Proposition \ref{prop_exact_exp_of_Lfgmn} 
where $f(n) := n$ is the identity function and $g(n) := \mu(n)$ is the 
M\"obius function. 
\end{proof} 

\begin{remark} 
We define the following shorthand notation for fixed arithmetic functions $f,g$, 
integers $n \geq 1$ and $w \in \mathbb{C} \setminus \{0\}$: 
\begin{equation*} 
\widehat{L}_{f,g}(n; w) := \sum_{m=1}^{n} L_{f,g,m}(n) w^m. 
\end{equation*}
In this notation we have that $u_{n,k}^{(-1)}(f, w) = \widehat{L}_{f,g}(n; w)$ when 
$g(n) := p(n-k)$. 
Moreover, if we denote by $T_n(x)$ the polynomial 
$T_n(x) := 1+x+x^2+\cdots+x^{n-1} = \frac{1-x^n}{1-x}$, 
then we have expansions of these sums as convolved ordinary divisor sums by 
polynomial terms of the form 
\begin{align}
\label{eqn_lfgnw_sum_divsum_exp_v1} 
\widehat{L}_{f,g}(n; w) & = \sum_{d|n} w^d f(d) T_{\frac{n}{d}}(w^d) 
     g\left(\frac{n}{d}\right) \\ 
\notag 
     & = (w^n-1) \times \sum_{d|n} \frac{w^d}{w^d-1} f(d) g\left(\frac{n}{d}\right). 
\end{align} 
The Dirichlet inverse of this divisor sums is also not difficult to express, though 
we will not give its formula here. 
These sums lead to the expressions for the 
ordinary matrix entries $u_{n,k}(f, w)$ given by the next corollary. 
\end{remark} 

\begin{cor}[A formula for the ordinary matrix entries] 
\label{cor_unkfw_ord_matrix_formula_v1} 
To distinguish notation, let $$\widehat{P}_{f,k}(n; w) := \widehat{L}_{f(n),p(n-k)}(n; w),$$ which is an 
immediate shorthand for the matrix inverse terms $u_{n,k}^{(-1)}(f, w)$ that we will precisely 
enumerate below. 
For $n \geq 1$ and $1 \leq k < n$, we have the following formula: 
\begin{align*} 
 & u_{n,k}(f, w) = -\frac{(1-w)^2}{w^2 (1-w^n)(1-w^k) f(1)^2}\Biggl( 
     \widehat{P}_{f,k}(n; w) \\ 
     & \phantom{=\ } + \sum_{m=1}^{n-k-1} 
     \left(\frac{w-1}{w f(1)}\right)^{m} \left[ 
     \sum_{k \leq i_1 < \cdots < i_m < n} \frac{\widehat{P}_{f,k}(i_1; w) 
     \widehat{P}_{f,i_1}(i_2; w) \times\cdots\times 
     \widehat{P}_{f,i_m}(n; w)}{ 
     (1-w^{i_1}) \times\cdots\times (1-w^{i_m})} 
     \right]\Biggr) 
\end{align*} 
When $k = n$, we have that 
$$u_{n,n}(f, w) = \frac{1-w}{w(1-w^n) f(1)}.$$ 
\end{cor} 
\begin{proof}
This follows inductively from the inversion relation between the coefficients of 
a matrix and its inverse. For any $n \times n$ invertible lower triangular 
matrix $(a_{i,j})_{1 \leq i,j \leq n}$, we can express a 
non-recursive formula for the inverse matrix entries as follows: 
\begin{align} 
\label{eqn_lt_invmatrix_formula_v1} 
a_{n,k}^{(-1)} & = \frac{1}{a_{n,n}} \times \begin{cases}
     -\frac{a_{n,k}}{a_{k,k}} + \sum\limits_{m=1}^{n-k-1} 
     \sum\limits_{k \leq i_1 < \cdots < i_m < n} \frac{(-1)^{m+1} 
     a_{i_1,k} a_{i_2,i_1} \cdots 
     a_{i_m,i_{m-1}} a_{n,i_m}}{a_{k,k} a_{i_1,i_1} \cdots 
     a_{i_m,i_m}}, & \text{ if $k < n$; } \\ 
     1, & \text{ if $k = n$; } \\ 
     0, & \text{ otherwise.} 
     \end{cases}
\end{align} 
The proof of our result is then just an application of the formula in 
\eqref{eqn_lt_invmatrix_formula_v1} when $a_{n,k} := u_{n,k}^{-1}(f, w)$. 
While the identity in \eqref{eqn_lt_invmatrix_formula_v1} is not immediately obvious 
from the known inversion formulas between inverse matrices in the form of 
\[
a_{n,k}^{(-1)} = \frac{\Iverson{n=k}}{a_{n,n}} - \frac{1}{a_{n,n}} \times 
     \sum_{j=1}^{n-k-1} a_{n,j} a_{j,k}^{(-1)}, 
\]
the result is easily obtained by induction on $n$ so we do not prove it here. 
\end{proof} 

\SubsubsectionGTThesisFormatted{Formulas for simplified variants of the ordinary matrices} 
\label{subSection_formulas_for_ordinary_matrices} 

In Corollary \ref{cor_unkfw_ord_matrix_formula_v1} we proved an exact 
expansion of the ordinary matrix entries $u_{n,k}(f, w)$ by sums of 
weighted products of the inverse matrices $u_{n,k}^{(-1)}(f, w)$ that is expressed in 
closed form through Proposition \ref{prop_unk_inverse_matrix}. 
We will now develop the machinery needed to more precisely 
express the ordinary forms of these matrices for 
general cases of the indeterminate indexing parameter $w \in \mathbb{C} \setminus \{0\}$. 

\begin{table}[ht!] 

\caption[Symbolic factorization matrices for type II GCD sums]{
	 The matrix entries $\widehat{u}_{n,k}(f, w)$ for $1 \leq n,k \leq 6$ where 
	 $\hat{f}(n) = \frac{f(n) w^n}{w^n-1}$ for any arithmetic function $f$ such that $f(1) \neq 0$.}
\label{table_simplifies_unkfw} 

\begin{equation*} 
\boxed{
\begin{array}{lllll}
 \frac{1}{\widehat{f}(1)} & 0 & 0 & 0 & 0 \\
 -\frac{\widehat{f}(2)}{\widehat{f}(1)^2}-\frac{1}{\widehat{f}(1)} & \frac{1}{\widehat{f}(1)} & 0 & 0 & 0 \\
 \frac{\widehat{f}(2)}{\widehat{f}(1)^2}-\frac{\widehat{f}(3)}{\widehat{f}(1)^2}-\frac{1}{\widehat{f}(1)} & -\frac{1}{\widehat{f}(1)} & \frac{1}{\widehat{f}(1)} & 0 & 0 \\
 \frac{\widehat{f}(2)^2}{\widehat{f}(1)^3}+\frac{\widehat{f}(2)}{\widehat{f}(1)^2}+\frac{\widehat{f}(3)}{\widehat{f}(1)^2}-\frac{\widehat{f}(4)}{\widehat{f}(1)^2} &
   -\frac{\widehat{f}(2)}{\widehat{f}(1)^2}-\frac{1}{\widehat{f}(1)} & -\frac{1}{\widehat{f}(1)} & \frac{1}{\widehat{f}(1)} & 0 \\
 -\frac{\widehat{f}(2)^2}{\widehat{f}(1)^3}+\frac{\widehat{f}(3)}{\widehat{f}(1)^2}+\frac{\widehat{f}(4)}{\widehat{f}(1)^2}-\frac{\widehat{f}(5)}{\widehat{f}(1)^2} &
   \frac{\widehat{f}(2)}{\widehat{f}(1)^2} & -\frac{1}{\widehat{f}(1)} & -\frac{1}{\widehat{f}(1)} & \frac{1}{\widehat{f}(1)} \\
-\frac{\widehat{f}(2)^2}{\widehat{f}(1)^3}+\frac{2 \widehat{f}(3)
   \widehat{f}(2)}{\widehat{f}(1)^3}+\frac{\widehat{f}(4)}{\widehat{f}(1)^2}+\frac{\widehat{f}(5)}{\widehat{f}(1)^2}-\frac{\widehat{f}(6)}{\widehat{f}(1)^2}+\frac{1}{\widehat{f}(1)} &
   \frac{\widehat{f}(2)}{\widehat{f}(1)^2}-\frac{\widehat{f}(3)}{\widehat{f}(1)^2} & 
   -\frac{\widehat{f}(2)}{\widehat{f}(1)^2} & -\frac{1}{\widehat{f}(1)} &
   -\frac{1}{\widehat{f}(1)}
\end{array}
}
\end{equation*} 

\end{table}  

\begin{remark}[Simplifications of the matrix terms] 
\label{remark_simplified_unkfw} 
Using the formula for the coefficients of $u_{n,k}(f, w)$ in 
\eqref{eqn_intro_two_sum_fact_types_def_v2} expanded by 
\eqref{eqn_lfgnw_sum_divsum_exp_v1}, we can simplify the 
form of the matrix entries we seek closed-form expressions for in the next 
calculations. In particular, we make the following definitions for 
$1 \leq k \leq n$: 
\begin{align*} 
\widehat{f}(n) & := \frac{w^n}{w^n-1} f(n) \\ 
\widehat{u}_{n,k}(f, w) & := (w^k - 1) u_{n,k}(f, w). 
\end{align*} 
Then an equivalent formulation of finding the exact formulas for $u_{n,k}(f, w)$ is 
to find exact expressions expanding the triangular sequence of 
$\widehat{u}_{n,k}(f, w)$ satisfying 
\[
\sum_{\substack{j \geq 0 \\ n-G_j > 0}} (-1)^{\ceiling{\frac{j}{2}}} g(n-G_j) = 
     \sum_{k=1}^{n} \widehat{u}_{n,k}(f, w) \times \sum_{d|k} 
     \widehat{f}(d) g\left(\frac{n}{d}\right). 
\] 
We will obtain precisely such formulas in the next few results. 
Table \ref{table_simplifies_unkfw} provides the first few rows of our 
simplified matrix entries. 
\end{remark}

\begin{table}[ht!] 

\caption[The multiple convolution function $D_f(n)$]{
	 The function $D_f(n)$ for $2 \leq n \leq 11$ where 
	 $\widehat{f}(n) := \frac{f(n) w^n}{w^n-1}$ for an 
         arbitrary arithmetic function $f$ such that $f(1) \neq 0$.} 
\label{table_cases_of_Dn} 

\centering
\begin{tabular}{|c|l|c|l|} \hline 
$n$ & $D_f(n)$ & $n$ & $D_f(n)$ \\ \hline 
2 & $-\frac{\widehat{f}(2)}{\widehat{f}(1)^2}$ & 7 & 
    $-\frac{\widehat{f}(7)}{\widehat{f}(1)^2}$ \\ 
3 & $-\frac{\widehat{f}(3)}{\widehat{f}(1)^2}$ & 8 & 
    $\frac{2 \widehat{f}(2) \widehat{f}(4)-\widehat{f}(1)
    \widehat{f}(8)}{\widehat{f}(1)^3}-\frac{\widehat{f}(2)^3}{\widehat{f}(1)^4}$ \\ 
4 & $\frac{\widehat{f}(2)^2-\widehat{f}(1)
    \widehat{f}(4)}{\widehat{f}(1)^3}$ & 9 & 
    $\frac{\widehat{f}(3)^2-\widehat{f}(1) \widehat{f}(9)}{\widehat{f}(1)^3}$ \\ 
5 & $-\frac{\widehat{f}(5)}{\widehat{f}(1)^2}$ & 10 & 
    $\frac{2\widehat{f}(2) \widehat{f}(5)-\widehat{f}(1)
   \widehat{f}(10)}{\widehat{f}(1)^3}$ \\ 
6 & $\frac{2 \widehat{f}(2) \widehat{f}(3)-\widehat{f}(1)
    \widehat{f}(6)}{\widehat{f}(1)^3}$ & 11 & 
    $-\frac{\hat{f}(11)}{\hat{f}(1)^2}$ 
    \\ \hline 
\end{tabular} 

\end{table} 

\begin{definition}[Special forms of multiple convolutions] 
\label{def_djn_Dn_func_defs} 
For $n,j \geq 1$, we define the following nested 
$j$-convolutions of the function $\widehat{f}(n)$ 
\cite{MERCA-SCHMIDT-RAMJ}: 
\begin{align*} 
\ds_j(f; n) = \begin{cases} 
     (-1)^{\delta_{n,1}} \widehat{f}(n), & \text{ if $j = 1$; } \\ 
     \sum\limits_{\substack{d|n \\ d>1}} \widehat{f}(d) \ds_{j-1}\left(f; \frac{n}{d}\right), & 
     \text{ if $j \geq 2$. } 
     \end{cases} 
\end{align*} 
Then we define our primary multiple convolution function of interest as 
\[
D_f(n) := \sum_{j=1}^n \frac{\ds_{2j}(f; n)}{\widehat{f}(1)^{2j+1}}. 
\]
\end{definition}

The first few cases of $D_f(n)$ for $2 \leq n \leq 16$ are 
computed in Table \ref{table_cases_of_Dn}. 
The examples in the table should clarify precisely what multiple convolutions 
we are defining by the function $D_f(n)$. Namely, a signed sum of all possible 
ordinary $k$ Dirichlet convolutions of $\widehat{f}$ with itself evaluated at $n$.

\begin{lemma} 
\label{lemma_Dastfhat} 
We claim that for all $n \geq 1$
\[
(D_f \ast \widehat{f})(n) \equiv \sum_{d|n} f(d) D_f\left(\frac{n}{d}\right) = 
     -\frac{\widehat{f}(n)}{\widehat{f}(1)} + \varepsilon(n). 
\]
where $\varepsilon(n) \equiv \delta_{n,1}$ is the multiplicative identity function 
with respect to Dirichlet convolution. 
\end{lemma} 
\begin{proof}[Proof of Lemma \ref{lemma_Dastfhat}]
The statement of the lemma is equivalent to showing that 
\begin{equation} 
\label{eqn_fhat_Dinv_formula_exp_v1} 
\left(D_f + \frac{\varepsilon}{\widehat{f}(1)}\right)(n) = \widehat{f}^{-1}(n). 
\end{equation} 
A general recursive formula for the inverse of $\widehat{f}(n)$ is given by 
\cite{APOSTOLANUMT} 
\[
\widehat{f}^{-1}(n) = \left(-\frac{1}{\widehat{f}(1)} \times \sum_{\substack{d|n \\ d > 1}} 
     \widehat{f}(d) \widehat{f}^{-1}\left(\frac{n}{d}\right)\right) 
     \Iverson{n > 1} + \frac{\Iverson{n = 1}}{\widehat{f}(1)}. 
\]
This definition is almost how we defined $\ds_j(f; n)$ above. Let's see how to modify this 
recurrence relation to obtain the formula for $D_f(n)$. 
We can recursively substitute in the formula for $\widehat{f}^{-1}(n)$ until we hit the 
point where successive substitutions only leave the base case of 
$\widehat{f}^{-1}(1) = \widehat{f}(1)^{-1}$. This occurs after $\Omega(n)$ substitutions 
where $\Omega(n)$ denotes the number of prime factors of $n$ counting multiplicity. 
We can write the nested formula for $\ds_j(f; n)$ as 
\[
\ds_j(f; n) = \widehat{f}_{\pm} \ast 
\undersetbrace{j-1\text{ factors}}{\left(\widehat{f}-\widehat{f}(1) \varepsilon\right) 
     \ast \cdots 
     \ast \left(\widehat{f}-\widehat{f}(1) \varepsilon\right)}(n), 
\]
where we define 
$\widehat{f}_{\pm}(n) := \widehat{f}(n) \Iverson{n > 1} - \widehat{f}(1) \Iverson{n = 1}$. 
Next, define the nested $k$-convolutions $C_k(n)$ recursively by 
\label{page_eqn_Ckn_kCvls} 
\[
C_k(n) = \begin{cases} 
     \widehat{f}(n) - \widehat{f}(1)\varepsilon(n), & \text{ if $k = 1$; } \\ 
     \sum\limits_{d|n} \left(\widehat{f}(d) - \widehat{f}(1) \varepsilon(d)\right) 
     C_{k-1}\left(\frac{n}{d}\right), & \text{ if $k \geq 2$. } 
     \end{cases} 
\] 
Then we can express the inverse of $\widehat{f}(n)$ using this definition as follows: 
\[
\widehat{f}^{-1}(n) = \sum_{d|n} \widehat{f}(d)\left( 
     \sum_{j=1}^{\Omega(n)} \frac{C_{2k}\left(\frac{n}{d}\right)}{\widehat{f}(1)^{\Omega(n)+1}} - 
     \frac{\varepsilon\left(\frac{n}{d}\right)}{\widehat{f}(1)^2}\right). 
\]
Then based on the initial conditions for $k = 1$ (or $j = 1$) in the definitions of 
$C_k(n)$ (and $\ds_j(f; n)$), we see that the function in 
\eqref{eqn_fhat_Dinv_formula_exp_v1} is in fact the inverse of $\widehat{f}(n)$. 
\end{proof}

\begin{prop} 
\label{prop_exact_pnrec_for_hatunk} 
For all $n \geq 1$ and $1 \leq k \leq n$, we have that 
\[
\sum_{i=0}^{n-1} p(i) \widehat{u}_{n-i,k}(f, w) = D_f\left(\frac{n}{k}\right) 
     \Iverson{n \equiv 0 \pmod{k}} + \frac{\Iverson{n = k}}{\widehat{f}(1)}. 
\] 
\end{prop} 
\begin{proof}[Proof of Proposition \ref{prop_exact_pnrec_for_hatunk}] 
We notice that Lemma \ref{lemma_Dastfhat} implies that 
\[
\varepsilon(n) = \left(\left(D_f + \frac{\varepsilon}{\widehat{f}(1)}\right) \ast 
     \widehat{f}\right)(n), 
\] 
where $\varepsilon(n)$ is the multiplicative identity for Dirichlet convolutions. 
The last equation implies that 
\[
\tag{i} 
g(n) = \left(\left(D_f + \frac{\varepsilon}{\widehat{f}(1)}\right) \ast 
     \widehat{f} \ast g\right)(n).  
\] 
Additionally, we know by the expansion of 
\eqref{eqn_intro_two_sum_fact_types_def_v2} and that 
$\widehat{u}_{n,n}(f, w) = \widehat{f}(1)^{-1}$ that we have the expansion 
\[
\tag{ii} 
g(n) = \sum_{k \geq 1} \left(\sum_{j=0}^{n-1} p(j) \widehat{u}_{n-j,k}\right) 
     \sum_{d|k} \widehat{f}(d) g\left(\frac{k}{d}\right). 
\] 
So we can equate (i) and (ii) to see that 
\[
\sum_{j=0}^{n-1} p(j) \widehat{u}_{n-j,k} = D_f\left(\frac{n}{k}\right) \Iverson{k | n} + 
     \frac{\Iverson{n=k}}{\widehat{f}(1)}. 
\] 
This establishes our claim. 
\end{proof} 

\begin{cor}[An exact formula for the ordinary matrices] 
\label{cor_exact_formula_for_hatunkfw} 
For all $n \geq 1$ and $1 \leq k \leq n$ 
\[
\widehat{u}_{n,k}(f, w) = \sum_{\substack{j \geq 0 \\ n-G_j > 0}} 
     (-1)^{\ceiling{\frac{j}{2}}} \left(D_f\left(\frac{n-G_j}{k}\right) 
     \Iverson{n-G_j \equiv 0 \bmod{k}} + \frac{\Iverson{n-G_j = k}}{\widehat{f}(1)} 
     \right). 
\]
\end{cor} 
\begin{proof} 
This is an immediate consequence of Proposition \ref{prop_exact_pnrec_for_hatunk} 
by noting that the generating function for $p(n)$ is $(q; q)_{\infty}^{-1}$ and that 
\[
(q; q)_{\infty} = \sum_{j \geq 0} (-1)^{\ceiling{\frac{j}{2}}} q^{G_j}. 
     \qedhere
\]
\end{proof} 

\SubsectionGTThesisFormatted{Applications to DTFTs and finite Fourier series expansions} 
\label{subSection_DFT_exps} 

We expand the left-hand-side function 
$g(x)$ in \eqref{eqn_intro_two_sum_fact_types_def_v2} by considering a new 
indirect method involving the type II sums $L_{f,g,k}(n)$. The expansions we derive in this section 
employ results for discrete Fourier 
transforms of functions of the greatest common divisor studied in 
\cite{kamp-gcd-transform,SCHRAMM}. 
This method allows us to study the factorization forms in 
\eqref{eqn_intro_two_sum_fact_types_def_v2} where we effectively bypass the complicated forms of the 
ordinary matrix coefficients $u_{n,k}(f, w)$. Results enumerating the ordinary matrices with coefficients 
given by $u_{n,k}(f, w)$ are treated in 
Corollary \ref{cor_unkfw_ord_matrix_formula_v1} of 
Section \ref{subSection_typeII_inv_matrices}. 

The discrete Fourier series methods we use to prove our theorems in these sections 
lead to the next key result proved in Theorem \ref{theorem_main_lfgmn_exp_v1}
which states that for any arithmetic functions $f,g$ we have 
\[
\sum_{d|k} \sum_{r=0}^{k-1} d L_{f,g,r}(k) \e{-\frac{rd}{k}} \mu\left(\frac{k}{d}\right) = 
  \sum_{d|k} \phi(d) f(d) \left(\frac{k}{d}\right)^2 g\left(\frac{k}{d}\right), k \geq 1, 
\] 
where $\e{x} = \exp(2\pi\imath \cdot x)$ is standard notation for the 
complex exponential function. 

The proof of the result given in 
Theorem \ref{theorem_main_lfgmn_exp_v1} below builds on several results on 
discrete Fourier transforms of functions evauated at the greatest common divisor, 
$(k, n) \equiv \gcd(k, n)$, developed in \cite{kamp-gcd-transform}. 
For the remainder of this section we take $k \geq 1$ to be fixed and 
consider the divisor sums of the following form which are periodic with respect to 
$k$ for any $n \geq 1$: 
\[
L_{f,g,k}(n) := \sum_{d|(n,k)} f(d) g\left(\frac{n}{d}\right). 
\]
In \cite{kamp-gcd-transform} these sums are called $k$-convolutions of $f$ and $g$. 
We will first need to discuss some more terminology related to discrete Fourier 
transforms before moving on. 

\begin{definition}
\label{page_DTFT_coefficient_defs}
A \emph{discrete Fourier transform} (DFT) maps a (finite) sequence of complex numbers 
$\{f[n]\}_{n=0}^{N-1}$ onto their associated Fourier coefficients 
$\{F[n]\}_{n=0}^{N-1}$ defined according to the following reversion 
formulas relating these sequences: 
\begin{align*} 
F[k] & = \phantom{\frac{1}{N}\times} \sum_{n=0}^{N-1} f[n] e\left(-\frac{kn}{N}\right), \\ 
f[k] & = \frac{1}{N} \times \sum_{n=0}^{N-1} F[k] e\left(\frac{kn}{N}\right). 
\end{align*} 
The discrete Fourier transform of functions of the greatest common 
divisor, which we will employ 
repeatedly to prove Theorem \ref{theorem_main_lfgmn_exp_v1} below,  
is characterized by the formula in the next lemma 
\cite{kamp-gcd-transform,SCHRAMM}. 
\end{definition}

\begin{lemma} 
If we take any two arithmetic functions $f$ and $g$, 
we can express periodic divisor sums modulo any $k \geq 1$ of the form 
\begin{subequations} 
\label{eqn_snk_akm_intro_Fourier_coeff_exps}
\begin{align} 
s_k(f, g; n) & := \sum_{d|(n,k)} f(d) g\left(\frac{k}{d}\right) 
     = \sum_{m=1}^k a_k(f, g; m) e^{\frac{2\pi\imath mn}{k}}, n \geq 1. 
\end{align} 
The discrete Fourier coefficients on the right-hand-side 
of the previous equation are given by  
\begin{equation}
a_k(f, g; m) = \sum_{d|(m,k)} g(d) f\left(\frac{k}{d}\right) \frac{d}{k}. 
\end{equation} 
\end{subequations}
\end{lemma}
\begin{proof}
For a proof of these relations consult the references 
\cite[\S 8.3]{APOSTOLANUMT} \cite[\cf \S 27.10]{NISTHB}. 
These relations are also related to the gcd-transformations proved in 
\cite{kamp-gcd-transform,SCHRAMM}.
\end{proof} 

\begin{definition}
The function $c_m(a)$ defined for integers $m,a \geq 1$ by 
\[
c_m(a) := \sum_{\substack{k=1 \\ (k, m) = 1}}^{m} e\left(\frac{ka}{m}\right), 
\] 
is called \emph{Ramanujan's sum}. Ramanujan's sum is expanded as in the divisor sums in 
Corollary \ref{cor_RamSumNewFormula} of the last subsection. 
\end{definition}

\begin{lemma}[DFT of functions of the greatest common divisor] 
Let $h$ be any arithmetic function. 
For natural numbers $m \geq 1$, the discrete Fourier transform (DFT) of $h$ in the GCD sense 
is defined by the following function: 
\[
\widehat{h}[a](m) := \sum_{k=1}^m h\left(\gcd(k, m)\right) 
     e\left(\frac{ka}{m}\right). 
\] 
This form of the DFT of $h(\gcd(n, k))$ (for $k \geq 1$ a free parameter) satisfies 
$\widehat{h}[a] = h \ast c_{-}(a)$. 
This function is summed explicitly at any 
$n \geq 1$ as the Dirichlet convolution 
\[
\widehat{h}[a](n) = (h \ast c_{-}(a))(n) = 
\sum_{d|n} h\left(\frac{n}{d}\right) c_d(a), \text{ for } n \geq 1, a \in \mathbb{C}. 
\]
\end{lemma} 

\begin{definition}[Notation for certain exponential sums] 
\label{def_NotationAndSpecialExpSums} 
In what follows, for $1 \leq \ell \leq k$ 
we denote the $\ell^{th}$ Fourier coefficient with respect to $k$ 
of the function $L_{f,g,k}(n)$ by $a_{k,\ell}$. This is well defined since 
$L_{f,g,k}(n) = L_{f,g,k}(n+k)$ is periodic with period $k$ over the 
integers $n \geq 1$. We have an expansion of this function of the form 
\[
L_{f,g,k}(n) = \sum_{\ell=0}^{k-1} a_{k,\ell} \times \e{\frac{\ell n}{k}}. 
\] 
We can compute the coefficients, $a_{m,\ell}$, 
directly from $L_{f,g,k}(n)$ according to the formula 
\[
a_{k,\ell} = \sum_{n=0}^{k-1} L_{f,g,k}(n) \times \e{-\frac{\ell n}{k}}. 
\] 
These Fourier coefficients are given explicitly in terms of 
$f$ and $g$ by the formulas cited in 
\eqref{eqn_snk_akm_intro_Fourier_coeff_exps}. 
\end{definition} 

\begin{theorem}
\label{theorem_main_lfgmn_exp_v1} 
For any arithmetic functions $f,g$ and $k \geq 1$, we have that 
\begin{equation}
\label{eqn_ThmMainLfgkn_exp}
\sum_{d|k} \sum_{r=0}^{k-1} d L_{f,g,r}(k) \times 
     \e{-\frac{rd}{k}} \mu\left(\frac{k}{d}\right) = 
     \sum_{d|k} \phi(d) f(d) \left(\frac{k}{d}\right)^2 g\left(\frac{k}{d}\right),
\end{equation}  
where $\phi(n)$ is Euler's totient function. 
\end{theorem}
\begin{proof}[Proof of Theorem \ref{theorem_main_lfgmn_exp_v1}] 
We see that the left-hand-side of 
\eqref{eqn_ThmMainLfgkn_exp} corresponds to a divisor sum of the form 
\begin{align*} 
\sum_{d|k} \sum_{r=0}^{k-1} d L_{f,g,r}(k) \times 
     \e{-\frac{rd}{k}} \mu\left(\frac{k}{d}\right) & = 
     \sum_{d|k} d \cdot \hat{a}_{k,d} \times \mu\left(\frac{k}{d}\right). 
\end{align*} 
The Fourier coefficients $a_{k,d}$ in this expansion are given by 
\eqref{eqn_snk_akm_intro_Fourier_coeff_exps} 
\cite[\S 8.3]{APOSTOLANUMT} so that \cite[\S 27.10]{NISTHB} 
\begin{align*}
\sum_{0 \leq r < k} L_{f,g,r}(k) \e{-\frac{rd}{k}} & = 
     \sum_{0 \leq r < k} \sum_{\ell=0}^{r-1} a_{k,\ell} \times 
     \e{\frac{\ell r}{k}} \e{-\frac{rd}{k}}, 
\end{align*} 
and where whenever we have that $k|d$, we get that the exponential sum 
\[
\sum_{0 \leq r < k} \e{-\frac{rd}{k}} = k. 
\]
Then we have that 
\[
\hat{a}_{k,d} = k \times \sum_{r|(k,d)} g(r) f\left(\frac{k}{r}\right) \frac{r}{k}. 
\]
The left-hand-side of our expansion then becomes 
(\cf \eqref{eqn_sum_over_divsum_interchange_exp_ident} below) 
\begin{align*} 
\sum_{d|k} d \cdot a_{k,d} \mu\left(\frac{k}{d}\right) & = 
     \sum_{d|k} \sum_{r|d} dr \cdot g(r) f\left(\frac{k}{r}\right) 
     \mu\left(\frac{k}{d}\right) \\ 
     & = 
     \sum_{d=1}^k d \cdot \mu\left(\frac{k}{d}\right) 
     \left(\sum_{r|d} r \cdot g(r) f\left(\frac{k}{r}\right)\right) \Iverson{d|k} \\ 
     & = 
     \sum_{r|k} r \cdot g(r) f\left(\frac{k}{r}\right) 
     \left(\sum_{d=1}^{\frac{k}{r}} dr \cdot \mu\left(\frac{k}{dr}\right) \Iverson{d|k}\right) \\ 
     & = 
     \sum_{r|k} r^2 \cdot g(r) f\left(\frac{k}{r}\right) 
     \phi\left(\frac{k}{r}\right). 
\end{align*} 
The conclusion follows by interchanging the index of summation by setting 
$d \leftrightarrow \frac{n}{d}$ and vice versa as inputs to the functions 
in the last equation.
\end{proof} 

\begin{cor} 
\label{cor_exact_formula_forn_gn} 
For any $n \geq 1$ and arithmetic functions $f,g$ we have the formula 
$$g(n) = \sum_{d|n} \sum_{j|d} \sum_{r=0}^{d-1} \frac{j \cdot L_{f,g,r}(d)}{d^2} \times 
  \e{-\frac{rj}{d}} \mu\left(\frac{d}{j}\right) y_f\left(\frac{n}{d}\right),$$
where $y_f(n) = (f \phi \Id_{-2})^{-1}(n)$ and $\operatorname{Id}_k(n) := n^k$ for any 
fixed integer $k$ and $n \geq 1$.  
\end{cor} 
\begin{proof}[Proof of Corollary \ref{cor_exact_formula_forn_gn}] 
We first divide both sides of the result in 
Theorem \ref{theorem_main_lfgmn_exp_v1} 
by $k^2$. 
Then we apply a Dirichlet convolution of the left-hand-side of the formula in 
Theorem \ref{theorem_main_lfgmn_exp_v1} with 
$y_f(n)$ defined as above to obtain the exact expansion for $g(n)$. 
\end{proof} 

\begin{cor}[The Mertens function] 
\label{cor_Mertens_function_v2} 
For all $x \geq 1$, the Mertens function, denoted by the partial sums 
$M(x) := \sum_{n \leq x} \mu(n)$, is expanded by Ramanujan's sum as 
\begin{align} 
\label{eqn_Mx_DFT_exp_v1} 
M(x) & = \sum_{d=1}^x \sum_{n=1}^{\floor{\frac{x}{d}}} \sum_{r=0}^{d-1} 
     \sum_{j|d} \frac{j \cdot c_d(r) y_M(n)}{d^2} \times \e{-\frac{rj}{d}} 
     \mu\left(\frac{d}{j}\right), 
\end{align} 
where $y_M(n) = (\phi \Id_{-1})^{-1}(n)$. 
\end{cor} 
\begin{proof}[Proof of Corollary \ref{cor_Mertens_function_v2}]
We begin by citing Theorem \ref{theorem_main_lfgmn_exp_v1} 
in the special case corresponding to 
$L_{f,g,k}(n)$ a Ramanujan sum for $f(n) = n$ and $g(n) = \mu(n)$. 
Then we sum over the left-hand-side $g(n)$ in 
Corollary \ref{cor_exact_formula_forn_gn} to obtain the 
initial summation identity for $M(x)$ given by 
\[
\tag{i} 
M(x) = \sum_{n \leq x} \sum_{d|n} \sum_{j|d} 
     \sum_{r=0}^{d-1} \frac{j}{d^2} \times c_d(r) \e{-\frac{rj}{d}} 
     \mu\left(\frac{d}{j}\right) y_M\left(\frac{n}{d}\right). 
\]
We can then apply the identity that for any arithmetic functions $h,u,v$ we can 
interchange nested divisor sums as
\begin{equation} 
\label{eqn_sum_over_divsum_interchange_exp_ident} 
\sum_{k=1}^n \sum_{d|k} h(d) u\left(\frac{k}{d}\right) v(k) = \sum_{d=1}^n h(d) \left( 
     \sum_{k=1}^{\floor{\frac{n}{d}}} u(k) v(dk)\right). 
\end{equation} 
Application of this identity to equation (i) leads to the 
first form for $M(x)$ stated in \eqref{eqn_Mx_DFT_exp_v1}. 
\end{proof} 

There is a related identity to compare to 
equation \eqref{eqn_sum_over_divsum_interchange_exp_ident} 
which allows us to interchange the order of 
summation in the Anderson-Apostol sums of the following form 
for any natural numbers $x \geq 1$ and arithmetic functions 
$f,g,h: \mathbb{N} \rightarrow \mathbb{C}$: 
\[
\sum_{d=1}^x f(d) \sum_{r|(d,x)} g(r) h\left(\frac{d}{r}\right) = 
       \sum_{r|x} g(r) \times \sum_{d=1}^{\frac{x}{r}} h(d) 
       f\left(\gcd(x,r) d\right). 
\]

\begin{cor}[Euler's totient function]
\label{cor_EulersTotientFunc_v5}
For any $n \geq 1$ we have 
\[
\frac{\phi(n)}{n} = \sum_{d|n} \sum_{j|d} \sum_{r=0}^{d-1} \frac{j}{d^2} \times c_d(r) 
     \e{-\frac{rj}{d}} \mu\left(\frac{d}{j}\right). 
\]
We have the following expansion of the average order sums for $\phi(n)$ 
given by 
\[
\sum_{2 \leq n \leq x} \phi(n) = \sum_{d=1}^x \sum_{r=0}^{d-1} \sum_{j|d}  
     j \cdot \e{-\frac{rj}{d}} \mu\left(\frac{d}{j}\right)  
     \frac{c_d(r)}{2d} \floor{\frac{x}{d}}\left(\floor{\frac{x}{d}}-1\right). 
\] 
\end{cor} 
\begin{proof}[Proof of Corollary \ref{cor_EulersTotientFunc_v5}]
We consider the formula in Theorem \ref{theorem_main_lfgmn_exp_v1} with 
$f(n) = n$ and $g(n) = \mu(n)$. Since the Dirichlet inverse of the M\"obius 
function is $\mu \ast \mathds{1} = \varepsilon$, we obtain our result by convolution and 
multiplication by the factor of $n$. The average order identity follows from the 
first expansion by applying \eqref{eqn_sum_over_divsum_interchange_exp_ident}. 
\end{proof} 

\newpage
\SectionGTThesisFormatted{Generalized factorization theorems} 

\SubsectionGTThesisFormatted{$K$-convolutions: A generalized form of Dirichlet convolutions} 
\label{subSection_GenSummatoryFuncIdents_KCvls}

\begin{definition}
\label{def_KCvlSumsDefs} 
Consider the next generalization of Dirichlet convolution, 
$(f \ast g)(n) := \sum_{d|n} f(d) g\left(\frac{n}{d}\right)$, 
or divisor sums with respect to some kernel function $K(n,d)$ that is well defined for all 
divisors $d|n$ of any $n \geq 1$ by the following $K$-convolution operation: 
\begin{align}
\label{eqn_KCvl_fg_def_v1}
(f \ast_K g)(n) := \sum_{d|n} f(d) g\left(\frac{n}{d}\right) K(n, d). 
\end{align}
Typically, we restrict the function $K$ so that 
$K(n, d) = K\left(n, \frac{n}{d}\right)$ to assure commuativity of 
the associated convolution operation as $f \ast_K g = g \ast_K f$, though this requirement is 
not a strict necessity for applications. 
Furthermore, we assume that any admissible kernel function $K$ (subject to the requirements of 
our definitions above) can be expressed in factored form by $K(n, d) = K_0(n,d) \Iverson{d|n}$ 
for some other suitably defined kernel function $K_0$ that characterizes the operation. 
For any fixed finite $1 \leq N < \infty$, we also use the notation 
\[
\mathcal{K}_0(N) := (K_0(n,d))_{1 \leq n,d \leq N}, \mathcal{K}(N) := (K(n,d))_{1 \leq n,d \leq N}. 
\]
We note that the sums of this type are expressed by matrix-vector type convolutions defined in 
Example \ref{example_KCvlSumsExpressedByTopelitzProductConstructions} in a later section.
\end{definition} 

\begin{prop}[Generalized M\"obius inversion] 
\label{prop_GenInversionForKCvl_v1} 
For any kernel functions $(K, K_0)$ associated to the definition given in 
\eqref{eqn_KCvl_fg_def_v1} above such that $K_0$ is invertible, 
and any arithmetic functions $f,g$ we have that 
\[
g(n) = \sum_{d|n} f(d) K_0(n,d) \iff f(n) = 
     \sum_{d|n} \sum_{r|d} g(r) \mu\left(\frac{d}{r}\right) K_0^{-1}(n, d), 
     n \geq 1. 
\]
\end{prop} 
\begin{proof}[Proof of Proposition \ref{prop_GenInversionForKCvl_v1}]
Notice that for fixed $N \geq 1$, with the $N \times N$ matrix 
$T_{\operatorname{div}}(N) := (\Iverson{j|i})_{1 \leq i,j \leq N}$ 
invertible via the M\"obius inversion theorem, we 
can write a matrix-vector product system of the form 
\[
\begin{bmatrix} f(1) \\ f(2) \\ \vdots \\ f(N)\end{bmatrix} = \mathcal{K}_0(N) T_{\operatorname{div}}(N) 
     \begin{bmatrix} g(1) \\ g(2) \\ \vdots \\ g(N)\end{bmatrix}. 
\]
By ordinary matrix inversion, the previous equation implies that for any $n \leq N$
\begin{align*}
     f(n) & = \sum_{d|n} \mu\left(\frac{n}{d}\right) \times \sum_{r=1}^d K_0^{-1}(d, r) g(r) \\ 
     & = \sum_{r=1}^{n} K_0^{-1}(n, r) \sum_{d|r} \mu\left(\frac{r}{d}\right) g(r). 
\end{align*}
Because we can restrict the non-trivial values of $\mathcal{K}_0^{-1}(N)$ to only those $(i,j)^{th}$ entries 
such that $j|i$, we obtain from the last equation that 
$f(n) = (K_0^{-1}(n, -) \ast \mu \ast g)(n)$ for all integers $n \geq 1$. 
\end{proof}

In Section \ref{Section_CanonicalReprsOfFactThms_KernelBasedDCVL}, 
we formulate properties that characterize a class of the most general expansions 
of weighted convolution type sums which we call $\mathcal{D}$-convolutions. 
Because the inversion relations we find in the divisor sum variants we have used to define the class of 
$K$-convlutions are simpler, and represented by products with the classical M\"obius function, 
we consider properties of the latter sums in this section before formalizing the general case 
constructions. 

\begin{theorem}[Factorization theorems for $K$-convolutions]
\label{theorem_InvertibleFactThmStmtsForKCvls_v1} 
For integers $n \geq 1$, let the triangular generating functions be defined as 
\[
\mathcal{H}_{K,n}(q) := \sum_{k \geq 1} K_0(n, k) q^{nk}. 
\]
Suppose that given a fixed $K$-convolution kernel pair 
$(K, K_0)$, and any arithmetic function $f$, 
we define the LGF analog for $K$-convolutions as 
\[
\mathcal{L}_{K,f}(q) := \sum_{n \geq 1} \mathcal{H}_{K,n}(q) f(n). 
\]
For all $n \geq 1$, $[q^n] \mathcal{L}_{K,f}(q) = (f \ast_K 1)(n)$. 
We define a factorized form of these generating functions in the form of 
\[
\mathcal{L}_{K,f}(q) := \prod_{j \geq 1} \mathcal{H}_{K,j}(q) \times 
     \sum_{n \geq 1} \left(\sum_{k=1}^n \hat{S}_K(n, k) f(k)\right) q^n, 
\]
so that we have the following results that for all integers $n \geq m \geq 1$: 
\begin{itemize} 
\item[(i)] $\hat{S}_K(n, m) = [q^m] \mathcal{H}_{K,m}(q) \times \prod_{j \geq 1} \mathcal{H}_{K,j}^{-1}(q)$; 
\item[(ii)] $\hat{S}_K^{-1}(n, m) = \sum\limits_{d|n} \sum\limits_{s|d} \mu\left(\frac{d}{s}\right) K_0(n, d) 
             \times [q^{s-r}] \prod\limits_{j \geq 1} \mathcal{H}_{K,j}$. 
\end{itemize}
\end{theorem}

\emph{Remarks on Theorem \ref{theorem_InvertibleFactThmStmtsForKCvls_v1}.} 
Some partition theoretic analogs that we should expect from the LGF factorization 
theorem cases appear in \cite[\S 14.10]{APOSTOLANUMT}. In general, the structure and exact combinatorial 
interpretation of the matrix coefficients, $\hat{s}_K(n, k)$, are more complicated at determined by the 
series coefficients of the products of the $\prod_{j \geq 1} \mathcal{H}_{K,j}(q)$. 
While we do not state a precise formula here, the analogs to the LGF factorization theorems for 
the Dirichlet convolutions $f \ast g$ with $g(1) \neq 0$ from 
Section \ref{Section_LSFactThms_CvlOfTwoFns_and_Apps}, 
must assume that the arithmetic function $g$ is invertible with respect to the operation of $K$-convolution. 
These generalized factorization theorem results, like the work proved in the preceeding section by Merca and Schmidt, 
involve expansions that are stated in terms of the inverse function of $g$ with respect to $K$-convolution. 
The prior LGF series cases we defined in 
Section \ref{Section_LGFFactTheorems_PriorWork}
correspond to choosing $K_0(n, k) \equiv 1$ for all $1 \leq k \leq n$. 
\begin{proof}[Proof of Theorem \ref{theorem_InvertibleFactThmStmtsForKCvls_v1}] 
To prove that the coefficients of $\mathcal{L}_{K,f}(n)$ are given by the $K$-concolution formula 
stated above, we consider the following expansions for any fixed $n \geq 1$: 
\begin{align*}
[q^n] \mathcal{L}_{K,f}(n) & = \sum_{j \geq 1} f(j) \times [q^n] \mathcal{H}_{K,n}(q), \\ 
     & = \sum_{j \geq 1} f(j) \left(\sum_{\substack{m \geq 1 \\ mj=n}} K_0(n, m)\right), \\ 
     & = \sum_{j|n} f(j) K_0\left(n, \frac{n}{j}\right), \\ 
     & = \sum_{j|n} f(j) K_0(n, j). 
\end{align*} 
The expansion in (i) follows from 
\begin{align*}
\hat{S}(n, m) & = [q^n][f(m)] \mathcal{L}_{K,f}(q) \times \prod_{j \geq 1} \mathcal{H}_{K,j}(q)^{-1}, \\ 
     & = [q^n] \mathcal{H}_{k,m} \times \prod_{j \geq 1} \mathcal{H}_{K,j}(q)^{-1}. 
\end{align*} 
The formula in (ii) is somewhat more complicated to show. For any fixed $r \geq 1$ and $n \geq r$, 
let the function $\bar{f}(n) := \hat{S}^{-1}_K(n, r)$. 
By orthogonality relations on the right multiplication of a matrix by its inverse, we see that 
\[
\sum_{k=1}^{n} \hat{S}_K(n, k) \bar{f}(k) = \Iverson{n = r}, n \geq 1. 
\]
The last equation implies that 
\begin{align*} 
(\bar{f} \ast_{K} 1)(n) & = \sum_{d|n} \hat{S}^{-1}_K(d, r) K_0(n, d), \\ 
     & = [q^n] \mathcal{L}_{K,\bar{f}}(q) 
     = [q^n] q^r \times \prod_{j \geq 1} \mathcal{H}_{K,j}(q). 
\end{align*} 
The inversion theorem in Proposition \ref{prop_GenInversionForKCvl_v1} 
applied to the local functions $f$ and $g$ given by 
$f(n) \equiv \bar{f}(n) = \hat{S}^{-1}(n, r)$ and $g(n) \equiv [q^{n-r}] \prod_{j \geq 1} \mathcal{H}_{K,j}(q)$ 
implies the formula in (ii) correct. 
\end{proof}

\SubsectionGTThesisFormatted{Topelitz matrix constructions to express discrete convolution sum types}

\nocite{MATRIX-COOKBOOK,TOPELITZ-GRAY}

\begin{definition}[Topelitz matrices]
For any sequence $\{t_n\}_{n \geq 1} \subseteq \mathbb{C}$ such that $t_1 \neq 0$, we define its 
associated $N \times N$ \emph{Topelitz matrix} to be the lower triangular operator 
\[
T_N \equiv T_N(t) := \begin{bmatrix} t_1 & 0 & \cdots & 0 \\ t_2 & t_1 & \cdots & 0 \\ 
          \vdots & & \ddots \\ 
          t_N & \cdots & & t_1
     \end{bmatrix} 
\]
Topelitz matrices define a structure by which we can express the discrete convolution of any 
two sequences (arithmetic functions) via a matrix-vector product. 
Namely, we have that for any $N \times 1$ vector $\vec{g} := (g_1, \ldots, g_N)^{T}$ 
\[
\left(T_N(f) \cdot \vec{g}\right)_{n} = \sum_{k=1}^{n} f_k g_{n+1-k}. 
\]
\end{definition} 

\begin{example} 
\label{example_KCvlSumsExpressedByTopelitzProductConstructions} 
The $K$-convolution sums from 
Definition \ref{def_KCvlSumsDefs} 
are naturally and concisely expressed by Topelitz matrices in the form of 
\[
\left(\mathcal{K} \cdot T_N(f) \cdot \vec{g}\right)_n = 
     \sum_{d|n} f(d) g\left(\frac{n}{d}\right) K_0(n, d), 
\]
where for $K_0$ an invertible tranformation, we have that 
$\mathcal{K}^{-1} = \left(\mu\left(\frac{n}{d}\right) \Iverson{d|n}\right)_{1 \leq n,d \leq N} \cdot \mathcal{K}_0^{-1}$. 
Note that we have the special case of the Topelitz matrix for the constant unitary function 
given by $T_N(1) = (\Iverson{i \leq j})_{1 \leq i,j \leq N}$ and such that its 
inverse is expressed by shift matrices: 
\[
T_N(1)^{-1} \equiv \left(I-U^T\right)^{-1} = 
     (\delta_{i,j} - \delta_{i+1,j})_{1 \leq i,j \leq N}, N \geq 1. 
\]
\end{example}

\begin{remark}[M\"obius inversion formulas on posets]
It happens that we can draw comparisons with these summation types, their 
corresponding generating factorizations by invertible matrices, and other inversion identities to 
constructions that arise in \emph{incidence alegbras}, of which the classically defined 
$\mu(n)$ is the special case where $\mu(n,d) = \mu\left(\frac{n}{d}\right)$. 
We recall that Rota studied 
generalized forms of the M\"obius function for posets $\mathcal{P}$ endowed with a 
partial ordering $\leq$ \cite{ROTA}. Within this context, we define 
\[
\mu(s,s) = 1, \forall s \in \mathcal{P}, \mu(s,u) = -\sum_{s \leq t < u} \mu(s,t), \forall s<u \in \mathcal{P}, 
\]
where we have an inversion relation for all $t \in \mathcal{P}$ of the following form:
\[
g(t) = \sum_{s \leq t} f(s) \iff f(t) = \sum_{s \leq t} g(s) \mu(s,t). 
\]
\end{remark}

\SubsectionGTThesisFormatted{Functional equations for generating functions of triangular sequences} 

\begin{example}[The Binomial transform]
Suppose that we have a sequence, $\{f_n\}_{n \geq 0}$, and its (formal) 
\emph{ordinary generating function} (OGF) is given by $F(z) := \sum_{n \geq 0} f_n z^n$. 
Then a standard method for generating the summatory functions of the $f_n$ is to scale by a 
factor of the geometric series as 
\[
[z^n] \frac{F(z)}{1-z} = \sum_{n=0}^{x} f_n. 
\]
In fact, we can actually go farther with a so-called \emph{binomial transform} of generating functions 
to express 
\[
[z^n] \frac{F\left(-\frac{z}{1-z}\right)}{1-z} = \sum_{n=0}^{x} \binom{x}{n} (-1)^{n} f_n. 
\]
Other generating function transformations can be used to generate finite sums of a sequence scaled 
by another lower triangular sequence, such as the Stirling numbers using the 
\emph{Stirling transform} to note, in a similar manner. 
\end{example}

\begin{prop}
\label{prop_ClassicalFuncOGFEGFEqns_InitSurveyOfSpCases_v1}
Fix any sequence $\{f_n\}_{n \geq 1}$ and some lower triangular seguence $g_{n,k}$. 
Let the associated sequence of sums be defined by 
\[
S_{f}[g](n) := \sum_{k=1}^{n} g_{n,k} f_k, n \geq 1.
\]
Let the column generating functions of $g$ be defined as 
absolutely convergent series when $|q| < \sigma_{g,a}$ for some $\sigma_{g,a} > 1$ 
for integers $k \geq 1$ as follows:
\[
G_k(q) := \sum_{n \geq k} g_{n,k} q^n. 
\]
Suppose that the OGF of $S_f[g](n)$ is given by 
$$\widetilde{S}_f[g](q) := \sum_{n \geq 1} S_f[g](n) q^n,$$ 
and that $F(q), \widehat{F}(q)$ denote the OGF and EGF of $\{f_n\}_{n \geq 0}$, respectively. 
We have the following formulas for special case series types that provide 
explicit relations between these OGFs:
\begin{itemize}
\item[(A)] If $G_k(q) = H_1(q) H_2(q)^{k}$ for some fixed functions $H_i(q)$, then 
           \[
           \widetilde{S}_f[g](q) = H_1(q) H_2(2)\left(F\left(H_2(q)\right) - f_0\right). 
           \]
\item[(B)] If $G_k(q) = H_3(q) \frac{H_4(q)^k}{k!}$ for some component functions $H_i(q)$, 
           then 
           \[
           \widetilde{S}_f[g](q) = H_3(q) \left(\widehat{F}\left(H_4(q)\right) - f_0\right). 
           \]
\end{itemize} 
\end{prop}
\begin{proof}
The proofs are nearly trivial provided that the OGFs $G_k(q)$ are absolutely convergent 
so that we can interchange the order of summation. With this assumption, each respective claim 
follows by summing a geometric or exponential series. 
\end{proof}
     
\begin{definition}
The \emph{Hadamard product} of two ordinary generating functions $F(q)$ and $G(q)$, respectively 
enumerating the sequences of $\{f_n\}_{n \geq 0}$ and $\{g_n\}_{n \geq 0}$ is defined by 
\begin{align*}
(F \circ G)(q) & := \sum_{n \geq 0} f_n g_n q^n,\ \text{ for $|q| < \sigma_{c,f} \sigma_{c,g}$, } 
\end{align*} 
where $\sigma_{c,f}$ and $\sigma_G{c,g}$ 
denote the radii (abscissa) of convergence of each respective generating function. 
Analytically, we have an integral formula and corresponding coefficient extraction formula 
for the Hadamard product of two generating functions when $F(q)$ is expandable in a 
fractional (Pusieux) series 
respectively given by \cite[\S 1.12(V); Ex.\ 1.30, p.\ 85]{ADVCOMB} \cite[\S 6.3]{ECV2} 
\begin{align*} 
(F \circ G)(q^2) & = \frac{1}{2\pi} \int_0^{\pi} F(q e^{\imath t}) G(q e^{-\imath t}) dt \\ 
(F \circ G)(q) & = [x^0] F\left(\frac{q}{x}\right) G(x). 
\end{align*}
\end{definition}

\begin{theorem}[An integral formula for generalized sums]
\label{theorem_IntFormula_HProdForSummationTypeOGFs_v2}
We adopt the notation for the summation and column type OGFs from 
Proposition \ref{prop_ClassicalFuncOGFEGFEqns_InitSurveyOfSpCases_v1}. 
Let the bivariate generating function 
\[
G(w, q) := \sum_{k \geq 1} G_k(q) w^k. 
\]
We have that 
\[
S_f[g](n) = [q^n] \left(\frac{1}{2\pi} \int_{-\pi}^{\pi} F\left(\sqrt{q} e^{\imath t}\right) 
     G(1, \sqrt{q} e^{-\imath t}) dt\right). 
\]
\end{theorem}
\begin{proof}
The proof of this theorem is trivial. 
\end{proof} 

\begin{example} 
Suppose that we set $g_{n,k} := \Iverson{(n,k)=1} \Iverson{n \geq k}$. 
We have seen that 
\[
G_k(q) = q^k \times \sum_{d|k} \frac{\mu(d)}{1-q^{d}}. 
\]
Theorem \ref{theorem_IntFormula_HProdForSummationTypeOGFs_v2} implies that the sums 
\[
S_f[g](n) := \sum_{\substack{d \leq n \\ (d,n)=1}} f(d), 
\]
are generated as the coefficients of $q^n$ in the expansion 
\cite[\S 1]{TAOCPV1} 
\begin{align*}
\frac{1}{2\pi} \times \int_{-\pi}^{\pi} & \left(\sum_{n,k \geq 1} \sum_{d|k} 
     \frac{\mu(d) f_n \sqrt{q}^{n+k}}{1-\sqrt{q}^d e^{-\imath dt}} e^{\imath (n-k) t}\right) dt \\ 
     & = 
     \sum_{\substack{n,k \geq 1 \\ m \geq 0}} \sum_{d|k} \mu(d)) f_n \Iverson{n=md+k} \\ 
     & = \sum_{k \geq 1} \sum_{0 \leq m < k} \sum_{d|(k,k-m)} \mu(d) \times \sum_{n \geq 1} 
     f_{nk+m} q^{nk+m} \\ 
     & = 
     \sum_{k \geq 1} \sum_{\substack{1 \leq m \leq k \\ (m, k)=1}} \sum_{0 \leq r < k} 
     F\left(q e^{\frac{2\pi\imath r}{k}}\right) \frac{e^{\frac{2\pi\imath m}{k}}}{k}. 
\end{align*} 
This argument also shows that 
\[
\phi_m(n) = \sum_{1 \leq k \leq n} k^m\left(\sum_{d|(k,n-k)} \mu(d)\right), n \geq 1. 
\]
More generally, for any arithmetic function $f$ we have that 
\[
\sum_{\substack{d \leq n \\ (d,n)=1}} f(d) = \sum_{1 \leq k \leq n} f(k)\left(\sum_{d|(k,n-k)} \mu(d)\right), n \geq 1. 
\]
\end{example}

\SubsectionGTThesisFormatted{Definitions of generalized kernel-based discrete convolution type sums} 
\label{subSection_DCvls_v1}

For a bivariate kernel function 
$\mathcal{D}: \mathbb{N}^2 \rightarrow \mathbb{C}$, 
we define the next class of convolution type sums, or $\mathcal{D}$-convolution type sums, 
according to the formula 
\begin{equation}
\label{eqn_DCvlSummationBasedExpDefs_restated_v2} 
(f \boxdot_{\mathcal{D}} g)(n) := \sum_{k=1}^{n} f(k) g(n+1-k) \mathcal{D}(n, k), n \geq 1. 
\end{equation} 
In the section ahead, we are able to connect these 
special convolution type sums that form a widely reaching class of applications through 
particular specializations of the $(f, g; \mathcal{D})$. 

\begin{definition}
We say that a kernel, or weight function $\mathcal{D}: (\mathbb{Z}^{+})^2 \rightarrow \mathbb{C}$ is 
\emph{lower triangular} if $\mathcal{D}(n, k) = 0$ for all $k > n \geq 1$. 
We say that this kernel function is invertible provided that 
\[
\det\left[(\mathcal{D}(n,k))_{1 \leq n,k \leq N}\right] \neq 0, \forall N \geq 1. 
\]
Suppose that $\mathcal{D}$ is an invertible, lower triangular kernel function, 
and that the arithmetic function $g$ is invertible with respect to 
$\mathcal{D}$-convolution, i.e., defined such that $g(1) \neq 0$. Then we can express 
the OGF of these sums according to the following parameterized 
invertible matrix based factorizations for $n \geq 1$: 
\begin{equation}
\label{eqn_DCvlGenFactThmExp_restated_v2} 
(f \boxdot_{\mathcal{D}} g)(n) = [q^n] 
     \sum_{n \geq 1} \left(\sum_{k=1}^{n} s_{n,k}(g; \mathcal{C},\mathcal{D}) f(k) 
     \right) \frac{q^n}{\mathcal{C}(q)}, 
     \text{\ for\ } \mathcal{C}(0) \neq 0. 
\end{equation}
\end{definition}

\begin{prop}[Inversion]
Suppose that $\mathcal{D}$ is a kernel function for a convolution sequence defined in 
\eqref{eqn_DCvlSummationBasedExpDefs_restated_v2} 
that is both invertible and satisfies $\mathcal{D}(n,n) \neq 0$ for all $n \geq 1$. 
Let the corresponding lower triangular sequence of entries for its inverse matrix be denoted by 
$\mathcal{D}^{-1}(n,k)$. 
Then for any $n \geq 1$ we have that 
\[
g(n) = (f \boxdot_{\mathcal{D}} 1)(n) \iff f(n) = \sum_{k=1}^{n} g(k) \mathcal{D}^{-1}(n,k). 
\]
\end{prop}
\begin{proof}
By construction, we suppose that $1 \leq N < +\infty$ and that we have 
two $N$-dimensional vectors, $\vec{f} := [f(1), \ldots, f(N)]^{T}$ and 
$\vec{g} := [g(1), \ldots, g(N)]^{T}$. 
It follows that 
\[
\vec{g} = \mathcal{D}(N) \cdot \vec{f} \implies 
     \vec{f} = \mathcal{D}(N)^{-1} \cdot \vec{g}, 
\]
where $\mathcal{D}(N) := (\mathcal{D}(n, k))_{1 \leq n,k \leq N}$. 
\end{proof}

Suppose that $\mathcal{D}(n, k)$ is an invertible, lower triangular kernel function. 
We say that an arithmetic function $g$ is invertible with respect to $\mathcal{D}$-convolution 
if there exists a (left) inverse function $g^{-1}[\mathcal{D}](n)$ such that for all 
integers $n \geq 1$ we have that $(g^{-1}[\mathcal{D}] \boxdot_{\mathcal{D}} g)(n) = \delta_{n,1}$. 
We can restrict ourselves to the cases where we take $\mathcal{D}$ to be 
\emph{symmetric} with respect to $\mathcal{D}$-convolution:
That is, where we have that $\mathcal{D}(n, k) = \mathcal{D}(n, n+1-k)$ for all $1 \leq k \leq n$. 
In these cases, we have that the left and corresponding right inverse functions of any $g$ with respect to 
$\mathcal{D}$-convolution are identical when they exist.

\begin{prop}[Inverses with respect to $\mathcal{D}$-convolution] 
An arithmetic function $g$ is invertible with respect to $\mathcal{D}$-convolution for a fixed invertibly 
lower triangular and symmetric kernel $\mathcal{D}$ if and only if 
$g(1) \neq 0$. When the function $g^{-1}[\mathcal{D}](n)$ exists, it is unique, and 
can be computed exactly by recursion via the following formula: 
\[
g^{-1}[\mathcal{D}](n) = \begin{cases} 
     \frac{1}{\mathcal{D}(1, 1) g(1)}, & n = 1; \\ 
     -\frac{1}{\mathcal{D}(n, n) g(1)} \times \sum\limits_{1 \leq k < n} 
     g^{-1}[\mathcal{D}](k) g(n+1-k) \mathcal{D}(n, k), & n \geq 2. 
\end{cases} 
\]
Moreover, provided that $g^{-1}[\mathcal{D}]$ exists, we have that 
\[
g^{-1}[\mathcal{D}](n) = \mathcal{D}^{-1}(n, 1) \times 
     [q^{n-1}]\left(\sum_{n \geq 0} g(n+1) q^n\right)^{-1}, n \geq 1. 
\]
\end{prop}
\begin{proof}
Fix any arithmetic function $g$ with $g(1) \neq 0$. 
The recursive formula follows by a rearrangement of the terms in the 
equation $(g^{-1}[\mathcal{D}] \boxdot_{\mathcal{D}} g)(n) = \delta_{n,1}$. 
To prove the exact formula, we see that we can set up an invertible matrix vector system of 
the form $A_{\mathcal{D},g}(N) \cdot \vec{f} = \vec{b}$ and solve for $\vec{f}$ where 
$A_{\mathcal{D},g}(N) := (\mathcal{D}(n, k) g(n+1-k) \Iverson{k \leq n})_{1 \leq n,k \leq N}$, 
$\vec{f} = [g^{-1}[\mathcal{D}](1), \ldots, g^{-1}[\mathcal{D}](N)]^{T}$, and 
$\vec{b} = [1, 0, \ldots, 0]^{T}$ for any $N \geq 1$. 
Then we have that 
$$g^{-1}[\mathcal{D}](n) = (A_{\mathcal{D},g}^{-1}(N))_{n,1}, \forall N \geq n \geq 1.$$
Notice that for $\mathcal{D}(N) := (\mathcal{D}(n,k))_{1 \leq n,k \leq N}$, we can write 
\[
A_{\mathcal{D},g}(N) = \mathcal{D}(N) \cdot \begin{bmatrix} 
     g(1) & 0 & 0 & \cdots & 0 \\ 
     g(2) & g(1) & 0 & \cdots & 0 \\ 
     g(3) & g(2) & g(1) & \cdots & 0 \\ 
          &      &      & \vdots & 0 \\ 
     g(N) & g(N-1) & g(N-2) & \cdots & g(1) 
\end{bmatrix},
\]
where the right-hand-side matrix involving $g$ is an invertible Topelitz matrix. 
Thus by inversion, we see that our claimed formula is correct. 
\end{proof}

\begin{table}[h!]

\caption[Symbolic computations of the inverse functions for D-convolutions]{
	 The symbolic inverse functions $g^{-1}[\mathcal{D}]$ 
         for any fixed arithmetic function $g$ such that 
         $g(1) \neq 0$, and any fixed invertible, lower 
         triangular kernel function satisfying $\mathcal{D}(n, n) \neq 0$ for all $n \geq 1$. }

\begin{center}
\begin{equation*}
\boxed{
\begin{array}{c|l}
 n & g^{-1}[\mathcal{D}](n) \\ \hline 
 1 & \frac{1}{\mathcal{D}(1,1) g(1)} \\
 2 & -\frac{\mathcal{D}(2,2) g(2)}{\mathcal{D}(1,1) \mathcal{D}(2,1) g(1)^2} \\
 3 & \frac{\mathcal{D}(2,2) \mathcal{D}(3,2) g(2)^2}{\mathcal{D}(1,1) \mathcal{D}(2,1) 
     \mathcal{D}(3,1) g(1)^3}-\frac{\mathcal{D}(3,3)
     g(3)}{\mathcal{D}(1,1) \mathcal{D}(3,1) g(1)^2} \\
 4 & \scriptsize{-}\frac{\mathcal{D}(2,2) \mathcal{D}(3,2) \mathcal{D}(4,2) g(2)^3}{\mathcal{D}(1,1) 
     \mathcal{D}(2,1) \mathcal{D}(3,1) \mathcal{D}(4,1)
	g(1)^4}\tiny{+}\frac{(\mathcal{D}(2,1) \mathcal{D}(3,3) \mathcal{D}(4,2)+\mathcal{D}(2,2) \mathcal{D}(3,1) 
     \mathcal{D}(4,3)) g(3) g(2)}{\mathcal{D}(1,1)
     \mathcal{D}(2,1) \mathcal{D}(3,1) \mathcal{D}(4,1) g(1)^3}\tiny{-}\frac{\mathcal{D}(4,4) g(4)}{\mathcal{D}(1,1) 
     \mathcal{D}(4,1) g(1)^2} \\
\end{array}
}
\end{equation*}
\end{center}

\end{table}

In what follows, we adopt the notation that 
$c_n(\mathcal{C}) := [q^n] \mathcal{C}(q)$ and 
$p_n(\mathcal{C}) := [q^n] \mathcal{C}(q)^{-1}$ for any $n \geq 0$ and any OGF 
$\mathcal{C}(q)$ such that $\mathcal{C}(0) \neq 0$. 

\begin{theorem}[Generalized factorization theorems for $\mathcal{D}$-convolution]
Suppose that $g$ is any arithmetic function that is invertible with respect to 
$\mathcal{D}$-convolution for some fixed invertible, lower triangular kernel function 
$\mathcal{D}(n, k)$. The matrices with entries given by 
$s_{n,k}(g; \mathcal{C},\mathcal{D})$ in 
\eqref{eqn_DCvlGenFactThmExp_restated_v2} 
are invertible and satisfy the following formulas for $1 \leq k \leq n$: 
\begin{align*}
s_{n,k}(g; \mathcal{C},\mathcal{D}) & = \sum_{j=1}^{n} c_{n-j}(\mathcal{C}) g(j+1-k) \mathcal{D}(j, k) \\ 
s_{n,k}^{-1}(g; \mathcal{C},\mathcal{D}) & = \sum_{j=1}^{n} g^{-1}[\mathcal{D}](n+1-j) p_{j-k}(\mathcal{C}) 
     \mathcal{D}(n, j). 
\end{align*}
\end{theorem}
\begin{proof}
The formula for the ordinary matrix entries is obvious upon multiplying both sides of 
\eqref{eqn_DCvlGenFactThmExp_restated_v2} by the 
OGF, $\mathcal{C}(q)$, and then extracting the coefficients of $q^n$ and $f(k)$ in the 
resulting expansion. The proof of the inverse matrix formulas is routine, but less obvious. 
Since $s_{n,k}(g; \mathcal{C},\mathcal{D})$ is lower triangular with non-zero entries when $n = k$ 
for all $n \geq 1$, it forms a sequence of invertible square matrices taking determinants over 
$1 \leq n,k \leq N$ for each fixed $N \geq 1$. 
Consider the special case of the $\mathcal{D}$-convolution sums 
where $f(n) := s_{n,k}^{-1}(g; \mathcal{C},\mathcal{D})$, an inverse sequence that 
we know is unique for each fixed pair $(f, g)$, for integers $k \geq 1$. 
By the orthogonality relations between the lower triangular ordinary and inverse matrices, 
we can see that 
\[
p_{n-k}(\mathcal{C}) = \sum_{j=1}^{n} s_{n,k}^{-1}(g; \mathcal{C},\mathcal{D}) g(n+1-j) \mathcal{D}(n, j). 
\]
Since $g$ is invertible with respect to $\mathcal{D}$-convolution, we recover our claimed formula for the 
inverse matrix entries. 
\end{proof}

\newpage
\SectionGTThesisFormatted{Canonical representations of factorization theorems for special sums}
\label{Section_CanonicalReprsOfFactThms_KernelBasedDCVL} 

The material we present in this concluding section of the thesis is not exhaustive, nor conclusive. 
Rather it serves to motivate a discussion of rigorously formulating ``best possible'' factorization theorems 
and relationships between application-dependent convolution type 
sequences. These so-termed ``canonical'' 
expressions that arise in other important applications and future useful constructions based on our work 
from Section \ref{Section_LGFFactTheorems_PriorWork}. 
A few conjectures are presented in the last subsection below that suggest a loose 
application-tied partition theoretic interpretation behind the ideal correlation for factorization 
theorems we have for the class of more general convolution type sums defined in 
equation \eqref{eqn_DCvlSummationBasedExpDefs_restated_v2}. 

\SubsectionGTThesisFormatted{Correlation statistics to quantify the notion of a ``canonically best'' property}

There is a vast body of modern literature in number theory that motivates semi-standardized ways to quantify 
relationships between functions and sequences we study via correlation based statistics. 
There is historically relevant literature about using statistical analysis to motivate 
studying number theoretic objects. For example, the non-trivial zeros of the Riemann zeta function 
have been related and bounded via pair correlation formulas. Moreover, this 
topic continues to be a active and fruitful way of understanding this complicated subject matter. 
We recall from \cite{WILLIAMS-MILLER-PCORR-OVERVIEW-REF} that results in analytic number theory that 
make sense of the distribution of the non-trivial zeros of $\zeta(s)$ originated in the work of Montgomery. 
Subsequent follow-up work that collectively builds on Montgomery's contributions in the context of 
$L$-functions, Gaussian Unitary Ensemble (or GUE), 
applications in random matrix theory and their associated correlation statistics is famously due to 
Hejal, Rudnick, Sarnak and Odlyzko. 

We posit by extension that using correlation metrics, 
or so-called sequence correlation \emph{statistics}, to precisely define and 
rigorously formulate what we consider to be best possible attainable relationships 
for the factorization theorems given in \eqref{eqn_DCvlGenFactThmExp_restated_v2}. 
In general, we can study the so-called \emph{sequence vector correlation} (including the 
information theoretic cross-correlation statistics seen below) 
that relate more general sequences and vectors of real and rational numbers. 
In our case, we need to identify and prove optimal representations 
for our notion of the ``best possible'', or optimal. e.g., canonical view point, for how we should express 
the OGF factorization theorems as they are identified in 
\eqref{eqn_DCvlGenFactThmExp_restated_v2}. 
We then set out to precisely construct formulas that can be maximized (minimized) with respect to all 
possible one-dimensional sequences in a way that captures the qualitatively meaningful relationships 
between the sequences from the LGF case. 
The goal is to do this in a very general setting that reveals 
underlying hidden relationships charaterizing any particular class of $\mathcal{D}$-convolution type sums in 
analog to the observations of natural relationships between multiplicative number theory and 
the partition functions from the LGF case witnessed in 
Section \ref{Section_LGFFactTheorems_PriorWork}. 

\begin{example}[A model starting point]
The exact bounded ranges we can expect for cross-correlation coefficients to express a numerical index 
between vectors in our problem context are, in general, variable and subject to the qualitative 
interpretation which we have to reason about separately to ensure a good model fit. 
If we wish to normalize the range to be within $[-1, 1]$, there is the standardized 
definition of the 
(non-central, or non-centralized) \emph{Pearson correlation coefficient}. 
It is defined as the numerical statistic relating any two $N$-tuples, 
$\vec{a} := (a_1, \ldots, a_N), \vec{b} := (b_1, \ldots, b_N) \in \mathbb{Q}^N$ for any fixed $N \geq 1$, 
given by 
\[
\operatorname{PearsonCorr}(N; \vec{a}, \vec{b}) := 
     \frac{1}{N} \times \frac{\sum\limits_{j=1}^{N} a_j b_j}{\sqrt{
     \sum\limits_{1 \leq i,j \leq N} a_i^2 b_j^2}} \in [-1, 1]. 
\]
\end{example}

\begin{notation}
We again define the shorthand sequence notation of 
$c_n(\mathcal{C}) := [q^n] \mathcal{C}(q)$ and 
$p_n(\mathcal{C}) := [q^n] \mathcal{C}(q)^{-1}$ for 
any $\mathcal{C}(q)$ such that $\mathcal{C}(0) \neq 0$. 
We are going to adapt the non-centralized Pearson cross-correlation formula 
by choosing our correlation statistic to be computed according to the following sums:
\begin{align}
\label{eqn_CrossCorrelationStatFormula_for_fixed_CDFuncs} 
\operatorname{Corr}(n; \mathcal{C}, \mathcal{D}) & := 
     \frac{1}{n} \times \frac{\sum\limits_{k=1}^{n} |c_k(\mathcal{C}) \mathcal{D}^{-1}(n, k)|}{ 
     \sqrt{\left(\sum\limits_{k=1}^{n} c_k(\mathcal{C})^2\right) \left( 
     \sum\limits_{k=1}^{n} \mathcal{D}^{-1}(n, k)^2\right)}} \\ 
\notag 
\operatorname{Corr}(\mathcal{C}, \mathcal{D}) & := 
     \sum_{n \geq 1} \operatorname{Corr}(n; \mathcal{C}, \mathcal{D}). 
\end{align}
\end{notation}

\begin{question}[The crux of our correlation statistic optimization problem]
For a fixed lower triangular, invertible kernel function $\mathcal{D}$, 
we need to identify a concrete candidate OGF, $\mathcal{C}(q)$, so that 
\[
0 \leq \operatorname{Corr}(\mathcal{C}, \mathcal{D}) < +\infty,
\]
is maximized or minimized (and finite) over all possible input functions 
$\mathcal{C}(q) \in \mathbb{Q}[[q]]$ such that $\mathcal{C}(0) \neq 0$, or alternately 
$\mathcal{C}(q) \in \mathbb{Z}[[q]]$ with $\mathcal{C}(0) = 1$. 
Note that this criteria and the corresponding maximization procedure 
is always independent of the arithmetic functions $f,g$ input to the 
weighted $\mathcal{D}$-convolution sums, $f \boxdot_{\mathcal{D}} g$. 
\end{question}

\begin{example}[Finding optimal statistics for the LGF case]
We will make a somewhat arbitrary decision that works well in practice to define 
\[
f(n,k) := \begin{cases} 
     \frac{1}{n} \times \frac{ 
     c_{k}(\mathcal{C}) \mathcal{D}^{-1}(n, k)}{ 
     \left(\sum\limits_{m \leq n} c_m(\mathcal{C})^2\right)^{\frac{1}{2}} \left( 
     \sum\limits_{m \leq n} \mathcal{D}^{-1}(n, m)^2\right)^{\frac{1}{2}}}, 
     & \text{if\ } 1 \leq k \leq n \leq N; \\ 
     0, & \text{otherwise,}
     \end{cases}
\]
Notice that in the cases we next look at for the LGF example, we have that 
\[
\sum_{m \leq n} \mathcal{D}^{-1}(n, m)^2 = \sum_{d|n} \mu^2(d) = 2^{\omega(n)}, n \geq 1. 
\]
We then want to optimize the minimal bounded formulas 
\[
\lim_{N \rightarrow \infty} \sum_{n \leq N} \sum_{k=1}^n f(n, k) 
     = \lim_{N \rightarrow \infty} 
     \sum_{n \leq N} \frac{(|c_{-}(\mathcal{C})| \ast |\mu|)(n)}{ 
     n \rho_{\mathcal{C}}(n) \left(\sqrt{2}\right)^{\omega(n)}} 
     \in [0, 1], 
\]
over all $\mathcal{C}(q)$ such that $\mathcal{C}(0) \neq 0$ and with 
$$\rho_{\mathcal{C}}(n) := \left(\sum_{0 \leq m \leq n} c_m(\mathcal{C})^2\right)^{\frac{1}{2}}.$$
We see that minimizing the reciprocal of the limiting series in the previous equation leads to a maximal 
possible bound on the cross-correlation statistics we defined in 
\eqref{eqn_CrossCorrelationStatFormula_for_fixed_CDFuncs}.
The preliminary numerical results we cite for this case in 
Section \ref{subSection_LGF_TheoreticalAnalysisOfBestPossibleBoundsForCorrStats} below 
is able to numerically predict how closely this statistic for the LGF case comes to attaining the 
theoretically maximal correlation statistic in limiting cases. 
These series approximations can be made very accurate for certain classes of OGFs that 
commonly arise in applications. 
\end{example}

\SubsectionGTThesisFormatted{Maximal correlation bounds for the LGF case} 
\label{subSection_LGF_TheoreticalAnalysisOfBestPossibleBoundsForCorrStats} 

Since the task of identifying the target limiting cross-correlation statistic in 
\eqref{eqn_CrossCorrelationStatFormula_for_fixed_CDFuncs} is substantially complicated 
in the general case, we first look at the problem of optimality for the LGF case. 
For each such OGF $\mathcal{C}(q)$, we define 
\begin{align}
\notag 
\operatorname{Corr}_{\operatorname{LGF}}(\mathcal{C}) & = 
     \lim_{n \rightarrow \infty} \sum_{k=1}^{n} \frac{\mu^2(k)}{k} \times 
     \sum_{j \leq \Floor{n}{k}} \frac{|c_j(\mathcal{C})|}{j \rho_{\mathcal{C}}(jk) 
     (\sqrt{2})^{\omega\left(jk\right)}} \\ 
\label{eqn_CorrLGFStat_DoubleSeriesExps_v2} 
     & = \sum_{j,k \geq 1} 
     \frac{\mu^2(j) |c_k(\mathcal{C})|}{(jk) \rho_{\mathcal{C}}(jk) 
     (\sqrt{2})^{\omega(jk)}}, 
\end{align} 
where we define the partial variance of $\mathcal{C}$ to be 
\[
\rho_{\mathcal{C}}(N) := \sqrt{\sum_{1 \leq i \leq N} c_i(\mathcal{C})^2}, N \geq 1. 
\]
The previous doubly infinite series for the LGF correlation statistic is non-trivial 
to tightly bound from above and below because we have, in general, for $j,k \geq 1$ that 
\[
\omega(jk) = \omega(j) + \omega\left(\frac{k}{(k, j)}\right). 
\]
That is, the function $\omega(n)$ is only \emph{strongly} (as opposed to \emph{completely}) \emph{additive}. 

\begin{definition}
\label{def_PairCorrStatsGenOGF_CharsAndParameterSpecs_v2} 
Fix any $0 < \delta < +\infty$. Provided an input test OGF $\mathcal{C}(q)$, we set 
\[
A_0(\mathcal{C}, \delta) := \lim_{N \rightarrow \infty} 
     \frac{1}{N^{\delta}} \times \sqrt{\sum_{1 \leq n \leq N} 
     c_n(\mathcal{C})^2}. 
\]
\end{definition}

We are naturally interested in finding the optimal, or so-termed ``canonically best`` correlation 
coefficient that corresponds to an explicit OGF $\mathcal{C}(q)$. 
We have already noticed that taking $\mathcal{C}(q) := (q; q)_{\infty}$ leads to very interesting 
relationships between the matrices in the Lambert series factorization theorems. 
Conjecture \ref{conj_OptimalDeltaParameter_FromLGFAndMertensRHCase_v1} 
given at the end of this section is suggestive of why the expected optimal OGF witness that 
attains the maximal correlation statistic, the function $\mathcal{C}(q)$, 
has series coefficients that satisfy (by the \emph{pentagonal number theorem}) 
$$\delta = \sup\left\{\rho > 0: A_0(\mathcal{C}, \rho) > 0\right\} \equiv \frac{1}{4}.$$ 
Hence, we are interested in considering OGFs $\mathcal{C}(q)$ with integer coefficients such that 
the maximal $\delta > 0$ in the definition for which $A_0(\mathcal{C},\delta) > 0$ is given by 
$\delta := \frac{1}{4}$. 

Using the same construction 
as the special case where $\mathcal{C}(q) := (q; q)_{\infty}$, OGF forms whose 
coefficients are in $\{0, \pm 1\}$ such that the corresponding $\delta = \frac{1}{4}$ 
can be seen as often having non-zero coefficients, $|c_n(\mathcal{C})|$, for 
$n \in \{p(n)\}_{n=-\infty}^{\infty} \setminus \{0\}$ where 
$\hat{p}(n) = \frac{n(an+b)}{2}$ for integers $a \geq 3$ and $1 \leq b < a$. 
As we can see through the next OGF examples, generating functions of this form 
are natural to consider in partition theoretic applications \citep[\S 19.9]{HARDYWRIGHT}: 
\begin{align*}
\tag{PF-A}
(q; q)_{\infty} & = \prod_{n \geq 0} \left\{(1-q^{3n+1})(1-q^{3n+2})(1-q^{3n+3})\right\} \\ 
     & = 
     \sum_{n=-\infty}^{\infty} (-1)^n q^{n(3n+1)/2} \\ 
(-q; q^2)_{\infty}^2 (q^2; q^2)_{\infty} & = \prod_{n \geq 0} \left\{(1+q^{2n+1})^2(1-q^{2n+2})\right\} \\ 
     & = 
     \sum_{n=-\infty}^{\infty} q^{n^2} \\ 
     (q; q^2)_{\infty}^2 (q^2; q^2)_{\infty} & = \prod_{n \geq 0} \left\{(1-q^{2n+1})^2(1-q^{2n+2})\right\} \\ 
     & = 
     \sum_{n=-\infty}^{\infty} (-1)^n q^{n^2} \\ 
\tag{PF-B}
(q; q^5)_{\infty} (q^4; q^5)_{\infty} (q^5; q^5)_{\infty} & = 
     \prod_{n \geq 0} \left\{(1+q^{5n+1})(1-q^{5n+4})(1-q^{5n+5})\right\} \\ 
     & = 
     \sum_{n=-\infty}^{\infty} (-1)^n q^{n(5n+3)/2} \\ 
\tag{PF-C}
(q^2; q^5)_{\infty} (q^3; q^5)_{\infty} (q^5; q^5)_{\infty} & = 
     \prod_{n \geq 0} \left\{(1+q^{5n+2})(1-q^{5n+3})(1-q^{5n+5})\right\} \\ 
     & = 
     \sum_{n=-\infty}^{\infty} (-1)^n q^{n(5n+1)/2}
\end{align*}
In these cases, we have that 
$A \equiv A_0(\mathcal{C}, \delta) = 2\left(\frac{2}{a}\right)^{\frac{1}{4}}$ with 
$a = 3$ (PF-A), $1$ (not labeled), $1$ (not labeled), $5$ (PF-B) and $5$ (PF-C), respectively. 
A comparison of these generating functions extending the visual projection of a correlation 
matrix onto the penguin image is provided below for reference 
with the OGFs labeled (PF-A), (PF-B) and (PF-C)
corresponding to the images in 
Figure \ref{figure_TucCompsOfPredictablePartFuncGFProductExps_v2} 
ordered from left to right. 

\begin{figure}[h!]

\caption[Visualization of OGF correlation for LGFs using partition function OGFs]{}
\label{figure_TucCompsOfPredictablePartFuncGFProductExps_v2}

\begin{center}
\fbox{\includegraphics[width=\textwidth]{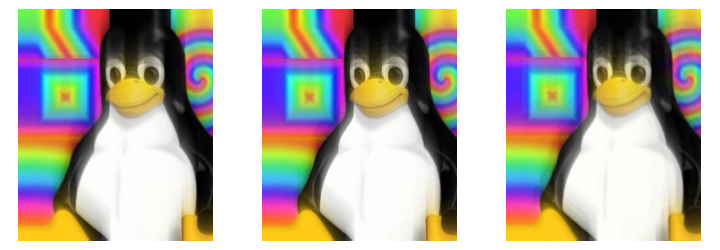}}
\end{center}

\end{figure}

Whenever $(a, b)$ are integers such that $a \geq 1$ and $1 \leq b < a$, let 
\[
\mathcal{C}_{a,b}(q) := \sum_{n=-\infty}^{\infty} (-1)^n q^{\frac{n(an+b)}{2}}. 
\]
We compute the special case values of 
$\operatorname{Corr}_{\operatorname{LGF}}(\mathcal{C}_{a,b})$ 
in Figure \ref{figure_CorrLGFCab_v1}. 

\begin{figure}[h!]

\caption[Numerical computations of correlation statistics for LGFs]{}
\label{figure_CorrLGFCab_v1}

\small
\begin{equation*}
\boxed{
\begin{array}{l|l|l|l|l|l|l|l}
(a, b) & (1, 0) & (3, 1) & (5, 1) & (5, 3) & (7, 1) & (7, 3) & (7, 5) \\ \hline
\operatorname{Corr}_{\operatorname{LGF}}(\mathcal{C}_{a,b}) & 
     0.7634 & 0.92008 & 0.062497 & 0.672979 & 0.010865 & 0.03742 & 0.6128 \\ \hline\hline
(a, b) & (11, 1) & (11, 3) & (11, 5) & (11, 7) & (11, 9) & (13, 1) & (13, 3) \\ \hline
\operatorname{Corr}_{\operatorname{LGF}}(\mathcal{C}_{a,b}) & 
     0.001585 & 0.002645 & 0.00471 & 0.02028 & 0.5875 & 0.000749 & 0.00108 \\ \hline\hline 
(a, b) & (13, 5) & (13, 7) & (13, 9) & (13, 11) & (17, 15) & (23, 21) & (29, 27) \\ \hline
\operatorname{Corr}_{\operatorname{LGF}}(\mathcal{C}_{a,b}) & 
     0.001824 & 0.00467 & 0.01882 & 0.5831 & 0.5824 & 0.5695 & 0.567 \\ 
\end{array}
}
\end{equation*}

\end{figure}

\begin{lemma}
\label{lemma_AvgOrder_of_omegan_along_quadratic_polys_v1} 
We have that the expectation (or average order)
\[
\mathbb{E}\left[\omega(\hat{p}(x)) + \omega(\hat{p}(-x))\right] 
     \sim \left(\frac{a+4}{a^2}\right) \left(\log\log x + B\right), 
     \text{ as } x \rightarrow \infty, 
\]
where $B \approx 0.261497$ is the \emph{Mertens constant} from Mertens' second theorem. 
That is, the average order of $\omega(n)$ over 
the two distinct integer-valued polynomials 
we get by expanding $\hat{p}(\pm n)$ over all non-zero integers $n$ is approximately 
a constant times $\mathbb{E}[\omega(x)]$ up to error terms that vanish as 
$x \rightarrow \infty$. 
\end{lemma}

\begin{remark}
The function $\omega(n)$ stays near its 
average order with a limiting centrally normal tendency as 
proved by the Erd\H{o}s-Kac theorem that states 
\cite[\S 1.7]{IWANIEC-KOWALSKI} \cite[\cf \S 7.4]{MV}
\[
\frac{1}{x} \times \#\left\{n \leq x: \frac{\omega(n) - \log\log n}{\sqrt{\log\log n}} \leq z\right\} = 
     \Phi(z) + o(1), z \in \mathbb{R}, \text{ as } x \rightarrow \infty. 
\]
We also have that uniformly for $0 < r \leq 1$ 
\[
\#\{n \leq x: \omega(n) \leq r \log\log x\} \ll x (\log x)^{r-1-r\log r}, 
     \text{ as } x \rightarrow \infty, 
\]
and that uniformly for $1 \leq r \leq R < 2$ 
\[
\#\{n \leq x: \omega(n) \geq r \log\log x\} \ll_R x (\log x)^{r-1-r\log r}, 
     \text{ as } x \rightarrow \infty. 
\]
Our intuition is hence to notice that $\omega(k)$ is so universally centered at its average order 
for almost every positive integer $k \geq 3$. 
Suppose that for $j, k \geq 16$ 
(so that the first n such that $\log\log(n) > 1$ is $n := \left\lceil e^e \right\rceil = 16$) 
we replace the terms in 
\eqref{eqn_CorrLGFStat_DoubleSeriesExps_v2} 
involving $\omega(jk)$ with the corresponding average order formula from 
Lemma \ref{lemma_AvgOrder_of_omegan_along_quadratic_polys_v1} evaluated at $x := jk$. 
We will denote this modified series by 
$\widehat{\operatorname{Corr}}_{\operatorname{LGF}}(\mathcal{C}_{a,b})$. 
The following table (in comparison to the one given above) is suggestive of the regularity 
of these series and hence may be an approach towards obtaining tight bounds on the 
actual LGF correlation statistics as we observe in 
Figure \ref{figure_CorrLGFHatCab_v2}. 
\end{remark}

\begin{figure}[h!]

\caption[Numerical computations of correlation statistics for LGFs (approximation)]{}
\label{figure_CorrLGFHatCab_v2}

\small
\begin{equation*}
\boxed{
\begin{array}{l|l|l|l|l|l|l|l}
(a, b) & (1, 0) & (3, 1) & (5, 1) & (5, 3) & (7, 1) & (7, 3) & (7, 5) \\ \hline
\widehat{\operatorname{Corr}}_{\operatorname{LGF}}(\mathcal{C}_{a,b}) & 
     0.7634 & 0.92008 & 0.062497 & 0.672979 & 0.010865 & 0.03742 & 0.6128 \\ \hline\hline
(a, b) & (11, 1) & (11, 3) & (11, 5) & (11, 7) & (11, 9) & (13, 1) & (13, 3) \\ \hline
\widehat{\operatorname{Corr}}_{\operatorname{LGF}}(\mathcal{C}_{a,b}) & 
     0.001585 & 0.002645 & 0.00471 & 0.02028 & 0.5875 & 0.000749 & 0.00108 \\ \hline\hline 
(a, b) & (13, 5) & (13, 7) & (13, 9) & (13, 11) & (17, 15) & (23, 21) & (29, 27) \\ \hline
\widehat{\operatorname{Corr}}_{\operatorname{LGF}}(\mathcal{C}_{a,b}) & 
     0.001824 & 0.00467 & 0.01882 & 0.5831 & 0.5824 & 0.5695 & 0.567 \\ 
\end{array}
}
\end{equation*}

\end{figure}

\begin{proof}[Proof of Lemma \ref{lemma_AvgOrder_of_omegan_along_quadratic_polys_v1}]
Suppose that we are evaluating the following sum:
\begin{align*}
E_{a,b}(x) & = \frac{1}{\sqrt{x}} \times \sum_{\substack{n \geq 1 \\ \frac{n(an \pm b)}{2} \leq x}} 
     \left[\omega\left(\frac{n(an + b)}{2}\right) + \omega\left(\frac{n(an - b)}{2}\right)\right]. 
\end{align*}
We can perform a change of variable in the form of $v = \frac{n(an \pm b)}{2}$ so that 
$n = \frac{\sqrt{b^2+8av} \mp b}{2a}$ and $dn = \frac{2dv}{\sqrt{8av+b^2}}$. 
We know that the average order of the original function $\omega(n)$ is given by \cite[\S 22.10]{HARDYWRIGHT} 
$$\mathbb{E}[\omega(n)] = \frac{1}{x} \times \sum_{n \leq x} \omega(n) = \log\log n + B + o(1).$$
Then we have by the Abel summation formula, taking the 
summatory function, $A(t) := t(\log\log t + B + o(1))$, that 
\begin{align*}
E_{a,b}(x) & = \frac{4}{\sqrt{x}} \times \sum_{v \leq x} \frac{\omega(v)}{\sqrt{8av+b^2}} \\ 
     & \sim \frac{4}{\sqrt{x}} \left( \frac{(\log\log x + B) x}{\sqrt{8ax+b^2}} + 
     \int_3^x \frac{4av (\log\log v + B)}{(8av+b^2)^{\frac{3}{2}}} dv\right) \\ 
     & \sim \frac{1}{2} \sqrt{\frac{8}{a}} (\log\log x + B) + \frac{16a}{(8a)^{\frac{3}{2}} \sqrt{x}} \times 
     \int_3^x \frac{\log\log v + B}{\sqrt{v}} dv \\ 
     & = \frac{1}{2} \sqrt{\frac{8}{a}} (\log\log x + B) + 
     \frac{32a}{(8a)^{\frac{3}{2}} \sqrt{x}} \left((\log\log x + B) \sqrt{x} - 
     2 \operatorname{Ei}\left(\frac{\log x}{2}\right)\right) \\ 
\tag{$\ast$} 
     & \sim \frac{1}{2} \sqrt{\frac{8}{a}}\left(1 + \frac{4}{a}\right) \left(\log\log x + B\right). 
\end{align*}
In determining the main term in transition from the second to last equation above, we have used that 
\[
1 - \frac{3\log x}{8} \leq 
     \operatorname{Ei}\left(\frac{\log x}{2}\right) - \log\log\left(\frac{x}{2}\right) \leq 
     1 - \frac{3\log x}{8} + \frac{11 (\log x)^2}{144}. 
\]
The main term for $\frac{2E_{a,b}(x)}{\sqrt{8a}}$ in ($\ast$) above 
corresponds to the average order formula we seek to evaluate since  
$n \leq \frac{\sqrt{8ax+b^2} \mp b}{2a}$ so that the correct scalar multiple in front of the average sum is 
similar to $\frac{2}{\sqrt{8ax}}$. 
\end{proof}

\SubsectionGTThesisFormatted{Conjectures on canonically best factorization theorems} 

\SubsubsectionGTThesisFormatted{Partition theoretic conjectures}
\label{subSection_LGFCaseExampleAnalysis_MotivatingByTheEulerTForm} 

Given the way in which we have chosen to expand the factorizations of $L_f(q)$ as 
\[
(f \ast \mathds{1})(n) = [q^n] L_f(q) = \frac{1}{\mathcal{C}(q)} \times \sum_{n \geq 1} \sum_{k=1}^{n} 
     s_{n,k}[\mathcal{C}] f(k) q^n, 
\]
the form of the invertible, lower triangular matrices with entries given by the $s_{n,k}[\mathcal{C}]$ are 
independent of any arithmetic $f$ that defines these expansions. 
Moreover, these matrices are completely determined by the choice of the reciprocal 
generating function factor of $\mathcal{C}(q)$ subject only to the 
requirement that $\mathcal{C}(0) \neq 0$. 
It follows that seeking an interpretation of the canonically ``best possible'' 
choice of this function as corresponding to the choice of taking 
$\mathcal{C}(q) := (q; q)_{\infty}$, a criteria which we define qualitatively as 
inducing an unexpected, or particularly revealing substructure to the left-hand-side 
divisor sums, $f \ast \mathds{1}$, is also independent of any fixed arithmetic $f$ that 
defines $L_{f}(q)$. 

\begin{definition}[The Euler transform]
We observe a property called the \emph{Euler transform} 
which nicely suggests motivation for why the partition function $p(n)$ arises so 
naturally in this class of LGF examples. 
Namely, we borrow the canonical integer sequence transformation 
identified by Bernstein and Sloane (circa 2002) 
called \texttt{EULER} \cite{OEIS-EULERTF-REF}. 
It states that if two arithmetic functions 
$a_n,b_n$ with a corresponding former OGF defined by 
$A(x) := \sum_{n \geq 1} a_n x^n$ are related by the identity 
\[
1 + \sum_{n \geq 1} b_n x^n = \prod_{i \geq 1} \frac{1}{(1-x^i)^{a_i}} \equiv 
     \exp\left(\sum_{k \geq 1} \frac{A(x^k)}{k}\right), 
\]
then we can explicitly relate these sequence by introducing an intermediate divisor sum, denoted by $c_n$. 
In particular, if $c_n := \sum_{d|n} d \cdot a_d$, then we have that 
\[
a_n = \frac{1}{n} \times \sum_{d|n} c_d \mu\left(\frac{n}{d}\right), n \geq 1. 
\]
\end{definition} 

Results from elementary number theory due to Euler show that the partition function $p(n)$ is 
related to the (ordinary) sum-of-divisors function, $\sigma(n) \equiv \sigma_1(n) := \sum_{d|n} d$, 
through the following recurrence relation:
\[
n \cdot p(n) = \sum_{0 \leq k < n} \sigma_1(n-k) p(k), n \geq 1. 
\]
Since $p(0) = 1$, the resulting ODE for the generating functions that relate these two sequences 
shows that 
$$p(n) = [q^n] \exp\left(\sum_{k \geq 1} \frac{\sigma_1(k) q^k}{k}\right).$$
On the other hand, when we take the constant sequence $a_n \equiv 1, \forall n \geq 1$, the 
product expanded through the \texttt{EULER} transformation we defined above 
corresponds to the infinite $q$-Pochhammer function product, $(q; q)_{\infty}^{-1}$, which again 
generates $p(n)$ for all $n \geq 0$. 
Since $\sigma_{-\alpha}(n) = \sigma_{\alpha}(n) n^{-\alpha}$ for all $n \geq 1$ and any 
real parameter $\alpha > 0$, taking the exponential of the sum over the OGF $A(q)$ 
yields that 
\[
\prod_{i \geq 1} (1-q^i)^{-1} = \exp\left(\sum_{k \geq 1} \frac{q^k}{k(1-q^k)}\right) = 
     \exp\left(\sum_{k \geq 1} \frac{\sigma_1(k) q^k}{k}\right). 
\]
Thus, we reason that the fundamental relation for $p(n)$ to the multiplicative divisor sums 
$\sigma_1(n)$ explains why the partition function arises here in the LGF case. 
We can look to the LGF special case for clues to see a good first order heuristic that we can use 
to measure how closely related the matrix and inverse matrix coefficients are for a fixed 
$\mathcal{D}$-convolution summation type. We clearly must define our metric to qunatify this heuristic 
so that is depends only on the kernel function $\mathcal{D}$, and the OGF $\mathcal{C}(q)$, 
and is always (of course) independent of $f,g$. 

\begin{conjecture}
An optimal OGF, $\mathcal{C}(q)$, that maximizes the correlation coefficients in 
\eqref{eqn_CrossCorrelationStatFormula_for_fixed_CDFuncs}, is given by 
\[
\mathcal{C}(q) := \prod_{k \geq 1} \left(\sum_{n \geq 0} \mathcal{D}(n+k-1, k) q^n\right)^{-1}. 
\]
\end{conjecture}

\begin{conjecture}
The LGF OGF matchings we saw in 
Section \ref{subSection_LGFCaseExampleAnalysis_MotivatingByTheEulerTForm} by 
applying the Euler transform of sequences suggests an optimal selection of 
$\mathcal{C}(q)$ satisfies the following expansions: 
\[
\mathcal{C}(q) = \exp\left(-\sum_{n \geq 1} \sum_{k=1}^{n} k \mathcal{D}(n,k) \frac{q^n}{n}\right) = 
     \prod_{n \geq 1} \left(1 + q \times \sum_{k=1}^{n} k \mathcal{D}(n, k)\right)^{-1}.
\]
We have the identification of the convolution sums in the last equation with 
a special case of equation \eqref{eqn_DCvlSummationBasedExpDefs_restated_v2} as 
\[
\left(\operatorname{Id}_1 \boxdot_{\mathcal{D}} \mathds{1}\right)(n) = \sum_{k=1}^{n} k \mathcal{D}(n, k). 
\]
\end{conjecture}

\begin{conjecture}[Equivalence of problems]
The cross-correlation statistic 
\[
\operatorname{Corr}_1(\mathcal{C},\mathcal{D}) := \sum_{n \geq 1} 
     \frac{\frac{1}{n} \times \sum\limits_{1 \leq k \leq n} c_k(\mathcal{C}) \mathcal{D}^{-1}(n, k)}{ 
     \sqrt{\sum\limits_{1 \leq k \leq n} c_k(\mathcal{C})^2 \times 
     \sum\limits_{1 \leq k \leq n} \mathcal{D}^{-1}(n, k)^2}}, 
\]
is maximized (minimized) over all possible OGFs $\mathcal{C}(q)$ if and only if 
\[
\operatorname{Corr}_2(\mathcal{C},\mathcal{D}) := \sum_{n \geq 1} 
     \frac{\frac{1}{n} \times \sum\limits_{1 \leq k \leq n} p_k(\mathcal{C}) \mathcal{D}(n, k)}{ 
     \sqrt{\sum\limits_{1 \leq k \leq n} p_k(\mathcal{C})^2 \times 
     \sum\limits_{1 \leq k \leq n} \mathcal{D}(n, k)^2}},
\]
is maximized (minimized) over all such OGFs. 
\end{conjecture}

\SubsubsectionGTThesisFormatted{Other conjectures} 

The values of certain signed sums are often modeled as a $\{\pm 1\}$-valued random walk on the integers whose 
height after the $x^{th}$ step is taken to be approximately $M(x)$ where the probabilities of 
moving by $\pm 1$ at any given step are randomized \cite{ONTHERH-FRACTALRW,SMOOTH-RDMFUNCS}. 
We know that the expectation of the absolute height at $x$ of 
a prototypical random walk of this type is asymptotically $C\sqrt{x}$ for $C > 0$ an absolute constant. 
Namely, suppose that $\{X_i\}_{i \geq 1}$ is a sequence of independent random variables 
defined such that $\mathbb{P}[X_i=-1] = \mathbb{P}[X_i=+1] = \frac{1}{2}$ for all $i \geq 1$. 
We can form the sums $Y_n := \sum_{i \leq n} X_i$ for $n \geq 1$. 
By computation we have that $\mathbb{E}[Y_n] = 0$. 
At the same time, we can show that  
\[
\sigma_{Y_n} := \sqrt{\mathbb{E}[Y_n^2]-\mathbb{E}[Y_n]^2} = \sqrt{n}. 
\]
This follows since 
\[
Y_{n+1}^2 = (Y_n + X_n)^2 \implies \mathbb{E}[Y_{n+1}^2] = \mathbb{E}[Y_n^2] + 1 \implies 
     \mathbb{E}[Y_n^2] = n, \text{ for all } n \geq 1. 
\]
An interpretation of the combined first and second moment analysis is that we should 
expect the random walk modeled by $Y_n$ to be approximately zero-valued most of the time but 
with an expected spread in actual values of as much as $\sqrt{n}$ \cite{RHRW-WEBSITE-REF}. 
The \emph{law of the iterated logarithm} more precisely implies that 
$Y_n = O\left(\sqrt{n \log\log n}\right)$ for all sufficiently large $n$. 

Since the coefficients enumerated by a LGF for any function $f$ are $f \ast \mathds{1}$, and 
we know by elementary number theory that $(\mu \ast \mathds{1})(n) = (1 \ast \mu)(n) = \delta_{n,1}$ 
(the multiplicative identity function with respect to Dirichlet convolution), 
we have that inversion of the convolution sums of this type is performed by Dirichlet convolution 
with $\mu(n)$ (\cf M\"obius inversion). 
The summatory function, or partial sums of $\mu(n)$ are defined by 
$M(x) := \sum_{n \leq x} \mu(n)$ for any $x \geq 1$. 
The values of $M(x)$ are often modeled by a similar random walk whose values are $\{0,\pm 1\}$-valued 
according to the distribution of $\mu(n) \mapsto \pm 1$. 
This is often viewed as a case of the $\pm 1$-valued 
random walk above with a different leading constant factor on the variance. 

The Riemann Hypothesis is equivalent to proving that 
\[
M(x) = O\left(x^{\frac{1}{2}+\epsilon}\right), \text{ for all } 0 < \epsilon < \frac{1}{2}. 
\]
The dependence of the divisors $d|n$ over which we sum $f$ to compute 
$(f \ast \mathds{1})(n)$ at each $n \geq 1$ is deeply connected to the distribution of the primes. 
We assert that the conventional interpretation of the primes as randomly determined in the sense of the 
$\pm 1$-valued random walk model from above plays a pivotal role in the maximal correlation statistic that relates 
the kernel $\mathcal{D}^{-1}(n, k) = \mu\left(\frac{n}{k}\right) \Iverson{k|n}$ to an optimal OGF, 
$\mathcal{C}(q) = (q; q)_{\infty}$ (as we have predicted it should be). 
The next conjecture cuts precisely to the crux of the matter with respect to 
why we seem to witness an optimal correlation statistic in the LGF case 
when $\mathcal{C}(q)$ satisfies $\rho_{\mathcal{C}}(n)^2 \asymp \sqrt{n}$.

\begin{conjecture}
\label{conj_OptimalDeltaParameter_FromLGFAndMertensRHCase_v1}
For any fixed lower triangular, invertible kernel function $\mathcal{D}(n, k)$, let 
\[
M_{\mathcal{D},k}(x) := \sum_{n \leq x} \mathcal{D}^{-1}(n, k), \text{ for } 1 \leq k \leq x. 
\]
Suppose that 
\[
\delta^{\frac{1}{2}} = \inf\left\{\rho > 0: \max_{1 \leq k \leq x} M_{\mathcal{D},k}(x) = O\left(x^{\rho+\varepsilon}\right), 
     \forall \varepsilon>0\right\}. 
\]
An optimal OGF, $\mathcal{C}(q)$ with integer-valued coefficients such that $\mathcal{C}(0) \neq 0$, 
that attains the theoretical maximum value of 
$\operatorname{Corr}(\mathcal{C},\mathcal{D})$ satisfies 
$\rho_{\mathcal{C}}(n) \asymp n^{\delta}$ as $n \rightarrow \infty$. 
That is, $\rho_{\mathcal{C}}(n)$ is bounded above and below in absolute value by 
bounded constant multiples of $n^{\delta}$ for all sufficiently large $n$. 
\end{conjecture}

\setcounter{section}{0}
\setcounter{gtsection}{0}
\renewcommand{\SectionNumberFormat}[1]{\Alph{#1}}
\renewcommand{\TheSectionPrefixName}[0]{Appendix}
\renewcommand{\TheSectionPrefixNameUC}[0]{APPENDIX}

\newpage
\SectionGTThesisFormatted{Computational experiments and supporting software}

\SubsectionGTThesisFormatted{Invertible magic sign-smoothing partition convolutions}

\begin{definition}
\label{def_p1np2n_magicPartFnSeqs_v1}
For integers $n \geq 0$, we define the sequences 
\cite[\seqnum{A000009}; \seqnum{A081362}]{OEIS} 
\begin{align*}
p_1(n) & := [q^n] \prod_{m \geq 1} (1+q^m), \\ 
p_2(n) & := [q^n] \prod_{m \geq 1} (1+q^m)^{-1}. 
\end{align*}
\end{definition}

\begin{definition}
\label{def_s1s2fn_cvlTfSeqs_v1} 
We define two invertible transformations on any arithmetic function $f$ formed by the 
next discrete convolutions for any $n \geq 0$. 
\begin{align*} 
s_1[f](n) & := \sum_{j=1}^{n} f(j) p_1(n-j), \\ 
s_2[f](n) & := (-1)^n \times \sum_{j=1}^{n} f(j) p_2(n-j) 
\end{align*}
\end{definition}

\begin{definition}
We say that the arithmetic sequence $\{f(n)\}_{n \geq 1}$ has 
\emph{property $\mathcal{P}_1$} at $N$ if the sign of $s_1[f](n)$ is constant for all $n \geq N$. 
Similarly, we say that the sequence has \emph{property $\mathcal{P}_2$} at $N$ if the sign of 
$s_2[f](n)$ is constant for all $n \geq N$. We define 
\begin{align*} 
M_{1}[f] & := \sup \left\{n \geq 1: f \text{\ does not have property\ } \mathcal{P}_1 \text{\ at\ } n\right\}, \\ 
M_{2}[f] & := \sup \left\{n \geq 1: f \text{\ does not have property\ } \mathcal{P}_2 \text{\ at\ } n\right\}. 
\end{align*} 
\end{definition} 

\begin{conjecture}[Magic sign-smoothing transformations of multiplicative Dirichlet inverse functions]
Suppose that $f$ is a non-constant Dirichlet invertible function. 
If $f(n)$ has constant sign for all $n \geq 1$, then 
\begin{itemize}
\item[\rm (C.1)] $M_1\left[f^{-1}\right]$ is bounded, i.e., for all sufficiently large $n$ the 
	transformation $s_1[f^{-1}](n)$ has constant sign, and its sign is given by 
	$\operatorname{sgn}\left(f\left(M_1\left[f^{-1}\right]\right)\right) = -\operatorname{sgn}\left(f(1)\right)$; and 
\item[\rm (C.2)] $M_2\left[f^{-1}\right]$ is bounded, i.e., for all sufficiently large $n$ the 
	transformation $s_2[f^{-1}](n)$ has constant sign, and its sign is given by 
	$\operatorname{sgn}\left(f\left(M_2\left[f^{-1}\right]\right)\right) = \operatorname{sgn}\left(f(1)\right)$.
\end{itemize}
If $\{-1, 1\} \subseteq \{\operatorname{sgn}(f(n)): n \geq 1\}$, i.e., 
$f(n)$ oscillates in sign over the positive integers, then 
\begin{itemize}
\item[\rm (C.3)] $M_1\left[f^{-1}\right]$ is bounded, i.e., for all sufficiently large $n$ the 
	transformation $s_1[f^{-1}](n)$ has constant sign, and its sign is given by 
	$\operatorname{sgn}\left(f\left(M_1\left[f^{-1}\right]\right)\right) = \operatorname{sgn}\left(f(1)\right)$; and 
\item[\rm (C.4)] $M_2\left[f^{-1}\right]$ is bounded, i.e., for all sufficiently large $n$ the 
	transformation $s_2[f^{-1}](n)$ has constant sign, and its sign is given by 
	$\operatorname{sgn}\left(f\left(M_2\left[f^{-1}\right]\right)\right) = -\operatorname{sgn}\left(f(1)\right)$.
\end{itemize}
\end{conjecture}

\begin{remark}
Examples of the sign-smoothing properties in (B.1) and (B.2) of the conjecture are provided in 
Table \ref{table_s12fn_SignSmoothingTF_v1}--\ref{table_s12fn_SignSmoothingTF_v5}. 
Proofs based on the circle method applied to the respective sequence OGFs show that 
\cite{ACOMB-BOOK}
\begin{align*}
p_1(n) & \sim \frac{3^{\frac{3}{4}}}{12 n^{\frac{3}{4}}} \times \exp\left(\pi\sqrt{\frac{n}{3}}\right), \\ 
p_2(n) & \sim \frac{(-1)^n}{2 \cdot 24^{\frac{1}{4}} n^{\frac{3}{4}}} \times \exp\left(\pi\sqrt{\frac{n}{6}}\right)
\end{align*}
A natural follow up question is whether there are sequences that are inverses of one another in the 
sense of Definition \ref{def_s1s2fn_cvlTfSeqs_v1} that also preserve 
analogous sign-smoothing properties to the conjecture. 
To give some intuition for why the partition function sequences from 
Definition \ref{def_p1np2n_magicPartFnSeqs_v1} are special, we have 
tabulated the analogous transforms for the more common pair of partition functions 
defined for integers $n \geq 0$ by 
\cite[\seqnum{A000041}; \seqnum{A010815}]{OEIS} 
\begin{align*}
p(n) & := [q^n] \prod_{m \geq 1} (1-q^m)^{-1}, \\ 
\hat{p}(n) & := [q^n] \prod_{m \geq 1} (1-q^m). 
\end{align*}
The corresponding invertible convolution transforms performed with respect to these functions 
are defined by 
\begin{align*} 
\widehat{s}_1[f](n) & := \sum_{j=1}^{n} f(j) \hat{p}(n-j), \\ 
\widehat{s}_2[f](n) & := (-1)^n \times \sum_{j=1}^{n} f(j) p(n-j). 
\end{align*}
Analogs to the first tables demonstrating the qualitatively sign-smoothing transformations 
by the operations in Definition \ref{def_s1s2fn_cvlTfSeqs_v1} for the corresponding 
operations defined in the previous equations are shown for comparison in 
Table \ref{table_s12Hatfn_SignSmoothingTF_v1}--\ref{table_s12Hatfn_SignSmoothingTF_v5}.
While we do not obtain the conjectured eventually constant signed values by 
convolution with the partition functions used to form the transformations 
$s_1[f^{-1}](n)$ and $s_2[f^{-1}](n)$ in these cases, we do still observe a curious and very 
interesting pattern in these tables of the sequences $\widehat{s}_2[f^{-1}](n)$. 
Namely, for each arithmetic function $f$, 
whether signed or unsigned over the positive integers, such that $f(1) \neq 0$ we have 
\[
	\exists N < +\infty \text{ such that for all } n \geq N:
	\operatorname{sgn}\left(\widehat{s}_2\left[f^{-1}\right](n)\right) = (-1)^{n+1}. 
\]
\end{remark}

\SubsectionGTThesisFormatted{Source code for the magic partition transformation experiments}

\begin{lstlisting}[caption={Mathematica code for the magic partitions tables},captionpos=t]
DirichletInverse[n_, ffunc_] := 
     Module[{icVec, postMultMtrx, dinvSols, mult, extraEqns},
     If[n <= 0 || ! IntegerQ[n] || ffunc[1] === 0, 
          Return[Null];
     ];
     If[n === 1, 
          Return[1 / ffunc[1]];
     ];
     Return[-1 / ffunc[1] * DivisorSum[n, 
            ffunc[#] * DirichletInverse[n / #, ffunc]&, # > 1&]
	   ];
   ];

p1[n_] := SeriesCoefficient[Product[(1 + q^m), {m, 1, n}], {q, 0, n}]
p2[n_] := SeriesCoefficient[1 / Product[(1 + q^m), {m, 1, n}], {q, 0, n}]
PartitionsEncode[n_, func_] := Sum[func[j] * p1[n - j], {j, 1, n}]
PartitionsDecode[n_, func_] := Sum[func[j] * p2[n - j], {j, 1, n}]
PartitionsEncode2[n_, func_] := Sum[func[j] * p2[n - j], {j, 1, n}]
PartitionsDecode2[n_, func_] := Sum[func[j] * p1[n - j], {j, 1, n}]

p[n_] := PartitionsP[n]
phat[n_] := SeriesCoefficient[Product[(1 - q^m), {m, 1, n}], {q, 0, n}]
PartitionsEncodeHat[n_, func_] := Sum[func[j] * phat[n - j], {j, 1, n}]
PartitionsDecodeHat[n_, func_] := Sum[func[j] * p[n - j], {j, 1, n}]
PartitionsEncodeHat2[n_, func_] := Sum[func[j] * p[n - j], {j, 1, n}]
PartitionsDecodeHat2[n_, func_] := Sum[func[j] * phat[n - j], {j, 1, n}]

(**** : Generate the LaTeX table data -- V1: ****)
GetAFuncEntry[n_, afunc_] := 
     With[{afuncInv = DirichletInverse[#, afunc]&}, 
          {DirichletInverse[n, afunc], PartitionsEncode[n, afuncInv], 
           Power[-1, n] * PartitionsEncode2[n, afuncInv]}
	 ]
GetTableRowEntry[n_] := 
     Flatten[Prepend[
	     Table[Map[postProcessAFunc[[afuncIndex]], 
	               GetAFuncEntry[n, afuncs[[afuncIndex]]]
		  ], {afuncIndex, 1, Length[afuncs]}], 
	     {n}], 
             1];
tables = {};

(****)
afuncs = {EulerPhi, (PrimeNu[#] + 1)&};
postProcessAFunc = {#&, #&};
Table[GetTableRowEntry[n], {n, 1, nupper}];
tables = Append[tables, %];

(****)
afuncs = {DivisorSigma[0, #]&, DivisorSigma[1, #]&};
postProcessAFunc = {#&, #&};
Table[GetTableRowEntry[n], {n, 1, nupper}];
tables = Append[tables, %];

(****)
afuncs = {(PrimePi[#] + If[# == 1, 1, 0])&, (PrimeOmega[#] + 1)&};
postProcessAFunc = {#&, #&};
Table[GetTableRowEntry[n], {n, 1, nupper}];
tables = Append[tables, %];

(****)
afuncs = {N[MangoldtLambda[#] + If[# == 1, 1, 0], 12]&, Power[-1, #]&};
postProcessAFunc = {
   #&, #&, 
   With[{preDecimalDigits = If[# < 1, 0, Max[0, Ceiling[Log[10, #]]]]}, 
        If[# < 0, NumberForm[#, {Max[1, preDecimalDigits], 4 - preDecimalDigits}],
                  NumberForm[#, {preDecimalDigits, 5 - preDecimalDigits}]]]&
};
Table[GetTableRowEntry[n], {n, 1, nupper}];
tables = Append[tables, %];

(****)
afuncs = {MoebiusMu, LiouvilleLambda};
postProcessAFunc = {#&, #&};
Table[GetTableRowEntry[n], {n, 1, nupper}];
tables = Append[tables, %];

(****)
Map[TeXForm, tables]

(****** : Generate the LaTeX table data -- V2: ******)
GetAFuncEntry[n_, afunc_] := 
     With[{afuncInv = DirichletInverse[#, afunc]&}, 
          {DirichletInverse[n, afunc], PartitionsEncodeHat[n, afuncInv], 
           Power[-1, n] * PartitionsEncodeHat2[n, afuncInv]}
	 ]
GetTableRowEntry[n_] := 
     Flatten[Prepend[
	     Table[Map[postProcessAFunc[[afuncIndex]], 
	               GetAFuncEntry[n, afuncs[[afuncIndex]]]
		  ], {afuncIndex, 1, Length[afuncs]}], 
	     {n}], 
             1];
tables = {};

(****)
afuncs = {EulerPhi, (PrimeNu[#] + 1)&};
postProcessAFunc = {#&, #&};
Table[GetTableRowEntry[n], {n, 1, nupper}];
tables = Append[tables, %];

(****)
afuncs = {DivisorSigma[0, #]&, DivisorSigma[1, #]&};
postProcessAFunc = {#&, #&};
Table[GetTableRowEntry[n], {n, 1, nupper}];
tables = Append[tables, %];

(****)
afuncs = {(PrimePi[#] + If[# == 1, 1, 0])&, (PrimeOmega[#] + 1)&};
postProcessAFunc = {#&, #&};
Table[GetTableRowEntry[n], {n, 1, nupper}];
tables = Append[tables, %];

(****)
afuncs = {N[MangoldtLambda[#] + If[# == 1, 1, 0], 12]&, Power[-1, #]&};
postProcessAFunc = {
   # &, # &, 
   With[{preDecimalDigits = If[# < 1, 0, Max[0, Ceiling[Log[10, #]]]]}, 
        If[# < 0, NumberForm[#, {Max[1, preDecimalDigits], 4 - preDecimalDigits}],
                  NumberForm[#, {preDecimalDigits, 5 - preDecimalDigits}]]]&
};
Table[GetTableRowEntry[n], {n, 1, nupper}];
tables = Append[tables, %];

(****)
afuncs = {MoebiusMu, LiouvilleLambda};
postProcessAFunc = {#&, #&};
Table[GetTableRowEntry[n], {n, 1, nupper}];
tables = Append[tables, %];

(****)
Map[TeXForm, tables]

\end{lstlisting}

\SubsectionGTThesisFormatted{Source code for the LGF correlation visualizations}
\begin{lstlisting}[caption={Mathematica code for the LGF correlation visualizations},captionpos=t]
(*********************************************************)
(*********************************************************)
(**** : ThesisPackage.m:                              ****)
(**** : Author:  Maxie D. Schmidt (maxieds@gmail.com) ****)
(**** : Created: 2020.02.29                           ****) 
(*********************************************************)
(*********************************************************)

(**** : Clear out any old definitions and stored data on reload: ****) 
ClearAll["ThesisPackage`*"];
Clear[
     DirichletInverse,
     DirichletInverseV2,
     M,
     p1, 
     p2,
     PartitionsEncode,
     PartitionsDecode,
     ImageCorrelateAndProject
];
LocalPartitionsPackageName = "ThesisPackage`";
BeginPackage[LocalPartitionsPackageName, {"Combinatorica`"}]

(**** : Package configuration for the core functionality : ****) 

StartMemoryInUse = Once[MemoryInUse[]];

(**** : Core package function usage and documentation strings: ****) 

DirichletInverse::usage = "DirichletInverse[n, f]: Computes the " <> 
                          "Dirichlet inverse of f at n, f^{-1}(n)"
DirichletInverseV2::usage = "DirichletInverseV2[n, f]: Computes the " <> 
			    "Dirichlet inverse of f at n, f^{-1}(n)"
M::usage = "M[x, n]: Indicator function of whether n divides x."
p1::usage = "p1[n]: Special \"sign smoothing\" partition function, " <> 
            "p_1(n) = [z^n] \prod_{m \geq 1} (1+q^m)"
p2::usage = "p2[n]: Special \"sign smoothing\" partition function, " <> 
	    "p_2(n) = [z^n] \prod_{m \geq 1} (1+q^m)^{-1}"
PartitionsEncode::usage = "PartitionsEncode[n, f]: Sign smoothing " <> 
                          "transformation formed by a discrete " <> 
                          "convolution of f with p_1(n)."
PartitionsDecode::usage = "PartitionsDecode[n, f]: Inverse of the " <> 
                          "sign smoothing transformation formed " <> 
                          "by a discrete convolution " <> 
                          "of f with p_2(n), whose OGF is " <> 
                          "the reciprocal of p_1(n)."
ImageCorrelateAndProject::usage = "The core routine used to visualize " <> 
                                  "so-called canonical, or best " <> 
                                  "possible, factorization theorem " <> 
                                  "expansions for the LGF and " <> 
                                  "more general D-convolution " <> 
                                  "type sum cases."

(**** : The crux of the key new functionality in the package: ****) 

DirichletInverse[n_, ffunc_] := DirichletInverse[n, ffunc] = 
     Module[{},
     If[n <= 0 || !IntegerQ[n] || ffunc[1] === 0, 
          Return[Null];
     ];
     If[n === 1, 
          Return[1 / ffunc[1]];
     ];
     Return[-1 / ffunc[1] * DivisorSum[n, ffunc[#] * 
            DirichletInverse[n / #, ffunc]&, # > 1&]
           ];
];

DirichletInverseV2[n_, ffunc_] := 
     DirichletInverseV2[n, ffunc] = 
     Module[{icVec, postMultMtrx, dinvSols, mult, extraEqns},
     If[n <= 0 || ! IntegerQ[n] || ffunc[1] === 0, 
          Return[Null];
     ];
     icVec = Table[If[x == 1, 1, 0], {x, 1, n}];
     postMultMtrx = Inverse[Table[M[d, x] * ffunc[x / d], 
                            {x, 1, n}, {d, 1, n}]];
     dinvSols = postMultMtrx.icVec;
     Return[dinvSols[[n]]];
];

eps[p_, n_] := eps[p, n] = 
     With[{pexp = Select[FactorInteger[n], #[[1]] == p&]}, 
     If[pexp != {}, pexp[[1]][[2]], 0]
]

HarmonicNumberHn[n_, r_] := HarmonicNumberHn[n, r] = 
     If[n < 0, 0, HarmonicNumber[n, r]]

M[j_, x_] := If[IntegerQ[j] && Divisible[x, j], 1, 0];

omega[n_] := omega[n] = PrimeNu[n]

Omega[n_] := Omega[n] = PrimeOmega[n]

p1[n_] := p1[n] = SeriesCoefficient[Product[(1 + q^m), 
                                    {m, 1, n}], {q, 0, n}]
p2[n_] := p2[n] = SeriesCoefficient[1 / Product[(1 + q^m), 
                                    {m, 1, n}], {q, 0, n}]

PartitionsEncode[n_, func_] := Sum[func[j] * p1[n + 1 - j], {j, 1, n}]
PartitionsDecode[n_, func_] := Sum[func[j] * p2[n + 1 - j], {j, 1, n}]

ImageCorrelateAndProject[initImg_, NUpper_, CgfFunc_, DInvKernelFunc_] :=
     Module[{cFunc, fnkFunc, fnkPointsTable, corrStatBoundRaw, 
             corrStatBound, corrStatBound2},
          cFunc = SeriesCoefficient[CgfFunc[q], {q, 0, #}]&;
          fnkFunc[n_, k_] := funkFunc[n, k] = 
               With[{cFuncVar = Sqrt[Sum[Power[cFunc[m], 2], {m, 1, n}]]},
                    If[1 <= k <= n && n <= NUpper, 
                         If[cFuncVar != 0, 
                            cFunc[k] * DInvKernelFunc[n, k] / n / 
	                    Sqrt[Sum[Power[DInvKernelFunc[n, m], 2], 
	                    {m, 1, n}] / cFuncVar], 0
                         ], 
                         0
                    ]
               ];
          fnkPointsTable = Table[N[Abs[fnkFunc[n, k]]], 
			         {n, 1, NUpper}, {k, 1, NUpper}];
          corrStatBoundRaw = N[Sum[fnkFunc[n, k], 
			       {n, 1, NUpper}, {k, 1, n}]];
          corrStatBound = N[Sum[Abs[fnkFunc[n, k]], 
			    {n, 1, NUpper}, {k, 1, n}]];
          corrStatBound2 = N[Sum[Abs[Sum[fnkFunc[n, k], {k, 1, n}]], 
	                     {n, 1, NUpper}]];
          Return[
               {
                    ImageCorrelate[initImg, fnkPointsTable], 
                    ImageConvolve[initImg, fnkPointsTable],
                    { corrStatBoundRaw, corrStatBound, corrStatBound2 }
               }
          ];
     ];

(**** : Text processing, display and general utility functions : ****) 

GetDatestamp[] := ToString[StringForm[
     "`1`/`2`/`3` : `4`:`5`:`6`", ##]]& @@ Map[DateValue, 
     {"Month", "Day", "Year", "Hour", "Minute", "Second"}]

PrintNotebookNotificationDetailed[textColor_, backgroundColor_, borderColor_] := 
     Function[cellText, 
     If[$Notebooks,
     CellPrint[Cell[cellText, "Print", 
                    FontColor -> textColor, 
                    FontSize -> 14, 
                    CellFrame -> 2.5, 
                    CellFrameColor -> borderColor, 
                    Background -> backgroundColor, 
                    CellFrameMargins -> 14, 
                    ShowCellBracket -> False
                   ]
              ],
     Print[cellText]; 
    ]
];

PrintNotebookNotification[baseColor_:Green] := 
     PrintNotebookNotificationDetailed[
     Darker[baseColor, 0.8], Lighter[baseColor, 0.5], Black];

PrintError[errorText_] := 
     PrintNotebookNotification[Red][
     StringJoin["\[ScriptX]\n", " !!! ERROR : >>> @ ", 
                GetDatestamp[], "\n", errorText, "\n\[ScriptX]"]
     ]

(**** : Package help and information printing routines : ****) 

PackageSingleStringFromList[headerStr_, listComps_, 
     bulletMarker_:"\[RightTriangle]"] := 
     Module[{constructListElementFunc}, 
     constructListElementFunc = (" " <> bulletMarker <> 
	                         " " <> ToString[#] <> "\n")&;
     Return[headerStr <> "\n" <> 
	    StringJoin @@ Map[constructListElementFunc, listComps]
     ];
];

Options[GetHPPackageUsageString] = 
     {BulletPointMarker -> "\[RightTriangle]"};
GetHPPackageUsageString[OptionsPattern[]] := 
Module[{lineFunc, usageDescList}, 
     lineFunc = ToString[ToString[" "] <> 
                ToString[OptionValue[BulletPointMarker]] <> 
                " " <> ToString[#]]&;
     usageDescList = {"Loading the Thesis Functions Package: \n", 
                      lineFunc["SetDirectory['ThesisPackage.m']\n"], 
                      lineFunc["<< ThesisPackage.m\n"], 
                      "Core Functions Provided by the Package: \n"
     };
     Return[StringJoin @@ usageDescList];
];

Options[HPPackageUsage] = 
     {NotebookDisplayColor -> RGBColor[1.0, 0.0, 0.54]};
HPPackageUsage[OptionsPattern[]] := 
     PrintNotebookNotification[OptionValue[NotebookDisplayColor]][
	                       GetHPPackageUsageString[]];

Options[HPPackageHelp] = {NotebookDisplayColor -> Yellow};
HPPackageHelp[OptionsPattern[]] := Module[{helpHeader, helpStr}, 
     helpHeader = "Help for the Package: ";
     helpStr = "Helpful documentation for a particular constant " <> 
               "or package feature can be accessed by " <> 
               "running ?ConstantOrFunctionName " <> 
               "in your working notebook. " <> 
               "If you are unsure where to " <> 
               "start, or just want a detailed list of " <> 
               "what exactly we " <> 
               "have implemented here, you can " <> 
               "evaluate the following in your notebook " <> 
               "with the package loaded " <> 
               "?" <> LocalPartitionsPackageName <> "`* ... ";
    PrintNotebookNotification[OptionValue[NotebookDisplayColor]][
                              helpHeader <> "\n\n" <> helpStr
                             ];
];             

(**** : Perform pretty printing of the package details and  : ****) 
(**** : revsision information if the package is loaded in a : ****)
(**** : notebook                                            : ****)

(** Local package loaded announcement: **)
packageAnnouncement = ToString[LocalPartitionsPackageName] <> 
                      " Package (2020.02.29-v1) Loaded!\n\n";
PrintNotebookNotification[][packageAnnouncement <> 
                            GetHPPackageUsageString[]]

EndPackage[]
\end{lstlisting}

\begin{lstlisting}[caption={Mathematica notebook code for the LGF correlation visualizations},captionpos=t]
<< "ThesisPackage.m"
tuxPenguinImage = Import["BGEnhanced-800px-Tux.png"];

NUpper = 32;
DnkKernelFunc = Function[{n, k}, 
                         If[Divisible[n, k], MoebiusMu[n / k], 0]];
CogfFunc = Function[q, QPochhammer[q, q]];
returnedImages = ImageCorrelateAndProject[
	         tuxPenguinImage, NUpper, CogfFunc, DnkKernelFunc];
NCCScalingFactor -> returnedImages[[3]]
Show[GraphicsGrid[{
	{returnedImages[[1]], returnedImages[[2]]}, 
	{ImageAdjust[returnedImages[[1]]], tuxPenguinImage}, 
	{ImageAdjust[returnedImages[[2]]], tuxPenguinImage}}
     ]
] // Magnify[#, 1.0] &
GraphicsRow[{ImageAdjust[returnedImages[[1]]], returnedImages[[1]]}, 
            Frame -> All] // Image
ImageAdjust[returnedImages[[1]]] // Image

(**** : To generate the images in the introduction,           : ****)
(**** : repeat the above procedure with each of the following : ****
(**** : definitions for `CogfFunc`:                           : ****)
CogfFunc = Function[q, 1 / QPochhammer[q, q]];
CogfFunc = Function[q, QPochhammer[q^2, q^5]];
CogfFunc = Function[q, q * Exp[-q^4]];
CogfFunc = Function[q, q / (Exp[q] - 1)];
CogfFunc = Function[q, 1 / Power[1 - q, 3/2]];
CogfFunc = Function[q, 1 / (1 - q)];
CogfFunc = Function[q, q/(1 - q - q^2)];

(**** : Same thing for the partition product OGFs in the : ****)
(**** : last section of the thesis:                      : ****)
NUpper = 42;
CogfFunc1 = Function[q, QPochhammer[q, q]];
CogfFunc2 = Function[q, QPochhammer[q, q^5] * 
                     QPochhammer[q^4, q^5] * QPochhammer[q^5, q^5]];
CogfFunc3 = Function[q, QPochhammer[q^2, q^5] * QPochhammer[q^3, q^5] * 
                     QPochhammer[q^5, q^5]];
GraphicsRow[{c1Image, c2Image, c3Image}, 
            Spacings -> 20] // Magnify[#, 2]&
\end{lstlisting}

\begin{table}[h!]
\caption{Sign-smoothing transformations: $f_1 \equiv \phi$ and $f_2 \equiv \omega + \mathds{1}$}
\label{table_s12fn_SignSmoothingTF_v1}

\begin{equation*}
\begin{array}{|c|ccc|ccc|} 
 \hline
 n & f_1^{-1}(n) & \widehat{s}_1[f_1^{-1}](n) & \widehat{s}_2[f_1^{-1}](n) & 
     f_2^{-1}(n) & \widehat{s}_1[f_2^{-1}](n) & 
     \widehat{s}_2[f_2^{-1}](n) \\ \hline
 1 & 1 & 1 & -1 & 1 & 1 & -1 \\
 2 & -1 & 0 & -2 & -2 & -1 & -3 \\
 3 & -2 & -2 & 1 & -2 & -3 & 0 \\
 4 & -1 & -2 & 0 & 2 & 0 & 3 \\
 5 & -4 & -7 & 1 & -2 & -4 & 1 \\
 6 & 2 & -6 & 6 & 5 & 0 & 6 \\
 7 & -6 & -13 & 7 & -2 & -1 & 8 \\
 8 & -1 & -20 & 8 & -2 & -8 & 3 \\
 9 & -2 & -23 & 5 & 2 & -2 & 3 \\
 10 & 4 & -31 & 15 & 5 & -1 & 10 \\
 11 & -10 & -47 & 24 & -2 & -4 & 14 \\
 12 & 2 & -60 & 23 & -7 & -12 & 2 \\
 13 & -12 & -82 & 32 & -2 & -11 & 8 \\
 14 & 6 & -107 & 47 & 5 & -12 & 25 \\
 15 & 8 & -113 & 37 & 5 & -9 & 15 \\
 16 & -1 & -159 & 39 & 2 & -7 & 13 \\
 \hline
\end{array}
\end{equation*}
\end{table}

\begin{table}[h!]
\caption{Sign-smoothing transformations: $f_1 \equiv d \equiv \sigma_0$ and $f_2 \equiv \sigma \equiv \sigma_1)$}
\label{table_s12fn_SignSmoothingTF_v2}

\begin{equation*}
\begin{array}{|c|ccc|ccc|} 
 \hline
 n & f_1^{-1}(n) & \widehat{s}_1[f_1^{-1}](n) & \widehat{s}_2[f_1^{-1}](n) & 
     f_2^{-1}(n) & \widehat{s}_1[f_2^{-1}](n) & 
     \widehat{s}_2[f_2^{-1}](n) \\ \hline
 1 & 1 & 1 & -1 & 1 & 1 & -1 \\
 2 & -2 & -1 & -3 & -3 & -2 & -4 \\
 3 & -2 & -3 & 0 & -4 & -6 & 1 \\
 4 & 1 & -1 & 2 & 2 & -3 & 5 \\
 5 & -2 & -5 & 0 & -6 & -12 & 4 \\
 6 & 4 & -2 & 5 & 12 & -3 & 18 \\
 7 & -2 & -4 & 6 & -8 & -11 & 22 \\
 8 & 0 & -9 & 4 & 0 & -23 & 16 \\
 9 & 1 & -6 & 4 & 3 & -12 & 16 \\
 10 & 4 & -7 & 8 & 18 & -11 & 39 \\
 11 & -2 & -10 & 13 & -12 & -23 & 58 \\
 12 & -2 & -16 & 7 & -8 & -40 & 31 \\
 13 & -2 & -17 & 12 & -14 & -45 & 56 \\
 14 & 4 & -21 & 22 & 24 & -48 & 109 \\
 15 & 4 & -18 & 18 & 24 & -20 & 75 \\
 16 & 0 & -23 & 16 & 0 & -42 & 66 \\
 \hline
\end{array}
\end{equation*}
\end{table}

\begin{table}[h!]
\caption{Sign-smoothing transformations: $f_1 \equiv \pi + \varepsilon$ and $f_2 \equiv \Omega + \mathds{1}$}
\label{table_s12fn_SignSmoothingTF_v3}

\begin{equation*}
\begin{array}{|c|ccc|ccc|} 
 \hline
 n & f_1^{-1}(n) & \widehat{s}_1[f_1^{-1}](n) & \widehat{s}_2[f_1^{-1}](n) & 
     f_2^{-1}(n) & \widehat{s}_1[f_2^{-1}](n) & 
     \widehat{s}_2[f_2^{-1}](n) \\ \hline
 1 & 1 & 1 & -1 & 1 & 1 & -1 \\
 2 & -1 & 0 & -2 & -2 & -1 & -3 \\
 3 & -2 & -2 & 1 & -2 & -3 & 0 \\
 4 & -1 & -2 & 0 & 1 & -1 & 2 \\
 5 & -3 & -6 & 0 & -2 & -5 & 0 \\
 6 & 1 & -6 & 4 & 5 & -1 & 6 \\
 7 & -4 & -11 & 4 & -2 & -3 & 7 \\
 8 & -1 & -17 & 5 & 0 & -8 & 4 \\
 9 & 0 & -19 & 1 & 1 & -4 & 5 \\
 10 & 2 & -26 & 7 & 5 & -4 & 10 \\
 11 & -5 & -37 & 13 & -2 & -6 & 15 \\
 12 & 3 & -45 & 13 & -4 & -13 & 6 \\
 13 & -6 & -61 & 18 & -2 & -12 & 12 \\
 14 & 2 & -80 & 22 & 5 & -14 & 26 \\
 15 & 6 & -85 & 19 & 5 & -9 & 19 \\
 16 & 1 & -114 & 21 & 0 & -11 & 16 \\
 \hline
\end{array}
\end{equation*}
\end{table}

\begin{table}[h!]
\caption{Sign-smoothing transformations: $f_1 \equiv \Lambda + \varepsilon$ and $f_2(n) \equiv (-1)^n$}
\label{table_s12fn_SignSmoothingTF_v4}

\begin{equation*}
\begin{array}{|c|ccc|ccc|} 
 \hline
 n & f_1^{-1}(n) & \widehat{s}_1[f_1^{-1}](n) & \widehat{s}_2[f_1^{-1}](n) & 
     f_2^{-1}(n) & \widehat{s}_1[f_2^{-1}](n) & 
     \widehat{s}_2[f_2^{-1}](n) \\ \hline
 1 & 1.000000 & 1.000000 & -1.000000 & -1 & -1 & 1 \\
 2 & -0.693147 & 0.306852 & -1.693147 & -1 & -2 & 0 \\
 3 & -1.098612 & -0.791759 & 0.405465 & 1 & -1 & -2 \\
 4 & -0.212694 & -0.00445 & -0.114081 & -2 & -4 & -2 \\
 5 & -1.609437 & -2.30703 & -0.29640 & 1 & -4 & -3 \\
 6 & 1.523000 & -0.88265 & 2.53790 & 1 & -3 & -1 \\
 7 & -1.945910 & -2.73440 & 2.66168 & 1 & -6 & -3 \\
 8 & -0.065265 & -5.20086 & 2.68285 & -4 & -10 & -9 \\
 9 & 0.108336 & -4.57398 & 1.15160 & 0 & -12 & -6 \\
 10 & 2.23115 & -5.9026 & 4.70078 & 1 & -13 & -3 \\
 11 & -2.397895 & -9.1498 & 8.24036 & 1 & -19 & -9 \\
 12 & -0.06049 & -12.9305 & 5.62768 & 2 & -18 & -10 \\
 13 & -2.564949 & -16.3708 & 8.63533 & 1 & -24 & -9 \\
 14 & 2.69760 & -21.3719 & 14.6211 & 1 & -30 & -11 \\
 15 & 3.53629 & -19.5413 & 11.5039 & -1 & -35 & -9 \\
 16 & -0.02002 & -28.0212 & 10.0733 & -8 & -50 & -24 \\
 \hline
\end{array}
\end{equation*}
\end{table}

\begin{table}[h!]
\caption{Sign-smoothing transformations: $f_1 \equiv \mu$ and $f_2 \equiv \lambda$}
\label{table_s12fn_SignSmoothingTF_v5}

\begin{equation*}
\begin{array}{|c|ccc|ccc|} 
 \hline
 n & f_1^{-1}(n) & \widehat{s}_1[f_1^{-1}](n) & \widehat{s}_2[f_1^{-1}](n) & 
     f_2^{-1}(n) & \widehat{s}_1[f_2^{-1}](n) & 
     \widehat{s}_2[f_2^{-1}](n) \\ \hline
 1 & 1 & 1 & -1 & 1 & 1 & -1 \\
 2 & 1 & 2 & 0 & 1 & 2 & 0 \\
 3 & 1 & 3 & 0 & 1 & 3 & 0 \\
 4 & 1 & 5 & -1 & 0 & 4 & -2 \\
 5 & 1 & 7 & 0 & 1 & 6 & -1 \\
 6 & 1 & 10 & -1 & 1 & 9 & -1 \\
 7 & 1 & 14 & 0 & 1 & 12 & -1 \\
 8 & 1 & 19 & -1 & 0 & 16 & -3 \\
 9 & 1 & 25 & -1 & 0 & 20 & -2 \\
 10 & 1 & 33 & -1 & 1 & 27 & -1 \\
 11 & 1 & 43 & -1 & 1 & 35 & -3 \\
 12 & 1 & 55 & -1 & 0 & 44 & -4 \\
 13 & 1 & 70 & -2 & 1 & 56 & -5 \\
 14 & 1 & 88 & -1 & 1 & 70 & -3 \\
 15 & 1 & 110 & -2 & 1 & 87 & -5 \\
 16 & 1 & 137 & -2 & 0 & 108 & -8 \\
 \hline
\end{array}
\end{equation*}
\end{table}

\begin{table}[h!]
\caption{Sign-non-smoothing transformations: $f_1 \equiv \phi$ and $f_2 \equiv \omega + \mathds{1}$}
\label{table_s12Hatfn_SignSmoothingTF_v1}

\begin{equation*}
\begin{array}{|c|ccc|ccc|} 
 \hline
 n & f_1^{-1}(n) & \widehat{s}_1[f_1^{-1}](n) & \widehat{s}_2[f_1^{-1}](n) & 
     f_2^{-1}(n) & \widehat{s}_1[f_2^{-1}](n) & 
     \widehat{s}_2[f_2^{-1}](n) \\ \hline 
 1 & 1 & 1 & -1 & 1 & 1 & -1 \\
 2 & -1 & -2 & 0 & -2 & -3 & -1 \\
 3 & -2 & -2 & 1 & -2 & -1 & 2 \\
 4 & -1 & 2 & -2 & 2 & 6 & -1 \\
 5 & -4 & -1 & 7 & -2 & -2 & 5 \\
 6 & 2 & 8 & -8 & 5 & 6 & -2 \\
 7 & -6 & -5 & 21 & -2 & -7 & 8 \\
 8 & -1 & 2 & -30 & -2 & -6 & -11 \\
 9 & -2 & 3 & 51 & 2 & 6 & 15 \\
 10 & 4 & 1 & -69 & 5 & 1 & -14 \\
 11 & -10 & -11 & 120 & -2 & -2 & 28 \\
 12 & 2 & -2 & -159 & -7 & -14 & -36 \\
 13 & -12 & -4 & 252 & -2 & 9 & 52 \\
 14 & 6 & 9 & -333 & 5 & 16 & -61 \\
 15 & 8 & 19 & 479 & 5 & 7 & 85 \\
 16 & -1 & -27 & -643 & 2 & -11 & -101 \\
 \hline
\end{array}
\end{equation*}
\end{table}

\begin{table}[h!]
\caption{Sign-non-smoothing transformations: $f_1 \equiv d \equiv \sigma_0$ and $f_2 \equiv \sigma \equiv \sigma_1$}
\label{table_s12Hatfn_SignSmoothingTF_v2}

\begin{equation*}
\begin{array}{|c|ccc|ccc|} 
 \hline
 n & d^{-1}(n) & \widehat{s}_1[d^{-1}](n) & \widehat{s}_2[d^{-1}](n) & 
     \sigma^{-1}(n) & \widehat{s}_1[\sigma^{-1}](n) & 
     \widehat{s}_2[\sigma^{-1}](n) \\ \hline
 1 & 1 & 1 & -1 & 1 & 1 & -1 \\
 2 & -2 & -3 & -1 & -3 & -4 & -2 \\
 3 & -2 & -1 & 2 & -4 & -2 & 5 \\
 4 & 1 & 5 & -2 & 2 & 9 & -5 \\
 5 & -2 & -1 & 6 & -6 & -4 & 16 \\
 6 & 4 & 6 & -5 & 12 & 17 & -10 \\
 7 & -2 & -6 & 12 & -8 & -17 & 32 \\
 8 & 0 & -3 & -16 & 0 & -7 & -38 \\
 9 & 1 & 2 & 24 & 3 & 10 & 60 \\
 10 & 4 & -1 & -28 & 18 & 5 & -59 \\
 11 & -2 & -2 & 47 & -12 & -19 & 116 \\
 12 & -2 & -8 & -59 & -8 & -28 & -135 \\
 13 & -2 & 5 & 86 & -14 & 17 & 216 \\
 14 & 4 & 9 & -106 & 24 & 44 & -237 \\
 15 & 4 & 8 & 146 & 24 & 36 & 337 \\
 16 & 0 & -11 & -182 & 0 & -60 & -402 \\
 \hline
\end{array}
\end{equation*}
\end{table}

\begin{table}[h!]
\caption{Sign-non-smoothing transformations: $f_1 \equiv \pi + \varepsilon$ and $f_2 \equiv \Omega + \mathds{1}$}
\label{table_s12Hatfn_SignSmoothingTF_v3}

\begin{equation*}
\begin{array}{|c|ccc|ccc|} 
 \hline
 n & f_1^{-1}(n) & \widehat{s}_1[f_1^{-1}](n) & \widehat{s}_2[f_1^{-1}](n) & 
     f_2^{-1}(n) & \widehat{s}_1[f_2^{-1}](n) & 
     \widehat{s}_2[f_2^{-1}](n) \\ \hline
 1 & 1 & 1 & -1 & 1 & 1 & -1 \\
 2 & -1 & -2 & 0 & -2 & -3 & -1 \\
 3 & -2 & -2 & 1 & -2 & -1 & 2 \\
 4 & -1 & 2 & -2 & 1 & 5 & -2 \\
 5 & -3 & 0 & 6 & -2 & -1 & 6 \\
 6 & 1 & 6 & -8 & 5 & 7 & -4 \\
 7 & -4 & -3 & 18 & -2 & -7 & 11 \\
 8 & -1 & 1 & -27 & 0 & -4 & -14 \\
 9 & 0 & 3 & 43 & 1 & 2 & 21 \\
 10 & 2 & -2 & -61 & 5 & 0 & -22 \\
 11 & -5 & -7 & 99 & -2 & -2 & 39 \\
 12 & 3 & -1 & -133 & -4 & -11 & -48 \\
 13 & -6 & -5 & 202 & -2 & 8 & 70 \\
 14 & 2 & 2 & -272 & 5 & 12 & -82 \\
 15 & 6 & 13 & 381 & 5 & 9 & 113 \\
 16 & 1 & -12 & -511 & 0 & -13 & -136 \\
 \hline
\end{array}
\end{equation*}
\end{table}

\begin{table}[h!]
\caption{Sign-non-smoothing transformations: $f_1 \equiv \Lambda + \varepsilon$ and $f_2(n) \equiv (-1)^n$}
\label{table_s12Hatfn_SignSmoothingTF_v4}

\begin{equation*}
\begin{array}{|c|ccc|ccc|} 
 \hline
 n & f_1^{-1}(n) & \widehat{s}_1[f_1^{-1}](n) & \widehat{s}_2[f_1^{-1}](n) & 
     f_2^{-1}(n) & \widehat{s}_1[f_2^{-1}](n) & 
     \widehat{s}_2[f_2^{-1}](n) \\ \hline
 1 & 1.000000 & 1.000000 & -1.000000 & -1 & -1 & 1 \\
 2 & -0.693147 & -1.693147 & 0.306852 & -1 & 0 & -2 \\
 3 & -1.098612 & -1.405465 & -0.208240 & 1 & 3 & 2 \\
 4 & -0.212694 & 1.579065 & 0.30239 & -2 & -2 & -6 \\
 5 & -1.609437 & -0.29813 & 1.09879 & 1 & 2 & 7 \\
 6 & 1.523000 & 4.34513 & -0.27339 & 1 & 1 & -11 \\
 7 & -1.945910 & -2.55261 & 3.62496 & 1 & -2 & 15 \\
 8 & -0.065265 & 0.25903 & -5.1718 & -4 & -6 & -27 \\
 9 & 0.108336 & 1.21367 & 9.2977 & 0 & 0 & 34 \\
 10 & 2.23115 & -0.51996 & -11.3478 & 1 & 7 & -51 \\
 11 & -2.397895 & -3.42708 & 23.0725 & 1 & -1 & 67 \\
 12 & -0.06049 & -3.44910 & -29.7571 & 2 & 2 & -94 \\
 13 & -2.564949 & 0.35117 & 48.9493 & 1 & -4 & 123 \\
 14 & 2.69760 & 4.17862 & -62.779 & 1 & 0 & -169 \\
 15 & 3.53629 & 6.66814 & 91.464 & -1 & -7 & 217 \\
 16 & -0.02002 & -9.33079 & -120.591 & -8 & -4 & -300 \\
 \hline
\end{array}
\end{equation*}
\end{table}

\begin{table}[ht!]
\caption{Sign-non-smoothing transformations: $f_1 \equiv \mu$ and $f_2 \equiv \lambda$}
\label{table_s12Hatfn_SignSmoothingTF_v5}

\begin{equation*}
\begin{array}{|c|ccc|ccc|} 
 \hline
 n & f_1^{-1}(n) & \widehat{s}_1[f_1^{-1}](n) & \widehat{s}_2[f_1^{-1}](n) & 
     f_2^{-1}(n) & \widehat{s}_1[f_2^{-1}](n) & 
     \widehat{s}_2[f_2^{-1}](n) \\ \hline 1 & 1 & 1 & -1 & 1 & 1 & -1 \\
 1 & 1 & 1 & -1 & 1 & 1 & -1 \\
 2 & 1 & 0 & 2 & 1 & 0 & 2 \\
 3 & 1 & -1 & -4 & 1 & -1 & -4 \\
 4 & 1 & -1 & 7 & 0 & -2 & 6 \\
 5 & 1 & -1 & -12 & 1 & 0 & -11 \\
 6 & 1 & 0 & 19 & 1 & 1 & 17 \\
 7 & 1 & 0 & -30 & 1 & 0 & -27 \\
 8 & 1 & 1 & 45 & 0 & 0 & 39 \\
 9 & 1 & 1 & -67 & 0 & 0 & -58 \\
 10 & 1 & 1 & 97 & 1 & 3 & 83 \\
 11 & 1 & 1 & -139 & 1 & 1 & -119 \\
 12 & 1 & 1 & 195 & 0 & 0 & 164 \\
 13 & 1 & 0 & -272 & 1 & 0 & -229 \\
 14 & 1 & 0 & 373 & 1 & 0 & 311 \\
 15 & 1 & 0 & -508 & 1 & -1 & -423 \\
 16 & 1 & -1 & 684 & 0 & -2 & 564 \\
 17 & 1 & -1 & -915 & 1 & -1 & -754 \\
 18 & 1 & -1 & 1212 & 0 & -1 & 991 \\
 19 & 1 & -1 & -1597 & 1 & 0 & -1304 \\
 20 & 1 & -1 & 2087 & 0 & 0 & 1693 \\
 21 & 1 & -1 & -2714 & 1 & 0 & -2198 \\
 22 & 1 & -1 & 3506 & 1 & 0 & 2825 \\
 23 & 1 & 0 & -4508 & 1 & -1 & -3626 \\
 24 & 1 & 0 & 5763 & 0 & 1 & 4613 \\
 25 & 1 & 0 & -7338 & 0 & -2 & -5863 \\
 26 & 1 & 0 & 9296 & 1 & 1 & 7399 \\
 27 & 1 & 1 & -11732 & 0 & 1 & -9319 \\
 28 & 1 & 1 & 14742 & 0 & 2 & 11668 \\
 29 & 1 & 1 & -18460 & 1 & 2 & -14584 \\
 30 & 1 & 1 & 23025 & 1 & 0 & 18133 \\
 31 & 1 & 1 & -28629 & 1 & 0 & -22505 \\
 32 & 1 & 1 & 35471 & 0 & -1 & 27803 \\
 \hline
\end{array}
\end{equation*}
\end{table}

\label{page_LastPageOfDocumentMainMatterCount}
\clearpage

\newpage 
\renewcommand{\refname}{\uppercase{References}} 
\addcontentsline{toc}{section}{References}
\bibliographystyle{plain}

\label{TheLastPage}

\end{document}

%% file: PreambleGlossaries.tex
\usepackage{longtable}
\usepackage{arydshln} 
\usepackage[symbols,automake=true]{glossaries}

\newglossarystyle{glossstyleSymbol}{%
 {\begin{longtable}{lp{\glsdescwidth}}}%
 {\end{longtable}}%
 \setlength{\parskip}{3.5pt}

 \renewcommand*{\glsgroupheading}[1]{}%

  \renewcommand{\glossarymark}[1]{}
}

\setglossarystyle{glossstyleSymbol}
\glsaddall[types={symbols}]
\makeglossaries

\newglossaryentry{fCvlg}{
    symbol={\ensuremath{\ast; f \ast g}},
    sort={fg},
    description={The Dirichlet convolution of $f$ and $g$, $$(f \ast g)(n) := \sum\limits_{d|n} f(d) g\left(\frac{n}{d}\right),$$ 
                 for $n \geq 1$. This symbol for the discrete convolution of two arithmetic functions is the only notion of 
                 convolution of functions we employ within the article.},
    type={symbols},
    name={Dirichlet convolution}
    }
\newglossaryentry{coeffExtraction}{
    symbol={\ensuremath{[q^n] F(q)}},
    sort={coeffExtraction},
    description={The coefficient of $q^n$ in the power series expansion of $F(q)$ about zero.},
    type={symbols},
    name={Series coefficient extraction}
    }
\newglossaryentry{MoebiusMuFunc}{
    symbol={\ensuremath{\mu(n)}},
    sort={MoebiusMuFunc},
    description={The M\"obius function.},
    type={symbols},
    name={M\"obius function}
    }
\newglossaryentry{mGCDn}{
    symbol={\ensuremath{\operatorname{gcd}(m, n); (m,n)}},
    sort={mGCDn},
    description={The greatest common divisor of $m$ and $n$. Both notations for the GCD are used 
    interchangably within the article. },
    type={symbols},
    name={Greatest common divisor}
    }
\newglossaryentry{IversonI}{
    symbol={\ensuremath{\Iverson{n=k}}},
    sort={IversonI},
    description={Synonym for $\delta_{n,k}$ which is one if and only if $n = k$, and zero otherwise.},
    type={symbols},
    name={Iverson's convention}
    }
\newglossaryentry{IversonII}{
    symbol={\ensuremath{\Iverson{\mathtt{cond}}}},
    sort={IversonII},
    description={For a boolean-valued \texttt{cond}, $\Iverson{\mathtt{cond}}$ evaluates to one precisely when \texttt{cond} is true, and zero otherwise.},
    type={symbols},
    name={Iverson's convention}
    }
\newglossaryentry{yfn}{
    symbol={\ensuremath{y_f(n)}},
    sort={yfn},
    description={The function $y_f(n)$ denotes the Dirichlet inverse of the function 
    $h(n) := f(n) \phi(n) n^{-2}$ where \glssymbol{phin} is Euler's totient function and 
    $f$ is any invertible arithmetic function such that $f(1) \neq 0$. },
    type={symbols},
    name={A special case Dirichlet inverse function}
    }
\newglossaryentry{tnk}{
    symbol={\ensuremath{t_{n,k}}},
    sort={tnk},
    description={The matrix sequence involved in the generating function expansions of the type I sums defined as 
    \[
    T_{f}(x) = [q^x]\left(\frac{1}{(q; q)_{\infty}} \times \sum_{n \geq 2} \sum_{k=1}^n t_{n,k} f(k) q^n + 
     f(1) q\right)
    \]},
    type={symbols},
    name={Matrix coefficients}
    }
\newglossaryentry{tnkInv}{
    symbol={\ensuremath{t_{n,k}^{(-1)}}},
    sort={tnkInv},
    description={Inverse matrix of the sequence \glssymbol{tnk}.},
    type={symbols},
    name={Inverse matrix coefficients}
    }
\newglossaryentry{chikn}{
    symbol={\ensuremath{\chi_{1,k}(n)}},
    sort={chikn},
    description={The principal Dirichlet character modulo $k$, i.e., the indicator function of the natural numbers which are relatively prime for $n,k \geq 1$, $\chi_{1,k}(n) = \Iverson{(n,k)=1}$.},
    type={symbols},
    name={Principal Dirichlet character modulo $k$}
    }
\newglossaryentry{dn}{
    symbol={\ensuremath{d(n)}},
    sort={dn},
    description={The ordinary divisor function, $d(n) := \sum_{d|n} 1$.},
    type={symbols},
    name={Divisor function}
    }
\newglossaryentry{SigmaSumOfDivisorsFunc}{
    symbol={\ensuremath{\sigma(n)}},
    sort={dn},
    description={The ordinary sum-of-divisors function, $\sigma(n) := \sum_{d|n} d$.},
    type={symbols},
    name={Divisor function}
    }
\newglossaryentry{epsilonN}{
    symbol={\ensuremath{\varepsilon(n)}},
    sort={epsilonN},
    description={The multiplicative identity with respect to Dirichlet convolution, $\varepsilon(n) = \delta_{n,1}$.},
    type={symbols},
    name={Dirichlet multiplicative identity}
    }
\newglossaryentry{munk}{
    symbol={\ensuremath{\mu_{n,k}}},
    sort={munk},
    description={The corresponding invertible sequence is an analog to the role of the M\"obius function in M\"obius inversion. In this case these inversion coefficients are defined such that 
    \[
    g(n) = \sum_{\substack{d=1 \\ (d, n)=1}}^n 
     f(d) \quad\iff\quad f(n) = \sum_{d=1}^n g(d+1) \mu_{n,d}.
    \]
    },
    type={symbols},
    name={Matrix coefficients}
    }
\newglossaryentry{munkInv}{
    symbol={\ensuremath{\mu_{n,k}^{(-1)}}},
    sort={munkInv},
    description={Inverse matrix sequence of \glssymbol{munk}.},
    type={symbols},
    name={Inverse matrix coefficients}
    }
\newglossaryentry{unk}{
    symbol={\ensuremath{u_{n,k}(f, w)}},
    sort={unk},
    description={The matrix sequence defined in the expansion of the generating functions for the type II sums as 
    follows for $w \in \mathbb{C} \setminus \{0\}$: 
    \[
    g(x) = [q^x]\left(\sum_{n \geq 2} \sum_{k=1}^n u_{n,k}(f, w) \left( 
     \sum_{m=1}^k L_{f,g,m}(k) w^m\right) \frac{q^n}{(q; q)_{\infty}}\right). 
    \]},
    type={symbols},
    name={Matrix coefficients}
    }
\newglossaryentry{unkInv}{
    symbol={\ensuremath{u_{n,k}^{(-1)}(f, w)}},
    sort={unkInv},description={Inverse matrix terms of the sequence \glssymbol{unk}.},
    type={symbols},
    name={Inverse matrix coefficients}
    }
\newglossaryentry{DTFT}{
    symbol={\ensuremath{\operatorname{DTFT}[f](k)}},
    sort={DTFT},
    description={The discrete time Fourier transform (DTFT) of $f$ at $k$, also denoted by $F[k]$.},
    type={symbols},
    name={DTFT}
    }
\newglossaryentry{DFT}{
    symbol={\ensuremath{\operatorname{DFT}[f](k)}},
    sort={DTFT},
    description={The discrete Fourier transform (DFT) of $f$ at $k$.}, 
    type={symbols},
    name={DTFT}
    }
\newglossaryentry{phin}{
    symbol={\ensuremath{\phi(n)}},
    sort={phin},
    description={Euler's classical totient function, $\phi(n) := \sum\limits_{\substack{1 \leq d \leq n \\ (d,n) = 1}} 1$.},
    type={symbols},
    name={Euler totient function}
    }
\newglossaryentry{phikn}{
    symbol={\ensuremath{\phi_k(n)}},
    sort={phikn},
    description={Generalized totient function, $\phi_k(n) := \sum\limits_{\substack{1 \leq d \leq n \\ (d,n) = 1}} d^k$.},
    type={symbols},
    name={Generalized totient function}
    }
\newglossaryentry{Tfx}{
    symbol={\ensuremath{T_f(x)}},
    sort={Tfx},
    description={The type I sum over an arithmetic function $f$, $$T_f(n) := \sum\limits_{\substack{d \leq x \\ (d,x)=1}} f(d).$$},
    type={symbols},
    name={Type I sum}
    }
\newglossaryentry{Lfgkx}{
    symbol={\ensuremath{L_{f,g,k}(x)}},
    sort={Lfgkx},
    description={The type II Anderson-Apostol sum over the arithmetic functions $f,g$, 
                 $L_{f,g,k}(x) := \sum\limits_{d|(k,x)} f(d) g\left(\frac{x}{d}\right)$.},
    type={symbols},
    name={Type II sum}
    }
\newglossaryentry{pn}{
    symbol={\ensuremath{p(n)}},
    sort={pn},
    description={The partition function generated by $p(n) = [q^n] \prod\limits_{n \geq 1} (1-q^n)^{-1}$.},
    type={symbols},
    name={Partition function $p$}
    }
\newglossaryentry{Mx}{
    symbol={\ensuremath{M(x)}},
    sort={Mx},
    description={The Mertens function which is the summatory function over $\mu(n)$ denoted 
                 by the partial sums $M(x) := \sum\limits_{n \leq x} \mu(n)$.},
    type={symbols},
    name={Mertens function}
    }
\newglossaryentry{Phinz}{
    symbol={\ensuremath{\Phi_n(z)}},
    sort={Phinz},
    description={The $n^{th}$ cyclotomic polynomial in $z$ defined by 
    $$\Phi_n(z) := \prod\limits_{\substack{1 \leq k \leq n \\ (k,n)=1}} (z-e^{\frac{2\pi\imath k}{n}}).$$},
    type={symbols},
    name={Cyclotomic polynomial}
    }
\newglossaryentry{cqn}{
    symbol={\ensuremath{c_q(n)}},
    sort={cqn},
    description={Ramanujan's sum, $c_q(n) := \sum\limits_{d|(q,n)} d \mu\left(\frac{q}{d}\right)$.},
    type={symbols},
    name={Ramanujan's sum}
    }
\newglossaryentry{SigmaSOD}{
    symbol={\ensuremath{\sigma_{\alpha}(n)}},
    sort={SigmaSOD},
    description={The generalized sum-of-divisors function, $\sigma_{\alpha}(n) := \sum\limits_{d|n} d^{\alpha}$, for any $n \geq 1$ and $\alpha \in \mathbb{C}$.},
    type={symbols},
    name={Generalized sum-of-divisors function}
    }
\newglossaryentry{Zetas}{
    symbol={\ensuremath{\zeta(s)}},
    sort={Zetas},
    description={The Riemann zeta function, defined by $\zeta(s) := \sum_{n \geq 1} n^{-s}$ when $\Re(s) > 1$, and by analytic continuation to the entire complex plane with the exception of a simple pole at $s = 1$.},
    type={symbols},
    name={Riemann zeta function}
    }
\newglossaryentry{Expex}{
    symbol={\ensuremath{e(x)}},
    sort={Expex},
    description={The complex exponential function, $e(x) := \exp(2\pi\imath x)$.},
    type={symbols},
    name={Complex exponential function shorthand}
    }
\newglossaryentry{snk}{
    symbol={\ensuremath{s_{n,k}}},
    sort={snk},
    description={Matrix coefficients in Lambert series type factorizations. These coefficients are defined precisely as the coefficients of the generating function $[q^n] (q; q)_{\infty} q^k (1-q^k)^{-1}$ for $k \geq 1$ where 
    \glssymbol{qPochInfty} is the infinite $q$-Pochhammer symbol.},
    type={symbols},
    name={Lambert series factorization matrix coefficients}
    }
\newglossaryentry{fhatn}{
    symbol={\ensuremath{\hat{f}(n)}},
    sort={fhatn},
    description={A shorthand notation for scaled arithmetic function terms $\hat{f}(n) := w^n (w^n-1)^{-1} f(n)$ for some non-zero indeterminate $w$.},
    type={symbols},
    name={Auxiliary scaled functions}
    }
\newglossaryentry{qPochInftyV01}{
    symbol={$(a_1, \ldots, a_r; q)_n$},
    sort={qPochInftyV01},
    description={We use the shorthand that 
    $$(a_1, \ldots, a_r; q)_n = \prod_{i=1}^{r} (a_i; q)_n.$$},
    type={symbols},
    name={$q$-Pochhammer symbol}
    }
\newglossaryentry{qPochInftyV0}{
    symbol={\ensuremath{(a; q)_{n}}, $(q)_n$},
    sort={qPochInfty},
    description={The $q$-Pochhammer symbol defined as the product $(a; q)_{\infty} := \prod\limits_{n \geq 1} (1-aq^{n-1})$. We adopt the notation that $(q)_n \equiv (q; q)_n$ and that $(a; q)_{\infty}$ denotes the limiting case for $|q| < 1$ as $n \rightarrow \infty$.},
    type={symbols},
    name={Infinite $q$-Pochhammer symbol}
    }
\newglossaryentry{fpmn}{
    symbol={\ensuremath{f_{\pm}(n)}},
    sort={fpmn},
    description={For any arithmetic function $f$, we define 
    $f_{\pm}(n) = f(n) \Iverson{n > 1} - f(1) \Iverson{n = 1}$, i.e., 
    the function that has identical values as $f$ for all $n \geq 2$, and 
    whose initial value is $f_{\pm}(1) := -f(1)$ when $n = 1$.},
    type={symbols},
    name={Case-wise modified functions}
    }
\newglossaryentry{uhatnkfw}{
    symbol={\ensuremath{\hat{u}_{n,k}(f, w)}},
    sort={uhatnkfw},
    description={Matrix coefficients defined in terms of an indeterminate parameter $w$ as $\hat{u}_{n,k}(f, w) := (w^k-1) u_{n,k}(f, w)$.},
    type={symbols},
    name={Matrix coefficients}
    }
\newglossaryentry{dsjn}{
    symbol={\ensuremath{\operatorname{ds}_j(f; n)}},
    sort={dsjn},
    description={Summands in the formula for the Dirichlet inverse of an arithmetic function. The precise definition of this function is given by 
    \[
    \ds_j(f; n) = \begin{cases} 
     (-1)^{\delta_{n,1}} \widehat{f}(n), & \text{ if $j = 1$; } \\ 
     \sum\limits_{\substack{d|n \\ d>1}} \widehat{f}(d) \ds_{j-1}\left(f; \frac{n}{d}\right), & 
     \text{ if $j \geq 2$, } 
     \end{cases} 
    \]
    where the fixed function $\hat{f}$ is defined by glossary symbol 
    \glssymbol{fhatn}.},
    type={symbols},
    name={Inverse function component terms}
    }
\newglossaryentry{Dfn}{
    symbol={\ensuremath{D_f(n)}},
    sort={Dn},
    description={Function related to the Dirichlet inverse of a function $f$. 
    More precisely, this function is defined by the sum 
    $$D_f(n) := \sum\limits_{j=1}^n \frac{\operatorname{ds}_{2j}(f; n)}{\hat{f}(1)^{2j+1}},$$ where this definition involves the glossary symbols 
    \glssymbol{dsjn} and \glssymbol{fhatn}. }, 
    type={symbols},
    name={Inverse function sum function}
    }
\newglossaryentry{fInvn}{
    symbol={\ensuremath{f^{-1}(n)}},
    sort={fInvn},
    description={The Dirichlet inverse of $f$ with respect to convolution defined recursively by 
                 $$f^{-1}(n) = -\frac{1}{f(1)} \sum\limits_{\substack{d|n \\ d>1}} f(d) f^{-1}\left(\frac{n}{d}\right),$$ provided that $f(1) \neq 0$.},
    type={symbols},
    name={Dirichlet inverse of $f$}
    }
\newglossaryentry{qPochInfty}{
    symbol={\ensuremath{(q; q)_{\infty}}},
    sort={qPochInfty},
    description={The infinite $q$-Pochhammer symbol defined as the product $(q; q)_{\infty} := \prod\limits_{n \geq 1} (1-q^n)$ for $|q| < 1$.},
    type={symbols},
    name={Infinite $q$-Pochhammer symbol}
    }
\newglossaryentry{Fk}{
    symbol={\ensuremath{F[k]}},
    sort={Fk},
    description={Discrete Fourier transform coefficients.},
    type={symbols},
    name={DFT coefficients}
    }
\newglossaryentry{skfgn}{
    symbol={\ensuremath{s_k(f, g; n)}},
    sort={skfgn},
    description={Shorthand for the periodic (modulo $k$) divisor sums expanded by the functions listed in 
    \glssymbol{akfgn} of this glossary. The precise expansion and corresponding finite Fourier series expansion of this 
    function is given by $$s_k(f, g; n) = \sum\limits_{d|(n, k)} f(d) g\left(\frac{k}{d}\right) = 
    \sum\limits_{m=1}^{k} a_k(f, g; m) e^{\frac{2\pi\imath mn}{k}}.$$},
    type={symbols},
    name={Shorthand for periodic divisor sums}
    }
\newglossaryentry{akfgn}{
    symbol={\ensuremath{a_k(f, g; n)}},
    sort={akfgn},
    description={Discrete Fourier coefficients of the periodic divisor sums $s_k(f, g; n)$ defined as symbol \glssymbol{skfgn} in this glossary. The precise definition 
    of these sums is given by $$a_k(f, g; n) = \sum\limits_{d|(k, n)} g(d) f\left(\frac{n}{d}\right) \frac{d}{k}.$$},
    type={symbols},
    name={DFT coefficients}
    }
\newglossaryentry{fCvlDashVI}{
    symbol={\ensuremath{f \ast C_{-}(m)}},
    sort={fCvlDashVI},
    description={This notation indicates that the index over which we perform the Dirchlet convolution is given by the dash parameter, 
    $$(f \ast C_{-}(m))(n) := \sum_{d|n} f(d) C_{\frac{n}{d}}(m).$$},
    type={symbols},
    name={Definition of convolution with a fixed parameter}
    }
\newglossaryentry{fCvlDashVII}{
    symbol={\ensuremath{f \ast C_k(-)}},
    sort={fCvlDashVII},
    description={This notation indicates that the index over which we perform the Dirchlet convolution is given by the dash parameter, 
    $$(f \ast C_k(-))(n) := \sum_{d|n} f(d) C_k\left(\frac{n}{d}\right).$$},
    type={symbols},
    name={Definition of convolution with a fixed parameter}
    }
\newglossaryentry{akl}{
    symbol={\ensuremath{a_{k,\ell}}},
    sort={akl},
    description={Sequence of coefficients that are defined explicitly in the discrete Fourier series expansion of the type II sums \glssymbol{Lfgkx}. These coefficients are implicitly defined by Definition \ref{def_NotationAndSpecialExpSums} by the sums 
    $L_{f,g,k}(n) = \sum\limits_{\ell=0}^{k-1} a_{k,\ell} e\left(\frac{\ell n}{k}\right)$, where \glssymbol{Expex} is the shorthand for the complex exponential terms in the exponential sums we define in the article.},
    type={symbols},
    name={Local coefficient definition}
    }
\newglossaryentry{Idkn}{
    symbol={\ensuremath{\operatorname{Id}_k(n)}},
    sort={Idkn},
    description={The power-scaled identity function, $\operatorname{Id}_k(n) := n^k$ for $n \geq 1$.},
    type={symbols},
    name={Dirchlet identity function for powers}
    }
\newglossaryentry{Floorx}{
    symbol={\ensuremath{\lfloor x \rfloor}},
    sort={Floorx},
    description={The floor function $\lfloor x \rfloor := x - \{x\}$ where $0 \leq \{x\} < 1$ denotes the fractional part of $x \in \mathbb{R}$.},
    type={symbols},
    name={Floor function}
    }
\newglossaryentry{Ceilingx}{
    symbol={\ensuremath{\lceil x \rceil}},
    sort={Ceilingx},
    description={The ceiling function $\lceil x \rceil := x + 1 - \{x\}$ where $0 \leq \{x\} < 1$ denotes the fractional part of $x \in \mathbb{R}$.},
    type={symbols},
    name={Ceiling function}
    }
\newglossaryentry{fMultipleConvolution}{
    symbol={\ensuremath{f_{\ast_j}}},
    sort={fMultipleConvolution},
    description={Sequence of nested $j$-convolutions of an arithmetic function $f$ with itself for integers $j \geq 1$. We define $f_{\ast_0}(n) = \delta_{n,1}$, the multiplicative identity with respect to Dirichlet convolution.},
    type={symbols},
    name={Nested $j$-convolutions of $f$ with itself}
    }

\newglossaryentry{Gj}{
    symbol={\ensuremath{G_j}},
    sort={Gj},
    description={Denotes the interleaved (or generalized) sequence of pentagonal numbers defined explictly by the formula $G_j := \frac{1}{2} \ceiling{\frac{j}{2}} \ceiling{\frac{3j+1}{2}}$. The sequence begins as 
    $\{G_j\}_{j \geq 0} = \{0, 1, 2, 5, 7, 12, 15, 22, 26, 35, 40, 51, \ldots\}$.},
    type={symbols},
    name={Interleaved pentagonal numbers}
    }
\newglossaryentry{Ckn}{
    symbol={\ensuremath{C_k(n)}},
    sort={Ckn},
    description={Sequence of nested $k$-convolutions of an arithmetic function $f$ with itself. The precise definition of this sequence is given by 
    \[
     C_k(n) = \begin{cases} 
     \widehat{f}(n) - \widehat{f}(1)\varepsilon(n), & \text{ if $k = 1$; } \\ 
     \sum\limits_{d|n} \left(\widehat{f}(d) - \widehat{f}(1) \varepsilon(d)\right) 
     C_{k-1}\left(\frac{n}{d}\right), & \text{ if $k \geq 2$, } 
     \end{cases} 
     \] 
     where the symbol $\hat{f}(n)$ is defined in glossary entry \glssymbol{fhatn}.},
    type={symbols},
    name={Nested $k$-convolutions of $f$ with itself}
    }    
\newglossaryentry{IndicatorFunctions}{
    symbol={$\OneFunc{\mathbb{S}}$, $\chi_{\mathtt{cond}(x)}$},
    sort={IndCharFuncs},
    description={We use the notation $\mathds{1},\chi: \mathbb{N} \rightarrow \{0,1\}$ to denote indicator, or characteristic functions. In paticular, $\OneFunc{\mathbb{S}}(n) = 1$ if and only if $n \in \mathbb{S}$, and $\chi_{\mathtt{cond}}(n) = 1$ if and only if $n$ satisfies the condition \texttt{cond}.},
    type={symbols},
    name={Indicator function, characteristic function}
    }
\newglossaryentry{StirlingNumbers}{
    symbol={$\gkpSI{n}{k}$, $\gkpSII{n}{k}$},
    sort={StirlingNumbers},
    description={The Stirling numbers of the first and second kinds, respectively. Alternate notation for these triangles is given by $s(n, k) = (-1)^{n-k} \gkpSI{n}{k}$ and $S(n, k) = \gkpSII{n}{k}$.},
    type={symbols},
    name={Stirling numbers}
    }
\newglossaryentry{LambdaFunc}{
    symbol={$\Lambda(n)$},
    sort={LambdaFunc},
    description={The von Mangoldt lambda function $\Lambda(n) = \sum_{d|n} \log(d) \mu\left(\frac{n}{d}\right)$.},
    type={symbols},
    name={von Mangoldt lambda function}
    }
\newglossaryentry{JordanTotientFunc}{
    symbol={$J_t(n)$},
    sort={JordanTotientFunc},
    description={The Jordan totient function $J_t(n) = n^k \times \prod_{p|n} (1-p^{-t})$ satisfies 
    $\sum_{d|n} J_t(d) = n^t$.},
    type={symbols},
    name={Jordan totient function}
    }
\newglossaryentry{rkn}{
    symbol={$r_k(n)$},
    sort={rkn},
    description={The sum of $k$ squares function denotes the number of integer solutions to $n = x_1^2+\cdots+x_k^2$. A generating function is given by $r_k(n) = [q^n] \vartheta_3(q)^k$.},
    type={symbols},
    name={Sum of squares function}
    }
\newglossaryentry{PiPrimeCountingFunc}{
    symbol={$\pi(x)$},
    sort={PrimeCountingFunc},
    description={The prime counting function denotes the number of primes $p \leq x$, i.e., $\pi(x) = \sum_{p \leq x} 1$.},
    type={symbols},
    name={Prime counting function}
    }
\newglossaryentry{OmegaFuncs}{
    symbol={$\omega(n)$, $\Omega(n)$},
    sort={OmegaFuncs},
    description={If $n = p_1^{\alpha_1} \times \cdots \times p_r^{\alpha_r}$ 
    is the prime factorization of $n$ into distinct prime powers, 
    then $\omega(n) = r$ and $\Omega(n) = \alpha_1 + \cdots + \alpha_r$.},
    type={symbols},
    name={Prime omega functions}
    }
\newglossaryentry{JacobiThetaFuncs}{
    symbol={$\vartheta_i(z, q)$, $\vartheta_i(q)$},
    sort={JacobiThetaFuncs},
    description={For $i = 1,2,3,4$, these are the classical Jacobi theta functions where $\vartheta_i(q) \equiv \vartheta_i(0, q)$.},
    type={symbols},
    name={Jacobi theta functions}
    }

\glsaddall[types={symbols}]